\pdfoutput=1
\documentclass[10pt, oneside]{amsart}

\usepackage[colorlinks=true, urlcolor=black, citecolor=black, linkcolor=black, hyperfootnotes=true]{hyperref}
\usepackage{amssymb}
\usepackage{aliascnt} 
\usepackage{mathrsfs}
\usepackage{mathscinet} 

\usepackage{graphicx} 
\usepackage{tikz} 
\usetikzlibrary{arrows.meta} 
\usetikzlibrary{graphs} 
\usetikzlibrary{decorations.pathreplacing} 

\usepackage{enumitem} 
\setlist[enumerate, 1]{label=\upshape (\arabic*)}
\setlist{leftmargin=2.5em}

\numberwithin{equation}{section}

\newtheorem{thm}{Theorem}[section]

\newaliascnt{prp}{thm}
\newtheorem{prp}[prp]{Proposition}
\aliascntresetthe{prp}

\newaliascnt{lemma}{thm}
\newtheorem{lemma}[lemma]{Lemma}
\aliascntresetthe{lemma}

\newaliascnt{claim}{thm}

\aliascntresetthe{claim}

\newaliascnt{cor}{thm}
\newtheorem{cor}[cor]{Corollary}
\aliascntresetthe{cor}

\theoremstyle{definition}

\newaliascnt{dfn}{thm}
\newtheorem{dfn}[dfn]{Definition}
\aliascntresetthe{dfn}

\newaliascnt{xpl}{thm}
\newtheorem{xpl}[xpl]{Example}
\aliascntresetthe{xpl}

\newaliascnt{rmk}{thm}
\newtheorem{rmk}[rmk]{Remark}
\aliascntresetthe{rmk}

\newaliascnt{obs}{thm}
\newtheorem{obs}[obs]{Observation}
\aliascntresetthe{obs}

\newcommand{\sqin}{\mathrel{\vphantom{\sqsubset}\text{\mathsurround=0pt\ooalign{$\sqsubset$\cr$-$\cr}}}}
\newcommand{\sqni}{\mathrel{\vphantom{\sqsupset}\text{\mathsurround=0pt\ooalign{$\sqsupset$\cr$-$\cr}}}}
\newcommand{\sqiin}{\mathrel{\vphantom{\sqsupset}\text{\mathsurround=0pt\ooalign{$\sqsubset$\cr$=$\cr}}}}
\newcommand{\sqnii}{\mathrel{\vphantom{\sqsupset}\text{\mathsurround=0pt\ooalign{$\sqsupset$\cr$=$\cr}}}}
\makeatletter
\providecommand*{\Dashv}{%
 \mathrel{%
 \mathpalette\@Dashv\vDash
 }%
}
\newcommand*{\@Dashv}[2]{%
 \reflectbox{$\m@th#1#2$}%
}
\makeatother

\newcommand*{\op}{\mathrm{op}}
\newcommand*{\maps}{\colon}
\newcommand*{\PP}{\mathbb{P}}
\newcommand*{\QQ}{\mathbb{Q}}
\newcommand*{\RR}{\mathbb{R}}
\newcommand*{\im}[1]{[#1]}

\newcommand*{\Spec}[1]{\mathsf{S}#1}

\newcommand*{\diam}{\operatorname{diam}}

\newcommand*{\Fraisse}{Fra\"{i}ss\'e }

\definecolor{myblue}{rgb}{0,0.2,0.7}
\definecolor{mygreen}{rgb}{0,0.5,0}

\author{Adam Barto\v{s}}
\thanks{Adam Barto\v{s} was supported by the GA\v{C}R project EXPRO 20-31529X and RVO: 67985840.}
\address{Institute of Mathematics of the Czech Academy of Sciences, \v{Z}itn\'a 25, Prague}
\email{bartos@math.cas.cz}
\urladdr{https://math.cas.cz/~bartos/}

\author{Tristan Bice}
\thanks{Tristan Bice was supported by the GA\v{C}R project 22-07833K and RVO: 67985840.}
\address{Institute of Mathematics of the Czech Academy of Sciences, \v{Z}itn\'a 25, Prague}
\email{bice@math.cas.cz}

\author{Alessandro Vignati}
\thanks{Alessandro Vignati was partially supported by an Emergence en recherche grant from Universit\'e Paris Cit\'e}
\address{
Institut de Math\'ematiques de Jussieu - Paris Rive Gauche (IMJ-PRG)\\
Universit\'e Paris Cit\'e, Institut Universitaire de France\\
B\^atiment Sophie Germain\\
8 Place Aur\'elie Nemours \\ 75013 Paris, France}
\email{ale.vignati@gmail.com}
\urladdr{http://www.automorph.net/avignati}


\title[Generic Compacta from Graphs]{Generic Compacta from Relations between Finite Graphs: Theory building and examples}

\begin{document}

\begin{abstract}
In recent work, the authors developed a simple method of constructing topological spaces from certain partially ordered sets -- those coming from
sequences of relations between finite sets.
This method associates a given poset with its spectrum, which is a compact $\mathsf T_1$ topological space.

In this paper, we focus on the case where such finite sets have a graph structure and the relations belong to a given graph category. We relate topological properties of the spectrum to combinatorial properties of the graph categories involved. We then utilise this to exhibit elementary combinatorial constructions of well-known continua as \Fraisse limits of finite graphs in categories with relational morphisms.
\end{abstract}

\maketitle

\setcounter{tocdepth}{2}
\tableofcontents

\section{Introduction}

Almost since the birth of topology in the early twentieth century, people have been interested in constructing compact topological spaces from finite combinatorial structures. For example, Alexandroff \cite{Alexandroff1928} showed how to construct Hausdorff compacta from sequences of finite abstract simplicial complexes, which inspired later constructions of Flachsmeyer \cite{Flachsmeyer1961} and Kopperman et al. \cite{KoppermanTkachukWilson2003} from finite posets. More recently, we have constructions from functions between finite graphs due to Irwin and Solecki \cite{IrwinSolecki2006} and D\polhk{e}bski and Tymchatyn \cite{DebskiTymchatyn2018}. The Irwin--Solecki construction is aimed at spaces like the pseudoarc which arise naturally as projective \Fraisse limits and has since been applied to other well-known continua like the Lelek fan \cite{BartosovaKwiatkowska2015}, the Menger curve \cite{PanagiotopoulosSolecki2022}, fences \cite{BassoCamerlo.Fence}, Wa\.{z}ewski dendrites \cite{charatonik2024projectivefraisselimitstrees,codenotti2022projective}, and a previously unknown indecomposable one-dimensional hereditarily unicoherent continuum \cite{charatonik2022projective}.
One common feature of these constructions is that the space is not obtained directly from the finite structures, but rather as a quotient of an intermediate `pre-space'. This ends up putting some distance between the combinatorics and topology. Also, unlike algebraic quotients, topological quotients are generally not so well-behaved and are usually avoided if at all possible.

In a somewhat different direction, constructions of spaces like the pseudoarc and pseudosolenoids have recently been obtained in \cite{BartosKubis2022} as certain approximate \Fraisse limits. This has the advantage of not going through any kind of pre-space, while the general setting of metric enriched categories can also easily handle richer structures like the Gurari\u{\i} space and other objects naturally appearing in functional analysis (see \cite{Kubis2012}). However, here the structures are already more analytic and less combinatorial from the outset. There is also inevitably some technical machinery that has to be developed in the approximate setting.

In the present paper we combine the best parts of both approaches, while avoiding the need to deal with quotient spaces or metric-enriched categories. Specifically, we construct compacta directly (i.e. without a pre-space) from finite graphs but this time with more general morphisms between them -- relations instead of functions. These relations could be viewed as approximating functions, but as the structures involved are finite, these approximations can be handled in a purely combinatorial way.
In this approach, the relationship between the graphs and the resulting spaces is also more transparent. Specifically, the graphs get represented as a coinitial (w.r.t. refinement) family of open covers of the space -- each vertex corresponds to an open set of the cover while the edge relation indicates precisely when the sets overlap. The relations between the graphs correspond exactly to the inclusion relation between the associated open covers.

The starting point of this paper is the construction of \cite{BartosBiceV.Compacta}, which is briefly outlined for convenience in \autoref{S.preliminaries}, and specifically in \autoref{ss:spectrum}. In short, in \cite{BartosBiceV.Compacta} we associated to a poset $\PP$ a topological space $\Spec{\PP}$, the \emph{spectrum} of $\PP$, by topologising the space of certain subsets of $\PP$ called \emph{selectors}. The spectrum is, as long as the poset is suitable (an $\omega$-poset), a compact $\mathsf T_1$ second-countable space. We then described how combinatorial properties of the poset $\PP$ relate to topological properties of the space $\Spec{\PP}$. Moreover, we studied closed subspaces of the spectrum and, through the development of \emph{refiners} (relations between posets), we coded maps between topological spaces, ultimately obtaining a duality between appropriate categories.


In this article, we focus on the case where the poset arises from a sequence of graphs together with relational morphisms between them. A relation $\sqsupset$ between two graphs $G$ and $H$ generalises a function from $G$ to $H$, viewed as a subset of $H\times G$. Fix a sequence of finite graphs $G_n$ and relations  ${\sqsupset_{n+1}^n}\subseteq G_n\times G_{n+1}$. For $m\leq n$, set ${\sqsupset_n^m}={\sqsupset_{m+1}^m\circ\cdots\circ\sqsupset_n^{n-1}}$, where $\sqsupset_n^n$ is the identity. These graphs and relations, under some extremely mild and reasonable well-behavedness conditions, give rise to a graded $\omega$-poset $\mathbb{P}=\bigsqcup_nG_n$ ordered by $\bigcup_{m\leq n}\sqsupset_n^m$.  Then the machinery of \cite{BartosBiceV.Compacta} applies and we obtain a topological space $\mathsf{S}\mathbb{P}$, the spectrum of $\mathbb{P}$, constructed concretely from the given sequence of graphs and relations.

Two questions interest us: 
\begin{enumerate}
\item\label{q1}If the graphs and relations to belong to a given category of graphs $\mathbf C$, what can we say about the class of spaces obtained this way?
\item\label{q2} Say $\mathbf C$ is rich enough to have a Fra\"iss\'e sequence (i.e. $\mathbf C$ is directed and  has amalgamation). Can we then refer to the \emph{generic} compact space one gets, and if so, can we concretely derive such a space and its properties from the structure of $\mathbf C$?
\end{enumerate}

Most of our results are designed to answer the above questions. Let us streamline them, while summarising the structure of the paper.

In \autoref{S.preliminaries}, we recall the needed background from \cite{BartosBiceV.Compacta} on how to construct the spectrum of a given graded $\omega$-poset. 

In \autoref{S.GraphCategory}, we introduce notable  categories and ideals of morphisms between graphs, and see how these interact with each other. With a minimal requirement represented by edge-preservation (i.e. a relation from a graph $G$ to a graph $H$ should send adjacent vertices to adjacent vertices), there are several additional conditions one can impose to eventually obtain interesting information regarding the associated poset (or rather its spectrum). Particular attention is paid to star-refining morphisms (which ultimately ensure Hausdorffness of the spectrum), and monotone morphisms, especially in the setting of trees.

The main part of the theory is developed in \autoref{s.InducedPoset}. We start by providing several answers to Question~\ref{q1}.  After introducing the key notion of the \emph{induced poset} defined from a sequence of graphs (\autoref{dfn:inducedposet}), we start analysing subsequences and restrictions. This is done in \autoref{ss.subsequences}, after comparing our approach with previous work of \cite{IrwinSolecki2006} and \cite{DebskiTymchatyn2018} in \autoref{ss.threads}.

In \autoref{ss.faithful} we start heavily using the graph structure. We compare order-theoretic properties of the induced poset, graph-theoretic properties of the relations involved, and topological properties of the spectrum (see, for example, Propositions~\ref{HausdorffLimit}--\ref{SequenceForSpace}). In \autoref{ss.connected}, we analyse closed connected subsets of the spectrum in terms of the trace these leave on the associated poset, as well as unicoherence of the spectrum (for spectra of posets induced from sequences of trees).

In \autoref{ss.FraisseSequences} we switch gears and move towards answering Question~\ref{q2}. 
We define Fra\"{i}ss\'e categories and their lax versions. The tension between on-the-nose amalgamation and lax-amalgamation, not present in the function setting, has to be assessed here. We then go into proving key properties of (lax-)Fra\"iss\'e sequences. Notably, we show in \autoref{LaxFraisseSpectra} that under natural assumptions two lax-Fra\"iss\'e sequences give posets with homeomorphic spectra; this allows us to talk about a topological space as \emph{the} Fra\"iss\'e limit of a certain category of graphs.
The main summarizing result, which could be dubbed ``Fra\"iss\'e theorem for categories of graphs and relations'' is \autoref{cor:FraisseLimit}.

In \autoref{ss.clique} we study the clique functor, allowing us to often reduce the technical effort required to show that a given category has amalgamation. A clique in a graph $G$ is a complete subgraph of $G$. From a graph $G$, it is possible to construct the graph of cliques $\mathsf XG$, which partially remembers the graph structure of $G$. The main result of \autoref{ss.clique}, \autoref{cor:cliqueclosed}, asserts that if a category of graphs $\mathbf C$ is sufficiently closed by the $\mathsf X$ operation, then it is enough to check amalgamation in $\mathbf C$ for morphisms which are induced by functions. 

In \autoref{s.Examples} we study specific categories and their Fra\"iss\'e sequences, giving specific answers to Question~\ref{q2}.
The section starts with a summarizing checklist for application of the theory to particular examples.
Each subsection then contains a main theorem characterizing the Fra\"iss\'e limit and (lax-)Fra\"iss\'e sequences in the corresponding category.
We first consider the category $\mathbf{D}$ discrete graphs (giving, unsurprisingly, the Cantor space), and then move to richer categories, actually utilizing relational morphisms, such the category $\mathbf{A}$ of paths with monotone morphisms. We show that this category has amalgamation, and identify the interval $[0,1]$ as the spectrum of the poset induced by any Fra\"iss\'e sequence in it. In \autoref{ss.Fans}, we consider two categories of fans, $\mathbf X$ and $\mathbf L$, obtaining the Cantor and the Lelek fan as generic spaces respectively. A notable result is that $\mathbf{L}$ recovers all smooth fans (i.e. subfans of the Cantor fans) as spectra of sequences (\autoref{SmoothFansRealizable}). Lastly, \autoref{s.Concluding} is dedicated to concluding remarks.

\section{Preliminaries: Relations, graphs, and the spectrum}\label{S.preliminaries}
This introductory section summarises results contained in \cite{BartosBiceV.Compacta}.

\subsection{Relation and graphs}\label{ss.relations}
Graphs, and relations between them (including the edge relation encoded in the definition of a graph) are our main focus. 
\subsubsection{Relations}
We view any ${\sqsupset}\subseteq B\times A$ as a relation `from $A$ to $B$'.
We call $\sqsupset$
\begin{itemize}
\item \emph{a function} if every $a\in A$ is related to exactly one $b\in B$.
\item \emph{surjective} if every $b\in B$ is related to at least one $a\in A$.
\item \emph{injective} if, for every $a\in A$, we have some $b\in B$ which is only related to $a$.
\item \emph{bijective} if $\sqsupset$ is both surjective and injective.
\end{itemize}

These notions of sur/in/bijectivity for relations generalise the usual notions for functions. The prefix `co' will be used to refer to the opposite/inverse relation $\sqsupset^{-1}\ =\ \sqsubset\ \subseteq A\times B$ (where $a\sqsubset b$ means $b\sqsupset a$), e.g. we say $\sqsubset$ is co-injective to mean that $\sqsupset$ is injective. For example, one can note that every co-injective relation is automatically surjective, and the converse also holds for functions.

If ${\sqni}\subseteq C\times B$ and ${\sqsupset}\subseteq B\times A$ are relations, the relation ${\sqni}\circ{\sqsupset}\subseteq C\times A$ is the composition of $\sqni$ and $\sqsupset$, that is, for all $a\in A$ and $c\in C$,
\[
c\sqni\circ\sqsupset a\qquad\Leftrightarrow\qquad\exists b\in B\, (c\sqni b\text{ and } b\sqsupset a).
\]
If ${\sqsupset}\subseteq B\times A$ is a relation and $T\subseteq B$, the \emph{preimage} of $T$ is given by
\begin{align}
\tag{Preimage}T^\sqsupset=\{a\in A:\exists t\in T\ (t\sqsupset a)\}.\\
\intertext{Likewise, the \emph{image} of any $S\subseteq A$ is the preimage of the opposite relation $\sqsubset$, i.e.}
\tag{Image}S^\sqsubset=\{b\in B:\exists s\in S\ (b\sqsupset s)\}.
\end{align}
We say $S\subseteq A$ \emph{refines} $T\subseteq B$ if it is contained in its preimage. The resulting refinement relation will also be denoted by $\sqsubset$, i.e. for any $S\subseteq A$ and $T\subseteq B$,
\[ 
S\sqsubset T\qquad\Leftrightarrow\qquad S\subseteq T^\sqsupset\qquad\Leftrightarrow\qquad\forall s\in S\ \exists t\in T\ (s\sqsubset t).\]
We will be mainly interested in the refinement relation between open covers of a topological space.

\subsubsection{Graphs}
A \emph{graph} is a set $G$ (the set of \emph{vertices}) on which we are given a symmetric reflexive relation $\sqcap$ (the \emph{edge relation}). For example, if $G$ is any family of non-empty sets then $\sqcap$ could be the overlap relation, i.e. where $A\mathrel{\sqcap}B$ if and only if $A\cap B\neq\emptyset$. These overlap graphs are our primary motivating examples and this is the main reason we require the edge relation on graphs to be reflexive.  In this paper, a graph will also always be finite and non-empty unless stated otherwise.

Following standard practice, we often refer to `the graph $G$', leaving the edge relation implicit. By an \emph{edge}, we mean a single element of $\sqcap$, i.e. a pair $(g,h)$ satisfying $g\sqcap h$ -- if $g=h$ then we call $(g,h)$ a \emph{loop}, while if $g\neq h$ then we say that $g$ and $h$ are \emph{adjacent}. The resulting adjacency relation will be denoted by $\sim$, i.e.
\[
\tag{Adjacency}g\sim h\qquad\Leftrightarrow\qquad g\sqcap h\quad\text{and}\quad g\neq h.
\]
For any $g\in G$, we call $g^\sim=g^\sqcap\setminus\{g\}$  the \emph{neighbourhood} of $g$ in $G$. We call $g\in G$ \emph{isolated} if it has no neighbours, and a graph is \emph{discrete} if all its vertices are isolated.

An \emph{induced subgraph} of a graph $G$ is a subset $H\subseteq G$ with the same edge relation just restricted to $H$. When applying graph theoretic terms to subsets of vertices, it is always the induced subgraph structure we are referring to unless otherwise specified. For example, we call $D\subseteq G$ discrete if the induced subgraph $D$ is discrete.

We will work with relations between graphs which behave well with respect to the edge relation.

\begin{dfn}\label{defin:relationsgraphs}
We call a relation ${\sqsupset}\subseteq H\times G$ between graphs $G$ and $H$
\begin{itemize}
\item \emph{edge-preserving} if $h\sqsupset g\sqcap g'\sqsubset h'$ implies $h\sqcap h'$.
\item \emph{edge-surjective} if $h\sqcap h'$ implies $h\sqsupset g\sqcap g'\sqsubset h'$, for some $g,g'\in G$.
\item \emph{edge-witnessing} if $h\sqcap h'$ implies $h\sqsupset g\sqsubset h'$, for some $g\in G$.
\end{itemize}
\end{dfn}

As the edge relation is always reflexive, we immediately see that
\[
\sqsupset\,\text{is edge-witnessing}\ \ \Rightarrow\ \ \sqsupset\,\text{is edge-surjective}\ \ \Rightarrow\ \ \sqsupset\,\text{is surjective}.
\]

\subsection{The spectrum of a poset}\label{ss:spectrum}

If $X$ is a topological space and $B\subseteq\mathsf{P}(X)$ is a (sub)basis of open sets, the inclusion relation between elements of $B$ turns it into a poset. Topological properties of $X$, and combinatorial properties of the basis, are reflected by order theoretic properties of the associated poset.
The main focus of \cite{BartosBiceV.Compacta} is to reverse this process, constructing topological spaces from posets.

For the remaining part of this section, we fix a poset $(\PP,\leq)$. When mentioning $\PP$, the order $\leq$ will be implicit. We view $\geq\ \subseteq\mathbb P\times\mathbb P$ as a relation; we can therefore use the notation introduced in \autoref{ss.relations}, so for example, if $B\subseteq\mathbb P$, the set $B^{\leq}$ equals $\{p\in\mathbb P\mid \exists b\in B\ (b\leq p)\}$.
\begin{dfn}
Let $(\mathbb P,\leq)$ be a poset. Let $B\subseteq \mathbb P$. Then
\begin{itemize}
\item $B$ is a \emph{band} if it is finite and $\mathbb P= B^{\leq}\cup B^\geq$;
\item $B$ is a \emph{cap} if it is refined by a band, i.e. if we have a band $B'\leq B$.
\end{itemize}
\end{dfn}
The collections of bands and caps of $\PP$ will be denoted by $\mathsf B\PP$ and $\mathsf C\PP$ respectively.

Fix a $\mathsf T_1$ topological space $X$, and say that $\PP$ is a basis of non empty open subsets of $X$ ordered by inclusion $\subseteq$. Then every cap of $\PP$ is a cover of $X$, by \cite[Proposition 1.7]{BartosBiceV.Compacta}.  When the converse holds, i.e. when every cover of $X$ by elements of $\PP$ is a cap, $\PP$ is said to be a \emph{cap-basis}. Cap-bases exist for all compact second countable $\mathsf T_1$ topological spaces, and in this case they can even be found by `massaging' a given basis -- see \cite[Proposition 1.11]{BartosBiceV.Compacta}.

The minimal elements of $\mathbb{P}$ will be called \emph{atoms}, i.e. $p\in\mathbb{P}$ is an atom if and only if $p^>=\emptyset$. The poset $\PP$ is said to be:
\begin{itemize}
\item \emph{predetermined} if for every $p\in\PP$ which is not an atom there is $q<p$ such that $q^<=p^\leq$;
\item \emph{Noetherian} if it has no infinite increasing sequences, or, alternatively, if the partial order $\geq$ is well-founded.
\end{itemize} 
Noetherian posets come with a \emph{rank} ordinal valued function, defined inductively, which associates to $p\in\PP$ the height of the set $p^<$.
More formally, the rank of $p\in\PP$ is given by  $\mathsf r(p)=\sup\{\mathsf r(q)+1: q > p\}$, where maximal elements are given rank $0$.
\begin{dfn}
A Noetherian poset $\mathbb P$ is an $\omega$\emph{-poset} if every principal filter $p^\leq$ is finite and there are only finitely many principal filters of a given size.
(Equivalently, a Noetherian poset is an $\omega$-poset if every element has finite rank and every set $\{p \in \PP: \mathsf{r}(p) \leq n\}$ is finite.) 

An $\omega$-poset is \emph{graded} if the rank function maps intervals to intervals, that is, for every $q<p$ and $n$ such that $\mathsf r(p)<n<\mathsf r(q)$ there is $p'$ with $q<p'<p$ and $\mathsf r(p')=n$.
\end{dfn}
If $\PP$ is an $\omega$-poset, we write $\PP=\bigcup_{n \in \omega} \PP_n$, where $\PP_n$ is the $n$\textsuperscript{th} level of $\PP$, given by those elements of $\PP$ which are minimal in the set of elements of rank $\leq n$, or, more formally, 
\[
\PP_n=\{p\in\PP: \mathsf r(p)\leq n\text{ and } \mathsf r(q)>n \text{ if } q<p\}.
\]
Note that every level is finite. Furthermore, levels are co-initial among all caps,  by \cite[Proposition~1.13]{BartosBiceV.Compacta}, i.e.
\[\tag{Co-initial}\label{CapsCoInitial}
\forall C\in\mathsf C\PP\ \exists n\ (\PP_n\leq C).
\]
If the $\omega$-poset $\PP$ is graded, distinct levels share only atoms, by \cite[Proposition 1.31]{BartosBiceV.Compacta}. In fact, a graded poset $\PP$ is atomless if and only if all its levels are disjoint and hence $\mathbb P_n=\{p\in\mathbb P : \mathsf r(p)=n\}$, for all $n\in\mathbb{N}$.  This will hold for all the posets $\mathbb{P}$ we apply the theory to in this article.

The following directly follows from \cite[Corollary~1.27 and Theorem~1.34]{BartosBiceV.Compacta}.
\begin{prp}\label{prp:allspaces}
Every second-countable compact $\mathsf T_1$ space has a cap-basis which is a graded predetermined $\omega$-poset (when ordered by inclusion).
\end{prp}

A key notion is the auxiliary relation of cap-order. For $A,B\subseteq\PP$, we say that $A$ is cap-below $B$ if it is easier to make $B$ a cap than to make $A$ one by simply adding more elements of $\mathbb{P}$. More precisely,
\[
\tag{Cap-Order} A\precsim B\quad\Leftrightarrow\quad\forall F\subseteq\PP\ (A\cup F\in\mathsf C\PP\Rightarrow B\cup F\in \mathsf C\PP).
\]
If $p\leq q$, then $p\precsim q$. When the converse holds, we say that $\PP$ is \emph{cap-determined}. Examples of cap-determined posets are given by cap-bases of $\mathsf T_1$ spaces (see \cite[Proposition 1.38]{BartosBiceV.Compacta}).

Before moving on, we introduce two weaker notions that we will need later on. These can be viewed as properties that any reasonable basis of a compact second countable Hausdorff space should have.

If there is no $p\in\mathbb{P}$ such that $p\precsim\emptyset$ then we say that $\PP$ is \emph{prime}. In general, a subset $Q\subseteq\PP$ is prime if there is no $q\in Q$ such that $q\precsim \PP\setminus Q$. As we will see below, prime subsets correspond naturally to closed subspaces. Lastly, a poset $\PP$ is \emph{branching} if, whenever $p<q$, we can find $r<q$ which is not comparable with $p$.

The diagram in \autoref{figure1} summarises some of the implications between the different concepts we just introduced, in the setting of $\omega$-posets (see \cite[Figure~1]{BartosBiceV.Compacta}):
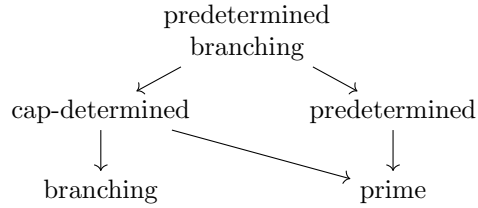
\begin{figure}[!ht]
\centering
\begin{tikzpicture}[
 x = {(5.5em, 0)},
 y = {(0, 3em)},
 text height = 1.5ex,
 text depth = 0.25ex,
 multiline/.style = {
 align = center,
 text height =,
 text width =,
 }
 ]
 \node (predetermined) at (1, -1) {predetermined};
 \node (prime) at (1, -2) {prime};
 \node (capdetermined) at (-1, -1) {cap-determined};
 \node (branching) at (-1, -2) {branching};
 \node (pb) at (0,0) [multiline] {predetermined \\ branching};
 
 \graph{
 (pb) -> {(capdetermined), (predetermined)},
 (predetermined)-> {(prime)},
 (capdetermined) -> {(branching), (prime)}
 };
\end{tikzpicture}
\caption{Implications between properties of $\omega$-posets.}
\label{figure1}
\end{figure}

\subsubsection{The spectrum}\label{ss:thespectrum}
Next we recall how to construct topological spaces from posets.
\begin{dfn}
Let $(\PP,\leq)$ be a poset. A subset $S\subseteq\PP$ is a \emph{selector} if it has non-empty intersection with every cap.
\end{dfn}
By Zorn's Lemma, every selector contains a minimal selector. We topologise the space of minimal selectors as follows.
\begin{dfn}
Let $(\mathbb P,\leq)$ be a poset. The \emph{spectrum} of $\mathbb P$ is the topological space 
\[
\mathsf S\mathbb P=\{S\subseteq\mathbb P : S\text{ is a minimal selector}\},
\] 
with the topology generated by the subbasic open sets of the form
\[p^\in=\{S\in \mathsf S\mathbb P: p\in S\}.\]
\end{dfn}
The representation $p\mapsto p^\in$ is not necessarily faithful, and there can even be $p\in\PP$ such that $p^\in$ is empty. This does not happen for prime posets and, in fact, $\PP$ is prime if and only if $p^\in\neq\emptyset$, for every $p\in\PP$ (see \cite[Definition 2.27]{BartosBiceV.Compacta}). More than that, one can characterise prime subsets of a poset $\PP$ as exactly those subsets which are unions of minimal selectors. The following was proved in \cite[\S 2.2]{BartosBiceV.Compacta}:
\begin{thm}\label{SpectrumCompactT1}
Let $(\PP,\leq)$ be a poset.
\begin{enumerate}
\item The spectrum $\mathsf S\PP$ is a compact $\mathsf T_1$ space.
\item If $\PP$ is an $\omega$-poset, then $\mathsf S\PP$ is second countable, every $S \in \Spec{\PP}$ is a filter, and $\{p^\in : p\in \PP\}$ is a basis.
\item If $\PP$ is a prime $\omega$-poset, then $\{p^\in: p\in \PP\}$ is a cap-basis for $\mathsf S\PP$.
\item $\PP$ is cap-determined if and only if the map $p\mapsto p^\in$ defines an order isomorphism onto $\{p^\in: p\in \PP\}$.
\end{enumerate}
\end{thm}

Before moving on, we record two lemmas about subposets of $\omega$-posets and their spectra. These did not appear in \cite{BartosBiceV.Compacta}, so we also provide proofs.

\begin{lemma} \label{UpwardsClosureHypothesis}
    Let $\QQ \subseteq \PP$ be atomless $\omega$-posets and let $S \in \Spec{\QQ}$.
    Then $S^{\leq_\PP} \in \Spec{\PP}$ if and only if $T \cap \QQ$ is infinite whenever $T \subseteq S^{\leq_\PP}$ is a minimal selector.
\end{lemma}

\begin{proof}
    Recall that minimal selectors are necessarily upwards closed and, for upwards closed subsets in atomless posets, being infinite is equivalent to being a selector.  So if $S^{\leq_\PP} \in \Spec{\PP}$ then the only possibility for $T$ is $S^{\leq_\PP}$ itself, and $S^{\leq_\PP}\cap\mathbb{Q}=S$ is infinite.  On the other hand, $S$ and hence $S^{\leq_\PP}$ is infinite and thus a selector in $\mathbb{P}$.  So $S^{\leq_\PP}$ contains a minimal selector $T$, necessarily with $T\cap\QQ\subseteq S$.  If $T\cap\QQ$ is infinite then it is a selector in $\mathbb{Q}$ and hence $T\cap\QQ=S$, by the minimality of $S$.  It then follows that $S^{\leq_\PP}=T\in\mathsf{S}\mathbb{P}$.
\end{proof}

\begin{lemma} \label{UpwardsClosureHomeomorphism}
  Let $\QQ \subseteq \PP$ be atomless $\omega$-posets.
  If $S^{\leq_\PP} \in \Spec{\PP}$, for every $S \in \Spec{\QQ}$, then $S \mapsto S^{\leq_\PP}$ embeds $\Spec{\QQ}$ onto the subspace $\{T \in \Spec{\PP}: T \cap \QQ\text{ is infinite}\}$, with the inverse map $T \mapsto T \cap \QQ$.
\end{lemma}
\begin{proof}
    Let $S \in \Spec{\QQ}$ and let $T = S^{\leq_\PP} \in \Spec{\PP}$.
    Since $S \subseteq \QQ$ is upwards closed, we have $T \cap \QQ = S$, which is infinite as $\QQ$ is atomless.
    For every $p \in S^{\leq_\PP}$ there is $q \in S$ with $q \leq p$, so $S \in q^\in$ and for every $S' \in q^\in \cap \Spec{\QQ}$ we have $S'^{\leq_\PP} \in q^\in \subseteq p^\in$.
    Hence, $S \mapsto S^{\leq_\PP}$ is continuous.

    Now take $T \in \Spec{\PP}$ such that $T \cap \QQ$ is infinite and let $S = T \cap \QQ$.
    It follows that $S$ is an upwards closed selector in $\QQ$ as $T$ is upwards closed and intersects all levels of $\QQ$ since the intersection is infinite.
    Hence, there is $S' \in \Spec{\QQ}$ with $S' \subseteq S$.
    We have $S'^{\leq_\PP} \in \Spec{\PP}$ with $S'^{\leq_\PP} \subseteq T$, so $S'^{\leq_\PP} = T$ since $T$ is minimal.
    By the first paragraph it follows that $S' = T \cap \QQ = S$, and we have $S^{\leq_\PP} = T$, so we have the desired mutually inverse bijections.
    Let $q \in T \cap \QQ$.
    Then $T \in q^\in$ and for every $T' \in q^\in$ we have $T' \cap \QQ \in q^\in$, so $T \mapsto T \cap \QQ$ is continuous.
\end{proof}

For two elements $p, q$ of a poset $\PP$, we write $p\wedge q$ if there is $r\in\PP$ with $r\leq p,q$.
Levels of a predetermined $\omega$-poset faithfully represent a co-initial family of minimal open covers of the spectrum. (Recall, a cover $C$ of a set $X$ is minimal if $C\setminus\{c\}$ does not cover $X$ for any $c\in C$.)

\begin{prp} \label{prp:coverslevels}
    Let $\PP$ be a predetermined $\omega$-poset, and for every $n \in \omega$ let $C_n = \{p^\in: p \in \PP_n\}$.
    \begin{enumerate}
        \item\label{itm:MinCover} Every $C_n$ is a minimal open cover of $\Spec{\PP}$.
        \item\label{itm:Overlap} For every $p, q \in C_n$ we have $p \wedge q$ if and only if $p^\in \cap q^\in \neq \emptyset$.
        \item\label{itm:CoInitial} Every open cover of $\Spec{\PP}$ is refined by some $C_n$.
    \end{enumerate}
\end{prp}
\begin{proof}
By \cite[Proposition~2.8]{BartosBiceV.Compacta}, a family $\{p^\in: p \in C\}$ for some $C \subseteq \PP$ is a cover if and only if $C$ is a cap.
For \ref{itm:MinCover}, note predetermined $\omega$-posets are level-injective (see \cite[\S1.4]{BartosBiceV.Compacta}), by \cite[Corollary~1.26]{BartosBiceV.Compacta}, so each level $\mathbb{P}_n$ is a minimal cap, by \cite[Corollary~1.21]{BartosBiceV.Compacta}.
Therefore, every $C_n$ is a minimal cover as subcovers of $C_n$ correspond to subcaps of $\PP_n$. 
\ref{itm:Overlap} follows from \cite[Proposition~2.30]{BartosBiceV.Compacta} as every level-injective $\omega$-poset is prime, by \cite[Proposition~2.28]{BartosBiceV.Compacta}.
For \ref{itm:CoInitial}, every open cover is refined by a basic cover of the form $\{p^\in: p \in C\}$, which is refined by some $C_n$ as the cap $C$ is refined by some level $\PP_n$ in the poset $\PP$, by \eqref{CapsCoInitial}.
\end{proof}

The concepts below are designed to characterise topological properties of $\mathsf S\PP$ (e.g. being Hausdorff) via properties of $\PP$. If $C\in\mathsf C\PP$ and $p\in\PP$, the \emph{star} of $p$ in $C$ is given by
\[
Cp=\{q\in C: q\wedge p\}.
\]
From this we define an auxiliary relation on $\PP$.  Specifically, if $p,q\in\PP$, we define
\[
p\vartriangleleft_C q\ \ \iff\ \ Cp\leq q.
\]
Uniformising over all possible caps, we get the \emph{star-below} relation defined by
\[\tag{Star-Below}
p\vartriangleleft q\ \ \iff\ \ \exists C\in\mathsf C\PP\ (p\vartriangleleft_C q).
\]
When $\mathbb{P}$ is prime, the relation $\vartriangleleft$ encodes closed containment of the corresponding open sets, i.e.~$p\vartriangleleft q$ if and only if $\mathrm{cl}(p^\in)\subseteq q^\in$, by \cite[Proposition 2.33]{BartosBiceV.Compacta}.

Important for us will be the following two facts,  isolated for future reference:
\begin{itemize}
\item $\vartriangleleft$ and $\leq$ are compatible, i.e.~for any cap $C$ and $p,p',q,q'\in\PP$, 
\[
\tag{Compatibility}\label{Compatibility}p\leq p'\vartriangleleft_Cq'\leq q\ \ \Rightarrow \ \ p\vartriangleleft_C q.
\]
\end{itemize}
Since if $B\leq C$ and $p\vartriangleleft_Cq$ then $pB\leq pC\leq q$ and hence $p\vartriangleleft_Bq$, it follows that
\begin{itemize}
\item if $B$ and $C$ are caps then 
\begin{equation}\label{RefiningStars}
B\leq C\ \ \Rightarrow\ \ {\vartriangleleft_C} \subseteq {\vartriangleleft_B}.
\end{equation}
\end{itemize}
For more information about stars see \cite[\S2.5]{BartosBiceV.Compacta}.

The last condition we introduce is a regularity condition: An $\omega$-poset $\PP$ is said to be \emph{regular} if every cap is star-refined by another cap, i.e.
\[
\forall C\in\mathsf C\PP\ \exists D\in\mathsf C\PP\ (D\vartriangleleft C).
\]
Regularity allows us to separate points in the spectrum, by \cite[Corollary 2.40]{BartosBiceV.Compacta}.

\begin{thm}\label{RegularImpliesHausdorff}
Let $\PP$ be a prime $\omega$-poset. Then $\mathsf S\PP$ is Hausdorff if and only if $\PP$ is regular.
\end{thm}

\section{The graph category}\label{S.GraphCategory}

Officially, we denote the objects and morphisms of a category $\mathbf{C}$ by $\mathsf{ob}(\mathbf{C})$ and $\mathsf{mor}(\mathbf{C})$ respectively. However, following standard abuse of notation, we also let $\mathbf{C}$ itself stand for both $\mathsf{ob}(\mathbf{C})$ and $\mathsf{mor}(\mathbf{C})$ whenever convenient. For two objects $A,B\in\mathbf{C}$, we denote the hom-set of morphisms from $A$ to $B$ by
\[
\mathbf{C}^B_A=\{{\sqsupset}\in\mathsf{mor}(\mathbf{C}):\mathrm{dom}(\sqsupset)=A\text{ and }\mathrm{cod}(\sqsupset)=B\}.
\]
We work exclusively in the graph category $\mathbf{G}$ (and various subcategories thereof). The objects of the graph category are just non-empty finite graphs, i.e.
\[
\mathsf{ob}(\mathbf{G})=\{(G,\sqcap):0<|G|<\infty\text{ and}=_G\ \subseteq\sqcap=\sqcap^{-1}\subseteq G\times G\}.
\]
The morphisms of $\mathbf{G}$ are edge-preserving relations (\autoref{defin:relationsgraphs}) under their usual composition (note that a composition of edge-preserving relations is again edge-preserving), i.e.
\[
\mathbf{G}_G^H=\{{\sqsupset}\subseteq H\times G:\forall g,g'\in G, h,h'\in H\ (h\sqsupset g\sqcap g'\sqsubset h'\Rightarrow h\sqcap h')\}.
\]
While the graph category is convenient for general definitions and observations, our real focus will be on certain subcategories $\mathbf{C}$ of $\mathbf{G}$ where both the graphs and morphisms satisfy various extra conditions. Often the morphisms satisfying these conditions will still be quite `large', in either or both of the following senses.

\begin{dfn}\label{dfn:idealswide}
We call a collection of morphisms $M$ in a category $\mathbf{C}$
\begin{itemize}
\item \emph{wide} if, for all $H\in\mathbf{C}$, there is $G\in\mathbf{C}$ such that $\mathbf{C}_G^H\cap M\neq\emptyset$;
\item an \emph{ideal} if $M$ is closed under composition with arbitrary morphisms of $\mathbf{C}$;\footnote{Ideals are always two-sided, i.e.~closed under composition on the right and on the left.}
\item \emph{lax-closed} if $M$ is closed under taking superrelations that are morphisms in $\mathbf{C}$, i.e. for every ${\sqsupset} \in M$ and ${\sqni} \in \mathbf{C}$ such that ${\sqsupset} \subseteq {\sqni}$ we have ${\sqni} \in M$.
\end{itemize}
\end{dfn}

If $\mathbf{D}$ is a subcategory of some larger category $\mathbf{C}$, we immediately see that its morphisms are wide in $\mathbf{C}$ precisely when $\mathbf{D}$ has the same objects as $\mathbf{C}$. In this case, following standard practice, we call $\mathbf{D}$ a \emph{wide subcategory} of $\mathbf{C}$. 

In particular, $\mathbf{G}$ has wide subcategories of co-injective, co-surjective and edge-surjective morphisms, which we denote by $\mathbf{I}$, $\mathbf{S}$ and $\mathbf{E}$, i.e.
\begin{align*}
\mathbf{I}_G^H&=\{{\sqsupset}\in\mathbf{G}_G^H:\ \sqsupset\text{is co-injective}\}=\{{\sqsupset}\in \mathbf G_G^H: \forall h\in H\ \exists g\in G\ (g^{\sqsubset}=\{h\})\}.\\
\mathbf{S}_G^H&=\{{\sqsupset}\in\mathbf{G}_G^H:\ \sqsupset\text{is co-surjective}\}=\{{\sqsupset}\in\mathbf G_G^H: \forall g\in G\ (g^\sqsubset\neq\emptyset)\}.\\
\mathbf{E}_G^H&=\{{\sqsupset}\in\mathbf{G}_G^H:\ \sqsupset\text{is edge-surjective}\}=\{{\sqsupset}\in\mathbf G_G^H:\forall h\in H\ (h^\sqcap\subseteq h^{\sqsupset\sqcap\sqsubset})\}.
\end{align*}
Recall that $h^\sqcap\subseteq h^{\sqsupset\sqcap\sqsubset}$ means 
\[
\forall h'\,(h\sqcap h'\text{ implies }h\sqsupset g\sqcap g'\sqsubset h'\text{ for some }g,g'\in G).
\]
Let $\mathbf{B}=\mathbf{I}\cap\mathbf{S}$ denote the wide subcategory of co-bijective morphisms, i.e.
\[
\mathbf{B}^H_G=\mathbf{I}^H_G\cap\mathbf{S}^H_G=\{{\sqsupset}\in\mathbf{G}_G^H:\ \sqsupset\text{is co-bijective}\}.
\]
If ${\sqsupset}\in\mathbf{G}^H_G$ and $C\subseteq H$, we  define the \emph{strict preimage} by
\[
C_\sqsupset=\{g\in G:g^\sqsubset=C\}.
\]
As usual, we omit the curly braces for singletons, writing $q_\sqsupset$ instead of $\{q\}_\sqsupset$.

Let us isolate further properties of morphisms.
\begin{dfn}\label{defin:relationsgraphs2}
A morphism ${\sqsupset}\in\mathbf{G}^H_G$ is said to be
\begin{itemize}
 \item \emph{anti-injective} if $|h^\sqsupset|\geq2$, for all $h\in H$.
 \item \emph{strictly anti-injective} if $|h_\sqsupset|\geq2$, for all $h\in H$.
 \item \emph{star-refining} if for all $g\in G$ we have $h\in H$ with $g^\sqcap\subseteq h^\sqsupset$.
\end{itemize}
\end{dfn}

These notions are crucial for us, see for example \autoref{prp:Summary1} below for an indication of why morphisms with the above properties are significant. As a toy example, note that every function from a discrete graph is trivially star-refining, while every function onto a discrete graph is trivially edge-witnessing.

We dedicate the rest of this section to study the above properties, how they interact, and some of their refinements in particular categories.

The following simple observation is easy to see from the definitions.
\begin{lemma} \label{LaxClosed}
    The families of anti-injective, edge-witnessing, and star-refining morphisms in $\mathbf{G}$ (and so in every subcategory) are lax-closed. \qed
\end{lemma}

Note that being star-refining strengthens co-surjectivity and strict anti-injectivity strengthens co-injectivity. Further note that identity morphisms can not be anti-injective. Discrete graphs aside, they are not star-refining or edge-witnessing either. For this reason, such morphisms do not form a subcategory. However, they do form ideals within appropriate subcategories.

\begin{prp}\label{Ideals}\ 
\begin{enumerate}
\item\label{P.Idealsc1} The anti-injective co-injective morphisms form an ideal in $\mathbf{I}$.
\item\label{P.Idealsc2} The strictly anti-injective morphisms form an ideal in $\mathbf{I}$.
\item\label{P.Idealsc3} The star-refining morphisms form an ideal in $\mathbf{S}$.
\item\label{P.Idealsc4} The edge-witnessing morphisms form an ideal in $\mathbf{E}$.
\end{enumerate}
\end{prp}

\begin{proof}
\ref{P.Idealsc1}: Take any anti-injective ${\sqsupset}\in\mathbf{G}_G^H$. If ${\sqni}\subseteq I\times H$ is co-injective and, in particular, surjective then, for every $i\in I$, we have $h\in H$ with $i\sqni h$. Then $|i^{\sqni\sqsupset}|\geq|h^\sqsupset|\geq2$, showing that $\sqni\circ\sqsupset$ is also anti-injective.

Now take co-injective ${\sqni}\subseteq G\times F$. For any $h\in H$, the anti-injectivity of $\sqsupset$ yields distinct $g,g'\in h^\sqsupset$. The injectivity of $\sqin$ then yields $f,f'\in F$ with $f^{\sqin}=\{g\}$ and $f'^{\sqin}=\{g'\}$. In particular, $f$ and $f'$ are distinct elements of $h^{\sqsupset\sqni}$, again showing that $\sqsupset\circ\sqni$ is anti-injective.

\ref{P.Idealsc2}: Take any strictly anti-injective ${\sqsupset}\in\mathbf{G}_G^H$. If ${\sqni}\in\mathbf{I}^I_H$ and $i\in I$ then we have $h\in i_\sqsupset$ so $|i_{\sqni\sqsupset}|\geq|h_\sqsupset|\geq2$, showing $\sqni\circ\sqsupset$ is also strictly anti-injective.

On the other hand, if ${\sqni}\in\mathbf{I}^G_F$ and $h\in H$ then the strict preimages $g_{\sqni}$, for $g\in h_\sqsupset$, are all disjoint and non-empty and hence $|h_{\sqsupset\sqni}|\geq|h_\sqsupset|\geq2$, showing that $\sqsupset\circ\sqni$ is again strictly anti-injective.



\ref{P.Idealsc3}: Take any star-refining ${\sqsupset}\in\mathbf{G}_G^H$. For any $g\in G$, we then have $h\in H$ with $g^\sqcap\subseteq h^\sqsupset$. If ${\sqni}\subseteq I\times H$ is co-surjective, we then have $i\in I$ with $i\sqni h$ so $h^\sqsupset\subseteq i^{\sqni\sqsupset}$ and hence $g^\sqcap\subseteq i^{\sqni\sqsupset}$, showing that $\sqni\circ\sqsupset$ is also star-refining.

Now take ${\sqni}\in\mathbf{S}_F^G$ so, for every $f\in F$, we have $g\in G$ with $g\sqni f$. For every $f'\sqcap f$, we then also have $g'\in G$ with $g'\sqni f'$, which implies $g\sqcap g'$, as $\sqni$ is $\sqcap$-preserving, showing that $f^\sqcap \subseteq g^{\sqcap\sqni}$. Taking $h\in H$ with $g^\sqcap\subseteq h^\sqsupset$, it follows that $f^\sqcap\subseteq h^{\sqsupset\sqni}$, so again ${\sqsupset}\circ{\sqni}$ is star-refining.

\ref{P.Idealsc4} Take any edge-witnessing ${\sqsupset}\in\mathbf{G}_G^H$. If ${\sqni}\in\mathbf{E}_F^G$ then $\sqni$ is edge-surjective and, in particular, surjective. So if $h,h'\in H$ satisfy $h\sqcap h'$ then we have $g\in G$ with $h,h'\sqsupset g$ and then we can further take $f\in F$ with $g\sqni f$, showing that ${\sqsupset}\circ{\sqni}$ is also edge-witnessing.

Now take ${\sqni}\in\mathbf{E}_H^I$. If $i,i'\in I$ satisfy $i\sqcap i'$, this means we have $h\sqin i$ and $h'\sqin i'$ with $h\sqcap h'$. As $\sqsupset$ is edge-witnessing, we then have $g\in G$ with $h,h'\sqsupset g$, again showing that ${\sqni}\circ{\sqsupset}$ is edge-witnessing.
\end{proof}

We summarize the main classes of special morphisms and implications between them in \autoref{fig:morphisms}.
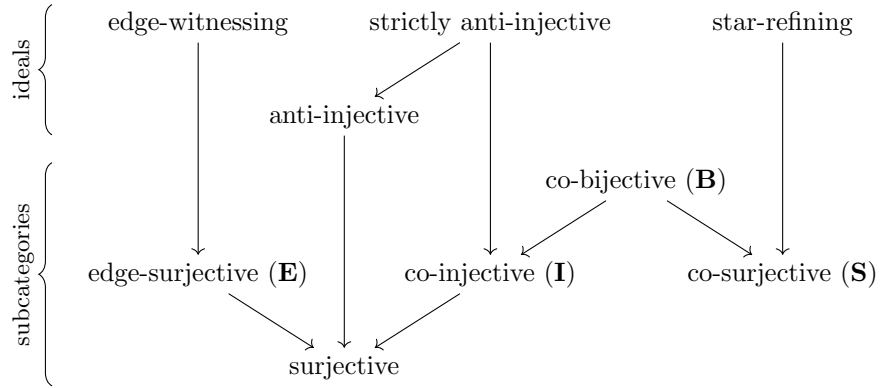
\begin{figure}[!ht]
\centering
\begin{tikzpicture}[
 x = {(5.5em, 0)},
 y = {(0, 3.5em)},
 text height = 1.5ex,
 text depth = 0.25ex,
 ]
 \node (s) at (0, 0) {surjective};
 \node (es) at (-1, 1) {edge-surjective ($\mathbf{E}$)};
 \node (ci) at (1, 1) {co-injective ($\mathbf{I}$)};
 \node (cs) at (3, 1) {co-surjective ($\mathbf{S}$)};
 \node (cb) at (2, 2) {co-bijective ($\mathbf{B}$)};
 \node (ew) at (-1, 3.7) {edge-witnessing};
 \node (ss) at (3, 3.7) {star-refining};
 \node (sai) at (1, 3.7) {strictly anti-injective};
 \node (ai) at (0, 2.7) {anti-injective};
 
 \graph{
 (cb) -> (ci) -> (s),
 (cb) -> (cs),
 (ew) -> (es) -> (s),
 (ss) -> (cs),
 (sai) -> {(ai), (ci)},
 (ai) -> (s),
 };
 
 \draw[decorate, decoration={brace, amplitude=5pt}] (-2, -0.2) -- node[above=4pt, font=\small, sloped]{subcategories} (-2, 2.2);
 \draw[decorate, decoration={brace, amplitude=5pt}] (-2, 2.5) -- node[above=4pt, font=\small, sloped]{ideals} (-2, 3.9);
\end{tikzpicture}
\caption{Implications between main classes of special morphisms.}
\label{fig:morphisms}
\end{figure}

Star-refining morphisms are related to co-edge-witnessing morphisms ${\sqsupset}\in\mathbf{G}_G^H$, i.e. where $\sqsubset$ is edge-witnessing. Explicitly this means that, for all $f,g\in G$,
\[\tag{Co-Edge-Witnessing}f\mathrel{\sqcap}g\qquad\Rightarrow\qquad\exists h\in H\ (f\sqsubset h\sqsupset g).\]
This is especially true for graphs of low degree. First recall that the \emph{degree} of a vertex $g$ is the size of its neighbourhood $|g^\sim|$, i.e. the number of vertices it is adjacent to. The degree of a graph $G$ is defined to be the maximum degree of its vertices, which we denote by
\begin{equation*}
\mathrm{deg}(G)=\max_{g\in G}|g^\sim|.
\end{equation*}

\begin{prp}\label{StarSurjectiveCoEdgeWitnessing}
 Any star-refining ${\sqsupset}\in\mathbf{G}_G^H$ is co-edge-witnessing. If $\mathrm{deg}(F)\leq2$, a composition of co-edge-witnessing ${\sqsupset}\in\mathbf{G}_G^H$ and ${\sqni}\in\mathbf{G}^G_F$ is star-refining, i.e.
 \[
 \text{$\sqsupset$ and $\sqni$ are co-edge-witnessing}\qquad\Rightarrow\qquad\text{$\sqsupset\circ\sqni$ is star-refining}.
 \]
\end{prp}

\begin{proof}
If ${\sqsupset}\in\mathbf{G}_G^H$ is star-refining then, for any $f,g\in G$ with $f\mathrel{\sqcap}g$, we have $h\in H$ with $f,g\in g^\sqcap\subseteq h^\sqsupset$, showing that $\sqsupset$ is co-edge-witnessing. On the other hand, say $\mathrm{deg}(F)\leq2$ and we have co-edge-witnessing ${\sqsupset}\in\mathbf{G}_G^H$ and ${\sqni}\in\mathbf{G}^G_F$. For any $f\in F$, we then have at most two other $e,e'\in f^\sqcap$. As $\sqni$ is co-edge-witnessing, we have $g$ and $g'$ with $e,f\in g^{\sqni}$ and $e,f'\in g'^{\sqni}$. In particular, $e\in g^{\sqni}\cap g'^{\sqni}$ so $g\sqcap g'$. As $\sqsupset$ is also co-edge-witnessing, we have $h\in H$ with $g,g'\in h^\sqsupset$ and hence $f^\sqcap=\{e,e',f\}\subseteq\{g,g'\}^{\sqni}\subseteq h^{\sqsupset\sqni}$. This shows that $\sqsupset\circ\sqni$ is star-refining.
\end{proof}

A \emph{triangle} $T$ in a graph $G$ is a complete subgraph with exactly $3$ vertices, i.e. $T=\{f,g,h\}\subseteq G$ and $f\sim g\sim h\sim f$. We call $G$ \emph{triangle-free} if it has no triangles.

\begin{prp}\label{StarEdgeSurjectiveEdgeWitnessing}
 If ${\sqsupset}\in\mathbf{G}_G^H$ and $H$ is triangle-free then
 \[\sqsupset\,\text{is star-refining and edge-surjective}\qquad\Rightarrow\qquad\sqsupset\,\text{is edge-witnessing}.\]
\end{prp}

\begin{proof}
 Assume ${\sqsupset}\in\mathbf{G}_G^H$ is star-refining and edge-surjective but not edge-witnessing. This means we have adjacent $h,h'\in H$ such that $h^\sqsupset$ and $h'^\sqsupset$ are disjoint. However, we still have $g\sqsubset h$ and $g'\sqsubset h'$ with $g\sqcap g'$, by edge-surjectivity. If we had $i\in H$ with $g,g'\sqsubset i$, necessarily with $h\neq i\neq h'$, then edge-preservation would imply that $\{h,h',i\}$ is a triangle of $H$. Thus if $H$ is triangle-free then $i^\sqsupset$ can not contain both $g$ and $g'$, for any $i\in H$. In particular, $g^\sqcap\nsubseteq i^\sqsupset$, for any $i\in H$, which contradicts our assumption that $\sqsupset$ is star-refining.
\end{proof}

Recall that a vertex $g\in G$ is isolated if it has degree $0$. If the graphs in question have no isolated vertices, co-injective morphisms are often anti-injective.

\begin{prp} \label{StarEdgeAntiInjective}
 If $H$ has no isolated vertices then, for any ${\sqsupset}\in\mathbf{I}^H_G$,
 \begin{align*}
 \sqsupset\,\text{is edge-witnessing}\qquad&\Rightarrow\qquad\sqsupset\,\text{is anti-injective}.\\
 \intertext{Similarly, if $G$ has no isolated vertices then, for any ${\sqsupset}\in\mathbf{I}^H_G$,}
 \sqsupset\,\text{is star-refining}\qquad&\Rightarrow\qquad\sqsupset\,\text{is anti-injective}.
 \end{align*}
\end{prp}

\begin{proof}
 Take an edge-witnessing ${\sqsupset}\in\mathbf{I}^H_G$. If $H$ has no isolated vertices, every $h\in H$ is adjacent to some $i\in H$. We then have $g\in G$ witnessing $h\mathrel{\sqcap}i$, i.e. $g^\sqsubset=\{h,i\}$. As $\sqsupset$ is co-injective, we must also have some other $f\in G$ with $f^\sqsubset=\{h\}$ and hence $\{f,g\}\subseteq h^\sqsupset$. As $h$ was arbitrary, this shows that $\sqsupset$ is anti-injective.
 
 Now take any star-refining ${\sqsupset}\in\mathbf{I}^H_G$. As $\sqsupset$ is co-injective, for every $h\in H$, we have $g\in G$ with $g^\sqsubset=\{h\}$. Since $\sqsupset$ is star-refining, there is $h'\in H$ with $g^\sqcap\subseteq h'^\sqsupset$. As $g\in g^\sqcap$, the only possibility is that $h'=h$. Thus $|h^\sqsupset|\geq|g^\sqcap|\geq2$, as long as $g$ is not isolated. If $G$ has no isolated vertices, this shows that $\sqsupset$ is anti-injective.
\end{proof}

Under certain conditions, we can also show surjectivity implies edge-surjectivity.

A \emph{cycle} is a connected graph in which very vertex has degree $2$. A graph $G$ is \emph{acyclic} if it contains no cycles, or, equivalently, if for all adjacent $a$ and $b$ we can partition $G$ into subsets $A\ni a$ and $B\ni b$ such that $(a,b)$ is the only edge between $A$ and $B$.

\begin{prp}\label{SurjectiveImpliesEdgeSurjective}
 If $G$ is connected and $H$ is acyclic then, for any ${\sqsupset}\in\mathbf{S}_G^H$,
 \[\text{$\sqsupset$ is surjective}\qquad\Rightarrow\qquad\text{$\sqsupset$ is edge-surjective}.\]
\end{prp}

\begin{proof}
 Take any adjacent $a,b\in H$. As $H$ is acyclic, we can partition $H$ into subsets $A\ni a$ and $B\ni b$ such that $(a,b)$ is the only edge between them. If $\sqsupset$ is surjective then $A^\sqsupset$ and $B^\sqsupset$ are non-empty and cover $G$. As $G$ is connected, it follows that we have $f\in A^\sqsupset$ and $g\in B^\sqsupset$ with $f\mathrel{\sqcap}g$ (if $A^\sqsupset$ and $B^\sqsupset$ are disjoint then they partition $G$ so this follows immediately from connectedness, otherwise we have $j\in A^\sqsupset\cap B^\sqsupset$ and we can just set $f=g=j$). Pick $a'\in A$ and $b'\in B$ such that $f\sqsubset a'$ and $g\sqsubset b'$. Since $f\sqcap g$ and $\sqsupset$ is edge-preserving, we have $a'\mathrel{\sqcap}b'$. Since $(a,b)$ is the only edge between $A$ and $B$, it follows that $a'=a$ and $b'=b$. This shows that $\sqsupset$ is edge-surjective.
\end{proof}

A \emph{tree} is a connected acyclic graph.
Since we shall be concerned with categories of trees in our examples, let us summarize the special properties of morphisms between trees, as a corollary of Propositions~\ref{StarEdgeSurjectiveEdgeWitnessing}--\ref{SurjectiveImpliesEdgeSurjective}.
\begin{cor} \label{TreeMorphismsSummary}
    Let $G$ and $H$ be nondegenerate trees and let ${\sqsupset} \in \mathbf{B}^H_G$.
    Then ${\sqsupset}$ is edge-surjective and we have
    \[\text{$\sqsupset$ is star-refining}\quad\Rightarrow\quad\text{$\sqsupset$ is edge-witnessing}\quad\Rightarrow\quad\text{$\sqsupset$ is anti-injective}.\]
\end{cor}


In \autoref{ss.arc} we will consider morphisms satisfying certain extra properties. First let us call ${\sqsupset}\in\mathbf{G}^H_G$ \emph{monotone} if, for all $C\subseteq H$,
\[\tag{Monotonicity}\label{def:monotone}C\text{ is connected}\qquad\Rightarrow\qquad C^\sqsupset\text{ is connected}.\]
If ${\sqsupset}\in\mathbf{G}^H_G$ is monotone, the preimage of each vertex $h^\sqsupset$ is connected. If $\sqsupset$ is edge-surjective, this is also sufficient for $\sqsupset$ to be monotone.

\begin{prp}\label{MonotoneEdgeSurjective}
 If ${\sqsupset}\in\mathbf{G}^H_G$ is surjective and monotone, then $\sqsupset$ is edge-surjective. Moreover, if ${\sqsupset}\in\mathbf{G}^H_G$ is edge-surjective, then
\[
\forall h\in H\ (h^\sqsupset\text{ is connected})\qquad\Rightarrow\qquad\sqsupset\,\text{is monotone}.
\]
\end{prp}

\begin{proof}
If ${\sqsupset}\in\mathbf{G}^H_G$ is monotone then, for any $h,h'\in H$ with $h\mathrel{\sqcap}h'$, it follows that $h^\sqsupset\cup h'^\sqsupset$ is connected. If $\sqsupset$ is also surjective then $h^\sqsupset\neq\emptyset\neq h'^\sqsupset$ and so connectivity implies that we have $g\sqsubset h$ and $g'\sqsubset h'$ with $g\mathrel{\sqcap}g'$, showing $\sqsupset$ is edge-surjective.

Now assume ${\sqsupset}\in\mathbf{G}^H_G$ is edge-surjective and $h^\sqsupset$ is connected, for all $h\in H$. Take any connected $C\subseteq H$ and say we had a discrete partition $A,B\subseteq C^\sqsupset$, i.e. $C = A \cup B$ with no edge between $A$ and $B$.
If we had $c\in A^\sqsubset\cap B^\sqsubset\cap C$ then $A\cap c^\sqsupset$ and $B\cap c^\sqsupset$ would discretely partition $c^\sqsupset$, contradicting our assumption on $\sqsupset$. But if $A^\sqsubset\cap B^\sqsubset\cap C=\emptyset$ then $A^\sqsubset\cap C$ and $B^\sqsubset\cap C$ would discretely partition $C$, as $\sqsupset$ is edge-surjective, contradicting its connectedness. Thus no such partition exists, i.e. $C^\sqsupset$ is also connected, showing that $\sqsupset$ is monotone.
\end{proof}

Monotone edge-witnessing morphisms are often also co-edge-witnessing.

\begin{prp}\label{EdgevsCoEdgeWitnessing}
 If $G$ is acyclic then, for any monotone ${\sqsupset}\in\mathbf{S}_G^H$,
 \begin{equation*}
 \sqsupset\,\text{is edge-witnessing}\qquad\Rightarrow\qquad\sqsupset\,\text{is co-edge-witnessing}.
 \end{equation*}
\end{prp}
\begin{proof}
 Take any $e,f\in G$ with $e\mathrel{\sqcap}f$. If $e=f$ then the co-surjectivity of $\sqsupset$ already yields $h\in e^\sqsubset=f^\sqsubset$. Otherwise, as $G$ is acyclic, we can partition $G$ into $E\ni e$ and $F\ni f$ with $E^\sqcap\cap F^\sqcap=\{e,f\}$. As $\sqsupset$ is co-surjective, we have $j\sqsupset e$ and $k\sqsupset f$, necessarily with $j\mathrel{\sqcap}k$, as $\sqsupset$ is also edge-preserving. If $\sqsupset$ is also edge-witnessing then we have $g\in j^\sqsupset\cap k^\sqsupset$ so $j\in g^\sqsubset\cap e^\sqsubset$ and $k\in g^\sqsubset\cap f^\sqsubset$. But $g$ lies in $E$ or $F$ and hence $j$ or $k$ lies in $E^\sqsubset\cap F^\sqsubset$. However, $\sqsupset$ is monotone so, for any $h\in E^\sqsubset\cap F^\sqsubset$, $h^\sqsupset$ must be a connected subset containing elements of $E$ and $F$, which can only happen if $e,f\in h^\sqsupset$. This shows that $\sqsupset$ is co-edge-witnessing.
\end{proof}

To show that the ideals of edge-witnessing and star-refining morphisms are wide in a given category, we often use the following result, which, for every graph $G$, yields a graph $H$ of a similar shape together with a suitable morphism from $H$ to $G$.
 \begin{lemma}[Edge splitting] \label{EdgeSplitting}
    For any graph $G$ we consider the graph $H$ where every edge of $G$ is split intro three edges by two new vertices, namely $H = G \cup \{s_{g, g'}: g \sim g' \in G\}$ with the edge relation $g \sqcap s_{g, g'} \sqcap s_{g', g} \sqcap g'$ for every adjacent $g, g' \in G$.
    Moreover we consider ${\sqsupset} \subseteq G \times H$ defined by $g^{\sqsubset} = \{g\}$ and $s_{g, g'}^\sqsubset = \{g, g'\}$.
    Then ${\sqsupset} \in \mathbf{B}^G_H$ and it is monotone, edge-witnessing, and star-refining.
\end{lemma}
\begin{proof}
    It is clear from the definition that $\sqsupset$ is co-bijective and edge-witnessing.
    As $\{g, s_{g, g'}, s_{g', g}, g'\}^\sqsubset = \{g, g'\}$ for every adjacent $g, g' \in G$, we have that $\sqsupset$ is edge-preserving.
    For every $h \in H$ we have $h^\sqcap = \{g, s_{g, g'}: g' \sim_G g\}$ if $h = g \in G$ and $h^\sqcap = \{g, s_{g, g'}, s_{g', g}\}$ if $h = s_{g, g'}$.
    In both cases, $h^\sqcap \subseteq g^\sqsupset$, and so $\sqsupset$ is star-refining.
    Finally, for every $g \in G$ we have $g^\sqsupset = \{g, s_{g, g'}, s_{g', g}: g' \sim_G g\}$, which is connected, and so $\sqsupset$ is monotone by \autoref{MonotoneEdgeSurjective}.
\end{proof}

We conclude the section with a little observation we will need later on.
\begin{lemma}\label{RestrictionsOfMonotone}
Let $S$ and $T$ be trees. If ${\sqsupset}\in\mathbf S_S^T$ is surjective and monotone, $S'\subseteq S$ is connected and $T' \subseteq T$ then ${\sqni} = {\sqsupset}\cap(T' \times S')$ is monotone.
\end{lemma}
\begin{proof}
    For every connected $C \subseteq T'$ we have $C^{\sqni} = C^\sqsupset \cap S'$, which is connected as the intersection of connected subgraphs of a tree.
\end{proof}

\section{Compacta from sequences of graphs} \label{s.InducedPoset}

Our aim is to construct, in the spirit of \autoref{ss:spectrum}, spectra of $\omega$-posets coming from sequences in suitable graph categories.

Let $\mathbf{C}$ be a subcategory of the graph category $\mathbf{G}$.
A \emph{sequence} in $\mathbf{C}$ is a functor from the thin category $\omega$ to $\mathbf{C}$.  More explicitly, this is a collection of morphisms ${\sqsupset}_n^m \in \mathbf{C}^{G_m}_{G_n}$, for $m \leq n \in \omega$, such that each ${\sqsupset}_n^n$ is the identity morphism and 
 \[
 l \leq m \leq n\in\omega\qquad\Rightarrow\qquad{\sqsupset_m^l\circ\sqsupset_n^m} ={ \sqsupset_n^l}.
 \]
When writing ``let $(G_n,\sqsupset_n^m)$ be a sequence in $\mathbf C$'' we mean that the $G_n$'s are objects of $\mathbf C$, ${\sqsupset_n^m}\in\mathbf C_{G_n}^{G_m}$ for all $m\leq n$, and $(\sqsupset_n^m)$ is a sequence.

We say that such a sequence $(\sqsupset^m_n)$ lies in $I
\subseteq \mathbf{C}$ if ${\sqsupset}^m_n \in I$ whenever $m < n$.  Sequences will be named accordingly, e.g. we call $(\sqsupset_n^m)$ an ``anti-injective sequence'' if $\sqsupset_n^m$ is anti-injective, whenever $m<n$.  Note that for $(\sqsupset^m_n)$ to lie in an ideal $I$, it suffices that ${\sqsupset}^m_{m + 1} \in I$, for all $m$.  In particular, a co-injective/co-surjective/edge-surjective sequence $(\sqsupset^m_n)$ is anti-injective/star-refining/edge-witnessing if and only if the morphisms $\sqsupset^m_{m+1}$ are, thanks to \autoref{Ideals}.

Now we come to the key definition of the paper. Recall that $\mathbf S$ is the wide subcategory of $\mathbf{G}$ with co-surjective morphisms (see \autoref{S.GraphCategory}).
\begin{dfn}\label{dfn:inducedposet}
Let $(G_n,\sqsupset_n^m)$ be a sequence in $\mathbf{S}$. We define the \emph{induced poset} $(\PP, \leq)$ of $(G_n,\sqsupset_n^m)$ as the disjoint union $\bigcup_{n \in \omega} G_n \times \{n\}$ with $(p, n) \leq (q, m)$ if and only if $p \sqsubset^m_n q$.
For convenience, we may assume that the sets $G_n$ are pairwise disjoint so we can identify them with the sets $G_n \times \{n\}$. Then we have
 \[
 \tag{Induced Poset}\label{Induced Poset} \mathbb{P}=\bigcup_{n\in\omega}G_n\qquad\text{and}\qquad {\leq}\ = \bigcup_{m\leq n}\sqsubset_n^m.
 \]
\end{dfn}

\begin{prp} \label{prp:InducedPoset}
	Let $(G_n,\sqsupset_n^m)$ be a sequence in $\mathbf{S}$  with induced poset $\PP$.
	Then $\PP$ is an infinite graded $\omega$-poset whose $n$\textsuperscript{th} rank-set $\mathsf{r}^{-1}(\{n\})$ is the graph $G_n$, for every $n\in\omega$.
	Moreover, the following are equivalent.
	\begin{enumerate}
		\item\label{itm:pos_atomless} The poset $\PP$ is atomless.
		\item\label{itm:pos_levels} The $n$\textsuperscript{th} level coincides with the $n$\textsuperscript{th} rank-set, i.e. $\PP_n = G_n$, for every $n \in \omega$.
		\item\label{itm:pos_surjective} The sequence $(\sqsupset^m_n)$ is surjective.
	\end{enumerate}
\end{prp}
\begin{proof}
	As each $G_n$ is non-empty, $\mathbb P$ is infinite.
	Since each $\sqsupset_n^0$ is co-surjective, elements of $G_0$ correspond to maximal elements.
	Fix $p<q$ with $p\in G_n$ and $q\in G_l$. For every $m$ such that $l<m<n$ there is $r\in G_m$ with $p<r<q$ as ${\sqsupset_m^l}\circ{\sqsupset_n^m}={ \sqsupset_n^l}$.
	Induction and co-surjectivity of each $\sqsupset_n^m$ then imply that each element of $G_k$ has rank $k$, and the rank function maps intervals to intervals. This shows that $\mathbb P$ is graded and that $G_n$ is its $n^\mathrm{th}$ rank-set.
	As each $G_n$ is finite, then each principal upset $p^<$ is finite and $\PP$ is an $\omega$-poset.

    We are left to show that the statements \ref{itm:pos_atomless}--\ref{itm:pos_surjective} are equivalent. Suppose \ref{itm:pos_atomless} holds. In general we have $G_n = \mathsf{r}^{-1}(\{n\}) \subseteq \PP_n$, so to show \ref{itm:pos_levels} it is enough to take $p \in \PP$ with $\mathsf{r}(p) < n$ and show that $p \notin \PP_n$. As $\mathbb{P}$ is atomless, we have $q < p$. Since $\PP$ is graded, we have $q'$ with $q \leq q' < p$ and $\mathsf{r}(q') = \mathsf{r}(p) + 1 \leq n$.  This shows that $p$ is not minimal in $\mathbb{P}^n$ and hence $p \notin \PP_n$.
	
	Now suppose \ref{itm:pos_levels} holds.  For every $p \in G_n = \PP_n$, we have $q \leq p$ such that $q \in \PP_{n + 1} = G_{n + 1}$, and so $\sqsupset^n_{n + 1}$ is surjective.
	Since composing surjective relations give surjective relations, each $\sqsupset_n^m$ is surjective.
	
	Finally suppose \ref{itm:pos_surjective}.
	For every $p \in \PP$, we have $n$ with $p \in G_n$ and then we have $q \in G_{n + 1}$ with $p \sqsupset^n_{n + 1} q$ so $p > q$ and hence $\PP$ is atomless.
\end{proof}

On the other hand, given an infinite graded $\omega$-poset $\PP$ we may put
\[
\tag{Induced Sequence}\label{Induced Sequence} G_n = \mathsf{r}^{-1}(\{n\}) \quad\text{and}\quad {\sqsupset}^m_n\ =\ {\geq} \cap {(G_m \times G_n)},
\]
for every $m \leq n \in \omega$.
Since $\PP$ is infinite, every set $G_n$ is non-empty, and from the definition of the rank, the relations $\sqsupset^m_n$ are co-surjective.

Note that the definition of the induced poset and the previous proposition do not use the edge relations of the graphs $G_n$, and so they apply when we are concerned with co-surjective relations between sets.
Also, when sets are viewed as graphs with no edges, every relation is edge-preserving.
It is then easy to see that \eqref{Induced Poset} and \eqref{Induced Sequence} are mutually inverse constructions yielding a one-to-one correspondence between infinite graded $\omega$-posets and sequences of co-surjective relations between non-empty finite sets.

\subsection{Threads and limits}\label{ss.threads}

In this short digression, we relate our notion of spectrum to the standard concept of inverse limits, using threads. We also sketch a comparison with similar approaches as defined in \cite{IrwinSolecki2006} and~\cite{DebskiTymchatyn2018}.

By the \emph{spectrum} of a sequence $(G_n, \sqsupset^m_n)$ in $\mathbf S$ we mean the spectrum $\Spec{\PP}$ of its induced poset $\PP$, i.e. the family of all minimal selectors in $\PP$.
Recall that $\Spec{\PP}$ is a second-countable compact $\mathsf T_1$ space (\autoref{SpectrumCompactT1}) which we then view as a kind of topological limit of the sequence of graphs.

By a \emph{thread} we mean a sequence of points $(t_n) \in \prod_{n \in \omega} G_n$ such that $t_0 > t_1 > t_2 > \cdots$ in $\PP$, i.e. we have $t_m \sqsupset^m_n t_n$ for every $m \leq n$.
We shall also view a thread as a subset of $\PP$ whenever convenient.


Recall that the standard inverse limit of a sequence $(X_n)$ of compact Hausdorff spaces and continuous bonding maps $f_n\maps X_{n+1} \to X_{n}$ is the space of all threads $\{(x_n) \in \prod_n X_n: \forall n\ x_n = f_n(x_{n + 1})\}$, which is a closed subspace of the product.
In the spectrum, threads play an important role as well.

\begin{lemma} \label{Thread}
	Let $(G_n, \sqsupset^m_n)$ be a surjective sequence in $\mathbf{S}$ with induced poset $\PP$, and let $S$ be a minimal selector.
	Then there is a thread $(s_n)$ such that $S = (s_n)^\leq$.
\end{lemma}
\begin{proof}
By \cite[Proposition 2.13]{BartosBiceV.Compacta}, $S$ is a filter. Hence we can build $(t_k)$, a strictly decreasing sequence in $S$, such that $(t_k)^{\leq}=S$. We can thus fill the gaps between $t_{k+1}$ and $t_{k}$ with linearly ordered elements so that every level is intersected. The resulting sequence $(s_n)$ will be as required.
\end{proof}

Let us compare our approach to constructing spaces from sequences of graphs to a more standard approach using quotients of inverse limits, as used by Irwin and Solecki~\cite{IrwinSolecki2006} in the context of projective Fra\"{i}ss\'e theory and by D\polhk{e}bski and Tymchatyn~\cite{DebskiTymchatyn2018} in the context of constructing topologically complete spaces.
We shall also demonstrate a possible closer connection on a simple example.

In the inverse limit approach, we start with a sequence of graphs $(G_n)$ and bonding maps $f_n\maps G_n \leftarrow G_{n + 1}$ that are quotient homomorphisms (i.e. edge-preserving edge-surjective functions).
We consider the inverse limit graph
\[\textstyle G_\infty = \{(x_n) \in \prod_{n \in \omega} G_n: \forall n\ x_n = f_n(x_{n + 1})\}\]
with the edge relation $(x_n) \sqcap (y_n) \Leftrightarrow \forall n\ x_n \sqcap y_n$.
Moreover, we endow $G_\infty$ with the subspace topology of the product of the finite discrete spaces.
Hence, $G_\infty$ becomes a zero-dimensional metrizable compact space, typically the Cantor space, and $\sqcap$ becomes a closed relation.
Note that $\sqcap$ is always reflexive and symmetric. If it turns out to be also transitive, $G_\infty$ is called a \emph{pre-space}, and we obtain the desired compact space as the quotient $G_\infty/{\sqcap}$.

In our approach we allow the bonding maps to be relations.
Namely, we start with an edge-faithful (see \autoref{ss.faithful}) sequence $(G_n, \sqsupset^m_n)$ in $\mathbf B$.
Then we consider the induced poset $\PP$ and we obtain our space as the spectrum $\Spec{\PP}$.
Note $\PP$ is well-defined also for a sequence of functions and the definition of $G_\infty$ may be extended to cover sequences of relations as well, though it will not be a categorical limit any more.

The elements $p \in \PP$ correspond to basic clopen subsets of $G_\infty$, which become closed sets in $G_\infty/{\sqcap}$, but they also correspond to basic open subsets of $\Spec{\PP}$.
Elements of a single graph $G_n$ give a clopen partition of $G_\infty$, a ``touching'' closed cover of $G_\infty/{\sqcap}$, and an overlapping open cover of $\Spec{\PP}$. In both cases the structure of touching or overlapping is given by the edge relation of the graph.

\begin{xpl} \label{ex:arc}
    Let $G_n$ be a path of $2^n$ vertices, and let $f_n\maps G_n \leftarrow G_{n + 1}$ be a monotone edge-preserving surjection splitting every vertex of $G_n$ into two vertices of $G_{n + 1}$.
    Then the induced poset $\PP$ is the Cantor tree, the limit graph $G_\infty$ is the Cantor space with the edge relation being the equivalence gluing gaps in the middle-third interval representation of the Cantor space.
    In other words, we obtain a canonical construction of the arc as $G_\infty/{\sqcap}$.
    If we imagine the arc as the unit interval $[0, 1]$, the elements of $\PP$ correspond to closed subintervals of the form $[k/2^n, (k+1)/2^n]$.

    Of course, $\Spec{\PP}$ is just the Cantor space. The poset $\PP$ does not see the edges of the graphs $G_n$.
    However, we may consider a closely related sequence of graphs and bonding relations.
    For every $n \in \omega$ let $H_n$ be the graph of all maximal cliques in $G_n$, with the edge relation on $H_n$ being the overlapping of cliques.
    Hence, for $n \geq 1$, $H_n$ is the path of $2^n - 1$ consecutive edges of $G_n$, while $H_0$ is a singleton.
    We put $h \sqsubset^n_{n + 1} h'$ if $f\im{h} \subseteq h'$.
    This induces a sequence $(\sqsupset^m_n)$ and a poset $\QQ$.
    The induced poset $\QQ$ is (ignoring the first level coming from $H_0$) exactly the poset from \cite[Example~2.16]{BartosBiceV.Compacta} whose spectrum is again the arc.
    The sequence $(\sqsupset^m_n)$ is edge-preserving, co-bijective, and edge-witnessing.
    We can identify the elements of $H_n$ with open intervals $[0, 2/2^n), \ldots, (k/2^n, (k+2)/2^n), \ldots, (1 - 2/2^n, 1] \subseteq [0, 1]$, obtained by taking the interiors of unions of pairs of consecutive closed intervals corresponding to the elements of $G_n$.
\end{xpl}
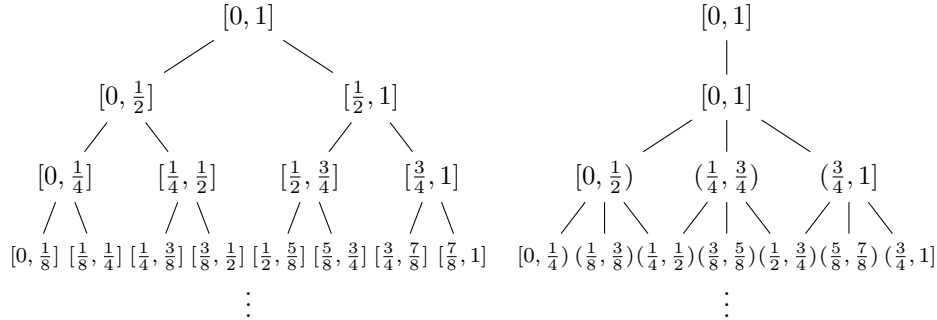
\begin{figure}[!ht]
    \centering
    \begin{tikzpicture}[x={(0, -3em)}, y={(4.6em, 0)}]
        \begin{scope}
                \node (nil) at (0, 0) {$[0, 1]$};
                \node (0) at (1, -1) {$[0, \frac{1}{2}]$};
                \node (1) at (1, 1) {$[\frac{1}{2}, 1]$};
                \node (00) at (2, -1.5) {$[0, \frac{1}{4}]$};
                \node (01) at (2, -0.5) {$[\frac{1}{4}, \frac{1}{2}]$};
                \node (10) at (2, 0.5) {$[\frac{1}{2}, \frac{3}{4}]$};
                \node (11) at (2, 1.5) {$[\frac{3}{4}, 1]$};
            \begin{scope}[font=\footnotesize]
                \node (000) at (3, -1.75) {$[0, \frac{1}{8}]$};
                \node (001) at (3, -1.25) {$[\frac{1}{8}, \frac{1}{4}]$};
                \node (010) at (3, -0.75) {$[\frac{1}{4}, \frac{3}{8}]$};
                \node (011) at (3, -0.25) {$[\frac{3}{8}, \frac{1}{2}]$};
                \node (100) at (3, 0.25) {$[\frac{1}{2}, \frac{5}{8}]$};
                \node (101) at (3, 0.75) {$[\frac{5}{8}, \frac{3}{4}]$};
                \node (110) at (3, 1.25) {$[\frac{3}{4}, \frac{7}{8}]$};
                \node (111) at (3, 1.75) {$[\frac{7}{8}, 1]$};
            \end{scope}
            \node at (3.5, 0) {$\vdots$};
            
            \graph{
                (nil) -- {(0), (1)},
                (0) -- {(00), (01)},
                (1) -- {(10), (11)},
                (00) -- {(000), (001)},
                (01) -- {(010), (011)},
                (10) -- {(100), (101)},
                (11) -- {(110), (111)},
            };
        \end{scope}

        \begin{scope}[xshift=18em]
                \node (prenil) at (0, 0) {$[0, 1]$};
                \node (nil) at (1, 0) {$[0, 1]$};
                \node (L) at (2, -1) {$[0, \frac12)$};
                \node (C) at (2, 0) {$(\frac14, \frac34)$};
                \node (R) at (2, 1) {$(\frac34, 1]$};
            \begin{scope}[font=\footnotesize]
                \node (LL) at (3, -1.5) {$[0, \frac14)$};
                \node (LC) at (3, -1) {$(\frac18, \frac38)$};
                \node (CL) at (3, -0.5) {$(\frac14, \frac12)$};
                \node (CC) at (3, 0) {$(\frac38, \frac58)$};
                \node (CR) at (3, 0.5) {$(\frac12, \frac34)$};
                \node (RC) at (3, 1) {$(\frac58, \frac78)$};
                \node (RR) at (3, 1.5) {$(\frac34, 1]$};
            \end{scope}
            \node at (3.5, 0) {$\vdots$};
            
            \graph{
                (prenil) -- (nil) -- {(L), (C), (R)},
                (L) -- {(LL), (LC), (CL)},
                (C) -- {(CL), (CC), (CR)},
                (R) -- {(CR), (RC), (RR)},
            };
        \end{scope}
    \end{tikzpicture}
    
    \caption{Obtaining the unit interval using the two approaches.}
    \label{fig:comparison}
\end{figure}

\subsection{Subsequences and restrictions}\label{ss.subsequences}

A \emph{subsequence} of a sequence $(G_n, \sqsupset^m_n)$ in $\mathbf G$ is a sequence of the form $(G_{\phi(n)}, \sqsupset^{\phi(m)}_{\phi(n)})$ for some strictly increasing  $\phi\maps \omega \to \omega$.

\begin{prp} \label{Subsequence}
	Let $(G_n,\sqsupset_n^m)$ be a surjective sequence in $\mathbf{S}$ with induced poset $\PP$, and let $(\sqsupset^{\phi(m)}_{\phi(n)})$ be a subsequence with induced poset $\QQ$.
	Then $\QQ$ is the subposet $\bigcup_n \PP_{\phi(n)} \subseteq \PP$, and $\Spec{\QQ}$ and $\Spec{\PP}$ are homeomorphic via 
	\[
		\Spec{\QQ} \ni S \mapsto S^{\leq_\PP} \qquad\text{and}\qquad \Spec{\PP} \ni T \mapsto T \cap \QQ.
	\]
\end{prp}
\begin{proof}
	Since $(\sqsupset^m_n)$ is surjective, we have $\PP_n = G_n$ and $\QQ_n = G_{\phi(n)} = \PP_{\phi(n)}$ for every $n \in \omega$, and ${\leq}_\QQ = {\leq}_\PP \restriction \QQ \times \QQ$.
	Since $\{\PP_{\phi(n)}: n \in \omega\}$ is a co-initial family of bands both in $\PP$ and $\QQ$, that the above maps are homeomorphisms follows from \cite[Corollary~3.8]{BartosBiceV.Compacta}, as well as from \autoref{DenseRestrictionHomeomorphic}.
\end{proof}

It is useful to keep in mind the following lemma when discussing properties like ``$(\sqsupset^m_n)$ has a star-refining and anti-injective subsequence''.
\begin{lemma} \label{IdealSubsequences}
	Let $(G_n, \sqsupset^m_n)$ be a sequence in a category $\mathbf C \subseteq \mathbf G$ and let $I \subseteq \mathbf C$ be an ideal.
	The following are equivalent.
	\begin{enumerate}
		\item \label{itm:subsequence}
			$(\sqsupset^m_n)$ has a subsequence in $I$.
		\item \label{itm:wide}
			For every $m$ there is $n \geq m$ such that ${\sqsupset}^m_n \in I$.
		\item \label{itm:subsubsequence}
			Every subsequence $(\sqsupset^{\phi(m)}_{\phi(n)})$ of $(\sqsupset^m_n)$ has a subsequence in $I$.
	\end{enumerate}
	Hence, for another ideal $J \subseteq \mathbf C$, we have that $(\sqsupset^m_n)$ has a subsequence in $I$ and a subsequence in $J$ if and only if it has a subsequence in $I \cap J$.
\end{lemma}
\begin{proof}
	Suppose \ref{itm:subsequence} and fix $m\in\omega$. Since $(\sqsupset^{m}_{n})$ has a subsequence in $I$ we can find $k$ and $k'$ such that $m \leq k<k'$ and  ${\sqsupset}^{k}_{k'} \in I$. Since $I$ is an ideal then ${\sqsupset}^m_{k'} \in I$, showing \ref{itm:wide}.
	Now suppose \ref{itm:wide} and let $(\sqsupset^{\phi(m)}_{\phi(n)})$ be a subsequence.
	We inductively define an increasing map $\psi\maps \omega \to \omega$ such that the subsequence $(\sqsupset^{\phi(\psi(m))}_{\phi(\psi(n))})$ is in $I$.
    First we set $\psi(0) = 0$.
	Now suppose we have already defined $\psi(k)$.
	By \ref{itm:wide}, we have $n \geq \phi(\psi(k))$ such that ${\sqsupset}^{\phi(\psi(k))}_n \in I$ and then we have $\psi(k + 1) > \psi(k)$ such that $\phi(\psi(k + 1)) \geq n$.
	Then ${\sqsupset}^{\phi(\psi(k))}_{\phi(\psi(k + 1))} \in I$, since $I$ is an ideal.
	This concludes the construction and proves \ref{itm:subsubsequence}.
		The implication \ref{itm:subsubsequence}$\Rightarrow$\ref{itm:subsequence} is trivial, and to obtain a subsequence in $I \cap J$ we just apply \ref{itm:subsubsequence} to a subsequence in $J$.
\end{proof}

Next we focus on another form of a ``subsequence'', related to subspaces of the spectrum.
\begin{dfn}
	A \emph{restriction} of a sequence $(G_n, \sqsupset^m_n)$ in $\mathbf G$ is a sequence $(H_n, \sqni^m_n)$ such that every $H_n \subseteq G_n$ is an (induced) subgraph and ${\sqni}^m_n = {\sqsupset}^m_n \restriction (H_m \times H_n)$ for every $m \leq n$.
	An \emph{upper restriction} is a restriction such that $H_n{}^{\sqsubset^m_n} \subseteq H_m$ for every $m \leq n$.
\end{dfn}

\begin{rmk}
	Note that not every sequence of subgraphs $H_n \subseteq G_n$ induces a restriction by putting ${\sqni}^m_n = {\sqsupset}^m_n \restriction (H_m \times H_n)$ as $(H_n, \sqni^m_n)$ may not be a sequence, only a \emph{lax-sequence}, i.e. it satisfies ${\sqni}^l_m \circ {\sqni}^m_n \subseteq {\sqni}^l_n$ for $l \leq m \leq n$.
    (Consider, for example, every $\sqsupset^m_n$ being the identity on a two-point discrete graph and $H_n$ omitting one vertex if and only if $n$ is odd.)
	We shall call such $(H_n, \sqni^m_n)$ a \emph{lax-restriction} of $(G_n, \sqsupset^m_n)$.
	
	On the other hand, every sequence of subgraphs such that $H_n{}^{\sqsubset^m_n} \subseteq H_m$ for every $m \leq n$ induces an upper restriction $(H_n, \sqni^m_n)$ by putting ${\sqni}^m_n = {\sqsupset}^m_n \restriction (H_m \times H_n)$.
\end{rmk}

Upper restrictions are often induced by \emph{traces}.

\begin{dfn}
    Let $(G_n, \sqsupset^m_n)$ be a  sequence in $\mathbf{S}$ and let $\PP$ be the induced poset.
    For every $A \subseteq \Spec{\PP}$ and $n \in \omega$ let $[A]_n$ denote the \emph{trace} $\{g \in G_n: g^\in \cap A \neq \emptyset\}$.
    In particular, for $S \in \Spec{\PP}$ we have $[\{S\}]_n = G_n \cap S$, and so $[A]_n = G_n \cap \bigcup A$.
\end{dfn}

Clearly we have $[A]_n^{\sqsubset^m_n} \subseteq [A]_m$ for every $m \leq n$.
If the sequence $(\sqsupset^m_n)$ is surjective, then every trace $[A]_n$ is non-empty since the levels of $\PP$ correspond to the graphs $G_n$ by \autoref{prp:InducedPoset}.
Moreover, every family $\{g^\in: g \in [A]_n\}$ is an open cover of $A$ in $\Spec{\PP}$.

\begin{prp} \label{Subspaces}
    Let $(G_n, \sqsupset^m_n)$ be a surjective sequence in $\mathbf{S}$ with induced poset $\PP$.
    \begin{enumerate}
        \item \label{itm:sublimit}
            For $(H_n, \sqni^m_n)$ a surjective upper restriction of $(G_n, \sqsupset^m_n)$ in $\mathbf{S}$ with induced poset $\QQ$, we have that $\Spec{\QQ}$ is the closed subspace $\{S \in \Spec{\PP}: S \subseteq \QQ\} \subseteq \Spec{\PP}$ and that $[\Spec{\QQ}]_n \subseteq H_n$ for every $n$.
        \item \label{itm:traces}
            For a subspace $A \subseteq \Spec{\PP}$, we have $[A]_n^{\sqsubset^m_n} = [A]_m$ for every $m \leq n \in \omega$, i.e. $([A]_n)$ induces a surjective upper restriction of $(G_n, \sqsupset^m_n)$, and we have $\Spec{\QQ} = \overline{A}$, where $\QQ$ is the poset induced by $([A]_n)$.
    \end{enumerate}
\end{prp}
\begin{proof}
    We show \ref{itm:sublimit}.
    From the definition of an upper restriction we have that $\QQ$ is an upwards closed subset of $\PP$.
    Since the sequences $(\sqsupset^m_n)$ and $(\sqni^m_n)$ are surjective, the posets $\PP$ and $\QQ$ are atomless, and so $\PP_n = G_n$ and $\QQ_n = H_n = G_n \cap \QQ$.
    It follows that upwards closed selectors in $\QQ$ are precisely upwards closed selectors in $\PP$ that are contained in $\QQ$, and so the same is true for minimal selectors.
		Clearly, $\{S \in \Spec{\PP}: S \subseteq \QQ\}$ is closed.

    To show \ref{itm:traces} let $A \subseteq \Spec{\PP}$ and $m \leq n \in \omega$.
    It is easy to see that $[A]_n^{\sqsubset^m_n} \subseteq [A]_m$.
    On the other hand, for $g \in [A]_m$ there is $S \in  \Spec{\PP}$ with $g \in S$.
    Let $(s_k)$ be a thread such that $(s_k)^\leq = S$, obtained by \autoref{Thread}.
		Note that $s_k \in [A]_k$ for every $k$.
    There is $k \geq n$ such that $g > s_k$, i.e. $g \sqsupset^m_k s_k$.
    From the definition of a sequence there is $h \in G_n$ with $g \sqsupset^m_n h \sqsupset^n_k s_k$.
    Necessarily $h \in S$, and so $h \in [A]_n$.
    This shows $[A]_n^{\sqsubset^m_n} = [A]_m$.
    From \ref{itm:sublimit} it follows that $\Spec{\QQ} = \{S \in \Spec{\PP}: [\{S\}]_n \subseteq [A]_n\text{ for every $n$}\}$.
    Hence, $A \subseteq \Spec{\QQ}$ and also $\overline{A} \subseteq \Spec{\QQ}$ as $\Spec{\QQ}$ is closed.
    On the other hand, for if for $S \in \Spec{\PP}$ we have that $[\{S\}]_n \subseteq [A]_n$ for every $n \in \omega$, then for every $p \in S$ there is $S' \in A \cap p^\in$, and so $S \in \overline{A}$.
\end{proof}

\begin{rmk}
    The previous proposition gives a correspondence between closed subspaces of the spectrum $\Spec{\PP}$ and surjective upper restrictions of $(G_n, \sqsupset^m_n)$, analogous to \cite[\S 2.4]{BartosBiceV.Compacta}: every closed $A \subseteq \Spec{\PP}$ gives the restriction induced by $([A]_n)$, whose spectrum is again $A$.
    On the other hand, if we start with a surjective upper restriction $(H_n, \sqni^m_n)$ and take the corresponding closed subspace $A = \Spec{\QQ}$, we have only $[A]_n \subseteq H_n$ by \autoref{Subspaces}~\ref{itm:sublimit}.
    The equality would correspond to the property that every $h \in H_n$ is an element of a minimal selector contained in $\QQ$, i.e. when $\QQ \subseteq \PP$ is a \emph{prime subset} \cite[\S 2.4]{BartosBiceV.Compacta}, or in this case equivalently when $\QQ$ is a prime poset.

    Hence, by a \emph{prime restriction} of $(G_n, \sqsupset^m_n)$ we mean a surjective upper restriction $(H_n, \sqni^m_n)$ such that its induced poset $\QQ$ is prime, or equivalently such that it is induced by $([A]_n)$ for some $A \subseteq \Spec{\PP}$.
    Note that if an upper restriction is in $\mathbf{B}$, then it is prime as every predetermined $\omega$-poset is prime (\autoref{figure1}).
    Altogether, we have a one-to-one correspondence between closed subspaces of the spectrum and prime restrictions, analogously to \cite[Corollary~2.25]{BartosBiceV.Compacta}.
\end{rmk}

Even if our surjective restriction is not upper, it is sometimes still possible to obtain a canonical embedding between the respective spectra. 
\begin{dfn}
	Let $(G_n, \sqsupset^m_n)$ be a surjective sequence in $\mathbf{S}$.
	We call a subset $H_n \subseteq G_n$ \emph{dense} if ${H_n}^{\sqsupset^n_k} = G_k$ for some $k \geq n$.
    Equivalently, if $H_n$ is a cap contained in $G_n$.
    
    A restriction $(H_n,\sqni_n^m)$ of $(G_n, \sqsupset^m_n)$ is dense if each $H_n$ is dense.
\end{dfn}
Note that if $(g_n)$ is a thread and we have $m\in\omega$ with $g_n{}^{\sqsubset_n^m}=\{g_m\}$, for all $n\geq m$, then $g_m$ belongs to all dense subsets of $G_m$.
Also note that if $H_n \subseteq G_n$ is dense, then $H_n^\sqcap = G_n$.

\begin{prp} \label{DenseRestrictionHomeomorphic}
    Let $(G_n, \sqsupset^m_n)$ be a surjective sequence in $\mathbf{S}$ with induced poset $\PP$, and let $(H_n, \sqni^m_n)$ be a surjective dense restriction of $(G_n, \sqsupset^m_n)$ with induced poset $\QQ$.
    Then $S \mapsto S^{\leq_\PP}$ is a homeomorphism $\Spec{\QQ} \to \Spec{\PP}$ with inverse $T \mapsto T \cap \QQ$.
\end{prp}
\begin{proof}
    By \autoref{prp:InducedPoset}, $\PP$ and $\QQ$ are atomless and their levels correspond to the $G_n$ and $H_n$, respectively.
    We want to use Lemmas~\ref{UpwardsClosureHypothesis} and~\ref{UpwardsClosureHomeomorphism}.
    To do so we need to show that for every $S \in \Spec{\QQ}$ and $T \subseteq S^{\leq_\PP}$ a minimal selector in $\PP$ we have that $T \cap \QQ$ is infinite.
    In fact, suppose $T$ is any upwards closed selector in $\PP$.
    We show that $T \cap \QQ$ is infinite.

    Fix $m \in \omega$.
    Since $H_m$ is dense, there is $n \geq m$ such that $H_m{}^{\sqsupset^m_n} = G_n$.
    Since $T$ is a selector, there is $t \in T \cap G_n$.
    By the choice of $n$, there is $h \in t^{\sqsubset^m_n} \cap H_m$.
    Since $T$ is upwards closed, $h \in (T \cap \QQ) \cap G_m$.
    As $m$ was arbitrary, $T \cap \QQ$ is infinite.
\end{proof}

\subsection{Faithful representations} \label{ss.faithful}

Note that so far there was no connection between edge relations of the graphs $G_n$ and the induced poset $\PP$.
If the edge relations are to properly detect overlaps of the corresponding open sets, then they should agree with the common lower bound relation. 
Recall that in a poset $\PP$ the relation $\wedge$ denotes the existence of a lower bound i.e. $p\wedge q$ if and only if $p,q\geq r$, for some $r\in\mathbb{P}$.
We call a sequence $(G_n,\sqsupset^m_n)$ \emph{edge-faithful} if, for every $n \in \omega$ and every $g, h \in G_n$, we have $g \sqcap h$ if and only if $g \wedge h$ in $\PP$.
Note that if every $G_n$ is just a set then there is a unique way to define the edge relations so that $(\sqsupset^m_n)$ becomes edge-faithful, specifically by setting ${\sqcap_n}=(G_n\times G
_n)\cap{\wedge}$.

We summarize the connection between the edge relation and the actual overlaps, and we characterize edge-faithful sequences.

\begin{lemma} \label{Overlapping}
    Let $(G_n, \sqsupset^m_n)$ be a  sequence in $\mathbf{S}$ and let $\PP$ be the induced poset.
    For every $m \in \omega$ and $g, h \in G_m$ we have
    \[
        g^\in \cap h^\in \neq \emptyset \quad\Rightarrow\quad g \wedge h \text{ (in $\PP$)} \quad\Rightarrow\quad g \sqcap h \text{ (in $G_m$).}
    \]
    If $(\sqsupset^m_n)$ is co-bijective and edge-faithful, then the implications become equivalences.
\end{lemma}
\begin{proof}
    Suppose $S \in g^\in \cap h^\in$, which means $g, h \in S$.
    As the minimal selector $S$ is a filter (\cite[Proposition 2.13]{BartosBiceV.Compacta}), there is $p \in S$ with $p \leq g, h$, and so $g \wedge h$.
    	
		Edge-faithfulness is by definition the equivalence $g \wedge h \Leftrightarrow g \sqcap h$.
    But one implication is true for any sequence in $\mathbf{S}$.
		Suppose $g \wedge h$.
		Then there is $n \geq m$ and $p_n \in G_n$ such that $g \sqsupset^m_n p_n \sqsubset^m_n h$, and hence $g \sqcap h$ since $\sqsupset^m_n$ is edge-preserving.
		Moreover, if $(\sqsupset^m_n)$ is co-bijective, there is $p_{n + 1} \in G_{n + 1}$ with $p_{n + 1}{}^{\sqsubset^n_{n + 1}} = \{p_n\}$, and by induction there is a thread $(p_i)$ such that $(p_i)^\leq \cap G_k = \{p_k\}$ for every $k \geq n$.
		Hence, $(p_k)^\leq$ is a minimal selector in $g^\in \cap h^\in$.
\end{proof}

\begin{prp} \label{OverlappingCoherence}
Let $G_n$ be a sequence of finite graphs, and let ${\sqsupset^m_n} \subseteq G_m\times G_n$, $m \leq n$, be a surjective co-surjective sequence (not necessarily edge-preserving). Let $\PP$ be the induced poset.
 The following are equivalent.
 \begin{enumerate}
 \item \label{OCdef} $(\sqsupset^m_n)$ is edge-faithful.
 \item \label{OCseq2} $(\sqsupset^m_n)$ is edge-preserving and for every $m \in \omega$ there is $n > m$ such that $\sqsupset^m_n$ is edge-witnessing.
 \item \label{OCseq}
 $(\sqsupset^m_n)$ is edge-preserving and edge-surjective and has an edge-witnessing subsequence.
 \end{enumerate}
\end{prp}
\begin{proof}
 Suppose \ref{OCdef}.
 Let $h \sqsupset^m_n g \mathrel{\sqcap} g' \sqsubset^m_n h'$. We have $g \mathrel{\wedge} g'$, so there is $p \in \PP$ such that $h \geq g \geq p \leq g' \leq h'$, and hence $h \wedge h'$ and $h \sqcap h'$. We have shown that $\sqsupset^m_n$ is edge-preserving.
 Now let $m \in \omega$.
 For every edge $g \sqcap g'$ in $G_m$ there is $p_{g, g'} \leq g, g'$.
 Since $(\sqsupset^m_n)$ is surjective and so $\PP$ is atomless, we can suppose $p_{g, g'} \in G_{n_{g, g'}}$ for some $n_{g, g'} > m$ even when $g = g'$.
 Again by the surjectivity, for $n \geq \max\{n_{g, g'}: g \sqcap g' \in G_n\} > m$ we have that $\sqsupset^m_n$ is edge-witnessing, giving \ref{OCseq2}.

 Next suppose \ref{OCseq2}.
 Clearly, we can inductively define and edge-witnessing subsequence.
 Also for every $m \in \omega$ and $g \sqcap g'$ in $G_m$ there is $n > m$ with edge-witnessing ${\sqsupset}^m_n = {\sqsupset}^m_{m + 1} \circ {\sqsupset}^{m + 1}_n$, and so there are $h$, $h'$ and $p$ with $g \sqsupset^m_{m + 1} h \sqsupset^{m + 1}_n p \sqsubset^{m + 1}_n h' \sqsubset^m_{m + 1} g'$.
 As $(\sqsupset^m_n)$ is edge-preserving, we have that $h \sqcap h'$ and $(\sqsupset^m_n)$ is edge-surjective, which proves \ref{OCseq}.
 
 Finally suppose \ref{OCseq}.
 If $g \sqcap g'$ in $G_m$, then there are $n' > n \geq m$ such that $\sqsupset^n_{n'}$ is edge-witnessing.
 Since $\sqsupset^m_n$ is edge-surjective, there are $h, h' \in G_m$ with $g \sqsupset^m_n h\mathrel{\sqcap}h' \sqsubset^m_n g'$.
 Since $\sqsupset^n_{n'}$ is edge-witnessing, there is $p \in G_{n'}$ with $g \geq h \geq p \leq h' \leq g'$, and so $g \wedge g'$.
 The other half of edge-faithfulness follows from \autoref{Overlapping}.
\end{proof}

\begin{cor}
An edge-surjective sequence in $\mathbf B$ is edge-faithful if and only if it has an edge-witnessing subsequence.
    \qed
\end{cor}

We want to make the point that co-bijective edge-faithful sequences are the right notion to reflect the topological properties of the spectrum.
This is demonstrated by the following proposition about metrizability as well as in \autoref{ConnectedLimit} about connectedness.

\begin{prp} \label{HausdorffLimit}
	Let $(G_n, \sqsupset_n^m)$ be an edge-faithful sequence in $\mathbf{B}$ with induced poset $\PP$.
	The following are equivalent.
	\begin{enumerate}
		\item\label{Hauslimit1} The sequence $(\sqsupset_n^m)$ has a star-refining subsequence.
		\item\label{Hauslimit2} The poset $\PP$ is regular.
		\item The spectrum $\Spec{\PP}$ is Hausdorff.
		\item The spectrum $\Spec{\PP}$ is metrizable.
	\end{enumerate}
\end{prp}
\begin{proof}
	First, notice that for second-countable compact spaces being Hausdorff and being metrizable coincide. Further, this is equivalent to $\mathbb P$ being regular (as $\mathbb P$ is a prime $\omega$-poset), by \autoref{RegularImpliesHausdorff}. We are left to show that \ref{Hauslimit1} and \ref{Hauslimit2} are equivalent. Notice that $\sqsupset_n^m$ is star-refining if and only if $G_n\vartriangleleft_{G_n} G_m$, where $\vartriangleleft$ is the star-below relation (see the end of \autoref{ss:thespectrum}).
	
	Suppose $\PP$ is regular and let $m\in\omega$. By regularity, we know that $G_m$ is $\vartriangleleft$-refined by another cap, and since levels are coinitial for refinement (\autoref{prp:coverslevels}) and by the compatibility properties of the relations $\vartriangleleft$ and $\leq$ (see \eqref{Compatibility}), there is $n$ such that $G_n\vartriangleleft G_m$. Since levels are finite and by \eqref{RefiningStars}, we can assume that there is $k>n$ such that $G_n\vartriangleleft_{G_k}G_m$. Using again \eqref{Compatibility}, we have that $G_k\vartriangleleft_{G_k}G_m$, hence $\sqsupset_k^m$ is star-refining. 
	
	Vice versa, suppose $(\sqsupset_n^m)$ has a subsequence of star-refining morphisms. As each cap is refined by a level, and $G_n\vartriangleleft_{G_n}G_m$ for all indexes belonging to the appropriate subsequence, \eqref{Compatibility} implies that every cap is $\vartriangleleft$-refined by a level, and therefore $\PP$ is regular.
\end{proof}

Before we continue, let us summarize the relationship between the properties of (sub)sequences of a sequence $(\sqsupset_n^m)$, the induced poset $\PP$, and its spectrum $\Spec{\PP}$.
Recall that a topological space is said to be \emph{perfect} if it has no isolated points.

\begin{prp} \label{prp:Summary1}
	Let $(G_n, \sqsupset_n^m)$ be a sequence in $\mathbf{S}$  with induced poset $\PP$.
	Then we have the following equivalences.
	\newcommand{\Iff}{$\, \Leftrightarrow\ $}
	\begin{enumerate} 
		\item\label{CaseAtomless} $({\sqsupset}_n^m)$ is surjective \Iff $\PP$ is atomless.
		\item\label{CaseCoinjective}
		$(\sqsupset_n^m)$ is co-injective \Iff $\PP$ is predetermined and atomless.
	\end{enumerate}
	For the following suppose that $(\sqsupset_n^m)$ is in $\mathbf{B}$. Then
	\begin{enumerate}[resume]
		\item\label{CaseAntiInj1} $(\sqsupset_n^m)$ has an anti-injective subsequence \Iff $\Spec{\PP}$ is perfect, and
		\item\label{CaseAntiInj2} $(\sqsupset_n^m)$ is anti-injective \Iff $\PP$ is branching \Iff $\PP$ is cap-determined.
	\end{enumerate}
	Let us additionally suppose that $(\sqsupset^m_n)$ is edge-faithful. Then
	\begin{enumerate}[resume]
		\item\label{CaseStarSurj} $(\sqsupset_n^m)$ has a star-refining subsequence \Iff $\PP$ is regular \Iff $\Spec{\PP}$ is Hausdorff \Iff $\Spec{\PP}$ is metrizable.
	\end{enumerate}
\end{prp}
\begin{proof}
	\ref{CaseAtomless} and \ref{CaseStarSurj} were already proved in Propositions \ref{prp:InducedPoset} and \ref{HausdorffLimit}, respectively.
	We now show \ref{CaseCoinjective}--\ref{CaseAntiInj2}. 

	\ref{CaseCoinjective}: If $\sqsupset_n^m$ is co-injective, it is in particular surjective, so $\mathbb P$ is atomless. Further, fix $p\in G_n$. As $\sqsupset_{n+1}^{n}$ is co-injective, there is $q\in G_{n+1}$ such that $q^{\leq}\cap G_n=\{p\}$. This witnesses that $\mathbb P$ is predetermined. Vice versa, notice that if each $\sqsupset^n_{n+1}$ is co-injective, then each $\sqsupset_n^m$ is such. Fix $p\in G_n$. Since $\mathbb P$ is atomless, $p$ is not minimal. Since $\mathbb P$ is predetermined, there is $q\in \mathbb P$ such that $q^{<}=p^{\leq}$.
	Since $\PP$ is graded, $q$ has rank $n+1$, and so $q\in G_{n+1}$ and $q^{\sqsubset^n_{n+1}}=\{p\}$. This shows that $\sqsupset^n_{n+1}$ is co-injective.
	
	\ref{CaseAntiInj1}:
	For $\Rightarrow$, let $p^\in$ for $p \in G_l$ be a basic open subset of $\Spec{\PP}$.
	Since $(\sqsupset^m_n)$ has an anti-injective subsequence, there are $n > m \geq l$ such that $\sqsupset^m_n$ is anti-injective, and so there are $q \neq r \in G_n$ with $q, r < p$.
	Since $(\sqsupset^m_n)$ is co-injective, as in the proof of \autoref{Overlapping}, there are threads $(q_i)$ and $(r_i)$ with $q_n = q$ and $r_n = r$ such that $(q_i) \cap G_k = \{q_k\}$ and $(r_i) \cap G_k = \{r_k\}$ for $k \geq n$, and so $(q_i)^\leq \neq (r_i)^\leq$ are minimal selectors in $p^\in$.
	
	For $\Leftarrow$, first note that to obtain an anti-injective subsequence it is enough that for all $p \in \PP$ there are $q, r < p$ with $q \nleq r$ and $r \nleq q$:
	If $q, r < p$ are incomparable with $p \in G_m$, $q \in G_i$, and $r \in G_j$ for $i, j > m$, then by co-injectivity for every $n \leq i, j$ there are $q', r' \in G_n$ such that $q'{}^{\sqsubset^i_n} = \{q\}$ and $r'{}^{\sqsubset^j_n} = \{r\}$, and so we can find $n > m$ that works for every $p \in G_m$.
	Now if $\Spec{\PP}$ is perfect, then every $p \in \PP$ has incomparable predecessors.
	Otherwise, the set of $\leq$-predecessors of $p$ is linearly ordered. Then $p^\leq\cup p^\geq$ is the only minimal selector containing $p$, and therefore $p^\in$ is a singleton.
	
	\ref{CaseAntiInj2}: Since \ref{CaseCoinjective} holds, $\mathbb P$ is predetermined. As $\mathbb P$ is an $\omega$-poset, being branching and being cap-determined coincide (see Figure~\ref{figure1}).
	For $\Rightarrow$, fix $p<q$, and let $n>m$ such that $p\in G_n$ and $q\in G_m$. Since $\sqsupset_n^m$ is anti-injective, we can find $r\in G_n\setminus\{p\}$ with $r<q$. As $p$ and $r$ are not comparable, $\mathbb P$ is branching.
	For $\Leftarrow$, let $q\in G_m$, and $p\in G_{m+1}$ with $p<q$ (such $p$ exists by \ref{CaseCoinjective}). Since $\mathbb P$ is branching, we can find $r<q$ which is incomparable with $p$. Let $n$ be such that $r\in G_n$. Since ${\sqsupset_n^m}={\sqsupset_{m+1}^m}\circ{\sqsupset_n^{m+1}}$, there is $r'\in G_{m+1}$ such that $r\leq r'\leq q$. As $r$ and $p$ are incomparable, $p\neq r'$. Since $m$ and $q$ are arbitrary, each $\sqsupset_{m+1}^{m}$ is anti-injective, and so each $\sqsupset_n^m$ is.
\end{proof}

Recall that every level of the induced poset induces an open cover of the spectrum.
In the following proposition we summarize properties of these covers. For two covers $C$ and $D$ of a set $X$, we say that $C$ \emph{consolidates} $D$ when $D$ refines $C$ and, moreover, every $c\in C$ is a union of elements of $D$, i.e.
\[
\forall c\in C\qquad (c=\bigcup\{d\in D : d\subseteq c\}).
\]

\begin{prp} \label{CoverSummary}
	Let $(G_n, \sqsupset^m_n)$ be a surjective sequence in $\mathbf{S}$, let $\PP$ be the induced poset, and let $C_n = \{g^\in: g \in G_n\}$ for every $n \in \omega$.
	Then we have the following.
	\begin{enumerate}
		\item Every $C_n$ is an open cover of $\Spec{\PP}$.
		\item Every open cover of $\Spec{\PP}$ is refined by some $C_n$.
		\item \label{itm:SequenceConsolidation} Every $C_n$ consolidates $C_{n + 1}$.
	\end{enumerate}
	For the following suppose that $(\sqsupset^m_n)$ is in $\mathbf{B}$.
	\begin{enumerate}[resume]
		\item Every $C_n$ is a minimal cover and $g \mapsto g^\in$ is a bijection $G_n \to C_n$.
        \item \label{itm:OrderFiathful} 
        We have $g \sqsupset^m_n h$ if and only if $g^\in \supseteq h^\in$ for every $m < n$, $g \in G_m$, $h \in G_n$.
		\item If $(\sqsupset^m_n)$ is edge-faithful, then $g \mapsto g^\in$ is an isomorphism of graphs when $C_n$ is endowed with the actual overlap relation.
		\item If $(\sqsupset^m_n)$ is anti-injective, then $p \mapsto p^\in$ is an isomorphism of the posets $(\PP, {\leq}) \to (\{p^\in: p \in \PP\}, \subseteq)$.
	\end{enumerate}
\end{prp}
\begin{proof}
	Most items are just a translation of \autoref{SpectrumCompactT1} and \autoref{prp:coverslevels} to the language of sequences of graphs.
	By \autoref{prp:InducedPoset}, the levels $\PP_n$ correspond to the graphs $G_n$.
	By \autoref{prp:Summary1}, if $(\sqsupset^m_n)$ is co-bijective (and anti-injective), then $\PP$ is predetermined (and cap-determined).
	
	To show \ref{itm:SequenceConsolidation} let $p \in G_n$ and $S \in p^\in$.
	By \autoref{Subspaces}, we have $[\{S\}]_n = [\{S\}]_{n + 1}{}^{\sqsubset^n_{n + 1}}$, and so there is $q \in G_{n + 1}$ with $q \leq p$ and $S \in q^\in$.
	Hence, $p^\in = \bigcup\{q^\in: p \geq q \in G_{n + 1}\}$.

    To show \ref{itm:OrderFiathful}, let $h \in G_n$.
    By co-injectivity there is a thread $(h_k)_{k\geq n}$ such that $h_n = h$ and $h_{k + 1}{}^{\sqsubset^k_{k + 1}} = \{h_k\}$ for every $k \geq n$. Then $S = (h_k)_{k \geq n} \cup h^\leq$ is a minimal selector, and for every $m \leq n$ and $g \in G_m$ such that $g^\in \supseteq h^\in$ we have $g \in S$, and so $g \sqsupset^m_n h$.
    The other implication is trivial.
\end{proof}

Let us conclude the subsection by showing that suitable compact spaces admit faithful representations.
\begin{prp} \label{prp:FaithfulRepresentation}
	Every non-empty second-countable compact $\mathsf{T}_1$ space $X$ is the spectrum of an edge-witnessing sequence in $\mathbf{B}$.
	
	Moreover, if $X$ is perfect or Hausdorff, we can obtain a sequence that is additionally anti-injective or star-surjective, respectively.
\end{prp}
\begin{proof}
	By \autoref{prp:allspaces}, $X$ is the spectrum of a non-empty predetermined cap-determined graded $\omega$-poset $\PP$.
	Let $\QQ \supseteq \PP$ be the modification of $\PP$ where below every atom $p \in \PP$ we add a decreasing sequence of new elements: $p = q_{p, 0} > q_{p, 1} > q_{p, 2} > \cdots$.
	Then $\QQ$ is still a predetermined graded $\omega$-poset, which is additionally infinite and atomless, and so it corresponds to a sequence $(G_n, \sqsupset^m_n)$ of finite sets and surjective co-surjective relations by \eqref{Induced Sequence}.
	If we put $g \sqcap h \Leftrightarrow g \wedge h$ for every $g, h \in G_n$ and $n \in \omega$, $(G_n, \sqsupset^m_n)$ becomes an edge-faithful surjective sequence in $\mathbf{S}$.
	
	The spectrum of $(G_n, \sqsupset^m_n)$ is homeomorphic to $X$ since its induced poset is $\QQ$ and the spectra of $\QQ$ and $\PP$ are canonically homeomorphic: a minimal selector in $\QQ$ is either contained in $\PP$, or it is of the form $p^\leq \cup \{q_{p, i}: i \in \omega\}$ for an atom $p \in \PP$.
	
	Finally, we use Propositions~\ref{OverlappingCoherence} and \ref{prp:Summary1}.
	Since $\QQ$ is predetermined, $(\sqsupset^m_n)$ is in $\mathbf{B}$.
	Since $(\sqsupset^m_n)$ is edge-faithful, it has an edge-witnessing subsequence.
	If $X$ is perfect, $(\sqsupset^m_n)$ has an anti-injective subsequence, and if $X$ is Hausdorff, $(\sqsupset^m_n)$ has a star-refining subsequence.
	By \autoref{Subsequence}, passing to a subsequence does not change the spectrum.
\end{proof}

In the case of metric compacta we describe a particular construction of a faithful representation that will be useful when we want to have control over the shapes of the graphs $G_n$.
\begin{prp} \label{SequenceForSpace}
    Let $X$ be a non-empty compact metric space and let $(G_n)$ be a sequence of minimal open covers of $X$ such that
    \begin{enumerate}
        \item\label{itm:consolidates} $G_n$ consolidates $G_{n + 1}$ for every $n \in \omega$,
        \item\label{itm:atomless} $G_n \cap G_{n + 1} = \emptyset$ for every $n \in \omega$,
        \item\label{itm:small} $\lim_{n \to \infty} \max\{\diam(g): g \in G_n\} = 0$, or equivalently every open cover of $X$ is refined by some $G_n$.
    \end{enumerate}
    Endow every $G_n$ with the edge relation $g \sqcap g' \Leftrightarrow g \cap g' \neq \emptyset$, and let, for $m\leq n$, ${\sqsupset}^m_{n} \subseteq G_m \times G_{n}$  be the restriction of the inclusion relation on $\mathsf{P}(X)$. Let $\PP$ denote the poset $\bigcup_{n \in \omega} G_n$ with inclusion as the order relation.

    Then $(\sqsupset^m_n)$ is an edge-faithful sequence in $\mathbf{B}$ with induced poset $\PP$, $\PP$ is a cap-basis of $X$, and hence $X$ is homeomorphic to $\Spec{\PP}$.
    Moreover, the sequence $(\sqsupset^m_n)$ is anti-injective and has an edge-witnessing and star-refining subsequence.
\end{prp}

\begin{proof}
    By \ref{itm:small}, $\PP$ is a basis of $X$.
    By \cite[Lemma~1.32]{BartosBiceV.Compacta}, $\PP$ is a predetermined graded $\omega$-poset with $\PP_n = G_n$ for every $n$, and it is a cap-basis of $X$.
    By \cite[Proposition~2.9]{BartosBiceV.Compacta}, $X$ is homeomorphic to $\Spec{\PP}$.
    By \ref{itm:atomless}, we have even $\PP_n = \mathsf r^{-1}(\{n\}) = G_n$, and so (also using that $\PP$ is graded) $(\sqsupset^m_n)$ is a well-defined sequence of co-surjective relations, and $\PP$ is its induced poset.  
    Since the levels of the graded $\omega$-poset $\PP$ are disjoint and consolidating, $\PP$ is atomless and branching.
    Also note that the sequence $(\sqsupset^m_n)$ is by definition edge-faithful.
    Hence, by \autoref{prp:Summary1}, the sequence $(\sqsupset^m_n)$ is co-bijective, anti-injective, and has a star-refining subsequence, and by \autoref{OverlappingCoherence}, it is edge-preserving, edge-surjective, and has an edge-witnessing subsequence.
\end{proof}

\begin{rmk}
    If $X$ is perfect then, even if \ref{itm:atomless} fails, we can always revert to a subsequence if necessary to make \ref{itm:atomless} hold.  Indeed, in this case \ref{itm:small} implies that, for every $m$, we have $n \geq m$ such that $\diam(g) < \diam(h)$ for every $g \in G_n$ and $h \in G_m$ and hence $G_m \cap G_n = \emptyset$.
\end{rmk}

\subsection{Connectedness} \label{ss.connected}

Next we focus on characterizing when the spectrum or its subsets are connected.

\begin{prp} \label{ConnectedLimit}
  Let $(G_n,\sqsupset_n^m)$ be an edge-faithful sequence in $\mathbf{B}$, and let $\PP$ be the induced poset.
  Then $\Spec{\PP}$ is connected if and only if every graph $G_n$ is connected.
\end{prp}
\begin{proof}
    By \autoref{prp:coverslevels}, every open cover of $\Spec{\PP}$ is refined by $\{p^\in: p \in G_n\}$ for some $n$, and for every $p, q \in G_n$ we have $p \wedge q$ if and only if $p^\in \cap q^\in \neq \emptyset$.
    In particular, $p^\in \neq \emptyset$ for every $p \in G_n$.
    It follows that $\Spec{\PP}$ is connected if and only if no $G_n$ admits a partition $G_n = A \cup B$ such that $(\bigcup_{a \in A} a^\in) \cap (\bigcup_{b \in B} b^\in) = \emptyset$, i.e. if and only if  every $G_n$ with $\wedge\restriction G_n \times G_n$ as the edge relation is a connected graph.
    The conclusion follows from edge-faithfulness.
\end{proof}

It turns out that for acyclic graphs, edge-faithfulness does not have to be assumed a priori.
\begin{prp} \label{ConnectedTreeLimit}
   Let $(G_n,\sqsupset_n^m)$ be a sequence in $\mathbf{B}$ and suppose each $G_n$ is acyclic. Let $\PP$ be the induced poset.
    Then $\Spec{\PP}$ is connected if and only if every graph $G_n$ is connected and $(\sqsupset^m_n)$ has an edge-witnessing subsequence.
\end{prp}
\begin{proof}
    Suppose $\Spec{\PP}$ is connected.
    Then the graphs $G_n$ are connected by \autoref{Overlapping} and the proof \autoref{ConnectedLimit}, or alternatively by \autoref{ConnectedSubsets}.
    If for every $n$ and every adjacent $g, h \in G_n$ we have $g \wedge h$, then we can construct an edge-witnessing subsequence as in the proof of \autoref{OverlappingCoherence}.
    Otherwise, there are adjacent $g, h \in G_n$ such that $g^\geq \cap h^\geq = \emptyset$.
    As $G_n$ is a tree, we may consider the connected components $G \ni g$ and $H \ni h$ of $G_n$ with the edge $(g, h)$ removed.
    Then $G^\geq \cap H^\geq = \emptyset$, and the the graph $(G_n, \wedge)$ is disconnected, and so $\Spec{\PP}$ is disconnected as in the proof of \autoref{ConnectedLimit}.

    Suppose that the graphs $G_n$ are connected and that $(\sqsupset^m_n)$ has an edge-witnessing subsequence.
    The sequence $(\sqsupset^m_n)$ is edge-surjective by \autoref{SurjectiveImpliesEdgeSurjective}, edge-faithful by \autoref{OverlappingCoherence}, and so $\Spec{\PP}$ is connected by \autoref{ConnectedLimit}.
\end{proof}

\begin{prp} \label{ConnectedSubsets}
    Let $(G_n, \sqsupset^m_n)$ be a sequence in $\mathbf{S}$ with induced poset $\PP$, and let $A \subseteq \Spec{\PP}$.
    We consider the following conditions:
    \begin{enumerate}
        \item \label{itm:connected} $A$ is connected,
        \item \label{itm:connected_traces} $[A]_n$ is connected for every $n \in \omega$,
        \item \label{itm:connected_fat_traces} $[A]_n^\sqcap$ is connected for every $n \in \omega$.
    \end{enumerate}
    We have \ref{itm:connected} $\Rightarrow$ \ref{itm:connected_traces} $\Rightarrow$ \ref{itm:connected_fat_traces}.
    If $(\sqsupset^m_n)$ is co-bijective and edge-faithful, $\Spec{\PP}$ is Hausdorff, and $A \subseteq \Spec{\PP}$ is closed, then \ref{itm:connected} $\Leftrightarrow$ \ref{itm:connected_traces} $\Leftrightarrow$ \ref{itm:connected_fat_traces}. 
\end{prp}
\begin{proof}
    \ref{itm:connected} $\Rightarrow$ \ref{itm:connected_traces}:
    Suppose $[A]_n$ is disconnected, i.e. $[A]_n = H_1 \cup H_2$ for some non-empty subgraphs $H_1, H_2 \subseteq G_n$ with no edge between them.
    Let $U_i = \bigcup\{h^\in: h \in H_i\}$ be the induced open subsets of $\Spec{\PP}$.
    We have $A \subseteq U_1 \cup U_2$.
    As $H_i \neq \emptyset$, we have $A \cap U_i \neq \emptyset$.
    As there is no edge between $H_1$ and $H_2$, we have $U_1 \cap U_2 = \emptyset$ by \autoref{Overlapping}.
    This shows that $A$ is disconnected.
    \ref{itm:connected_traces} $\Rightarrow$ \ref{itm:connected_fat_traces} is clear as $C^\sqcap$ is connected for every connected subgraph $C \subseteq G_n$.

    Next suppose that $(\sqsupset^m_n)$ is co-bijective and edge-faithful, so we have the equivalences in \autoref{Overlapping}, that $\Spec{\PP}$ is Hausdorff and so can be endowed with a compatible metric, and that $A$ is closed.
    To show \ref{itm:connected_fat_traces} $\Rightarrow$ \ref{itm:connected} we suppose that $A$ is disconnected, i.e. there are non-empty disjoint closed subsets $B, C \subseteq \Spec{\PP}$ such that $A = B \cup C$.
    As $\Spec{\PP}$ is a Hausdorff compactum, there are open sets $U_0 \supseteq B$ and $U_3 \supseteq C$ with disjoint closures.
    Let us put $U_1 = \Spec{\PP} \setminus (B \cup \overline{U_3})$ and $U_2 = \Spec{\PP} \setminus (\overline{U_0} \cup C)$.
    Then, $\{U_i: i < 4\}$ is an open cover of $\Spec{\PP}$ that is a chain, i.e. $U_i \cap U_j \neq \emptyset$ if and only if $|i - j| \leq 1$.
    Let $n \in \omega$ be such that the cover $\{g^\in: g \in G_n\}$ refines $\{U_i: i < 4\}$.
    Then $[B]_n^\sqcap \cap [C]_n^\sqcap = \emptyset$.
    Otherwise, there are $[B]_n \ni b \sqcap g \sqcap c \in [C]_n$ in $G_n$, and so     $b^\in \subseteq U_0$ and $c^\in \subseteq U_3$, but also $b^\in \cap g^\in \neq \emptyset \neq c^\in \cap g^\in$ by \autoref{Overlapping}.
    Then $g^\in$ cannot refine any member of the chain $\{U_i: i < 4\}$.
    Together with the facts that $[A]_n^\sqcap = [B]_n^\sqcap \cup [C]_n^\sqcap$ and that $[B]_n \neq \emptyset \neq [C]_n$, we have that $[A]_n^\sqcap$ is disconnected.
\end{proof}

Recall that by a \emph{continuum} we mean a non-empty compact connected metrizable space.
A continuum is \emph{locally connected} or \emph{Peano} if it has an open basis consisting of connected sets, or equivalently if it is a continuous image of the unit interval, by the Hahn--Mazurkiwicz theorem (see \cite[VIII]{Nadler1992}).
A continuum $X$ is \emph{tree-like} if every open cover of $X$ can be refined by a finite open cover whose overlap graph is a tree.
A continuum $X$ is \emph{hereditarily unicoherent} if for every two subcontinua $A, B \subseteq X$ we have that $A \cap B$ is connected.
Every tree-like continuum is hereditarily unicoherent (see \cite[Theorem~12.2 and page 232]{Nadler1992}) and clearly a circle cannot be embedded into a hereditarily unicoherent continuum.
For Peano continua, these conditions are equivalent and define the notion of a \emph{dendrite} (see \cite[X]{Nadler1992}).



\begin{prp} \label{LocallyConnected}
    Let $(G_n, \sqsupset^m_n)$ be an edge-faithful sequence in $\mathbf{B}$ with induced poset $\PP$ and suppose that $\Spec{\PP}$ is Hausdorff.
    If the relations $\sqsupset^m_n$ are monotone, then $p^\in$ is connected for every $p \in \PP$, and so $\Spec{\PP}$ is locally connected.
\end{prp}
\begin{proof}
    Let $m \in \omega$ be such that $p \in G_m$, and let $A = p^\in$.
    For every $n \geq m$ let $C_n = p^{\sqsupset^m_n}$, which is a connected subset of $G_n$ since $\sqsupset^m_n$ is monotone.
    We will show that $C_n^\sqcap = [A]_n$ for every $n \geq m$.
    It will follow that $p^\in = A$ is connected by \autoref{ConnectedSubsets}, and so $\Spec{\PP}$ is locally connected as sets $p^\in$ form a basis.

    Let $g \in G_n$.
    We have $g \in [A]_n$ if and only if $g^\in \cap p^\in \neq \emptyset$.
    By \autoref{CoverSummary} we have $p^\in = \bigcup\{c^\in: c \in C_n\}$, and so $g \in [A]_n$ if and only if for some $c \in C_n$ we have $g^\in \cap c^\in \neq \emptyset$.
    The latter is equivalent to $g \mathrel{\sqcap} c$ by \autoref{Overlapping}.
    Hence, $[A]_n = C_n^\sqcap$.
\end{proof}

By combining the previous observations we obtain the following corollary as an application.

\begin{cor} \label{ContinuumTreeLimit}
    Let $(G_n, \sqsupset^m_n)$ be a sequence in $\mathbf B$ such that every $G_n$ is a tree, and let $\PP$ be the induced poset.
    The following are equivalent.
    \begin{enumerate}
        \item $\Spec{\PP}$ is connected and Hausdorff.
        \item $(\sqsupset^m_n)$ has an edge-witnessing and a star-refining subsequence.
    \end{enumerate}
    If the conditions hold, then $(\sqsupset^m_n)$ is edge-faithful and $\Spec{\PP}$ is a tree-like continuum.
    If moreover $(\sqsupset^m_n)$ is monotone, then $\Spec{\PP}$ is a dendrite.
\end{cor}
\begin{proof}
    By \autoref{SurjectiveImpliesEdgeSurjective}, the sequence $(\sqsupset^m_n)$ is edge-surjective.
    By \autoref{ConnectedTreeLimit}, if $\Spec{\PP}$ is connected, then $(\sqsupset^m_n)$ has an edge-witnessing subsequence.
    Hence, $(\sqsupset^m_n)$ is edge-faithful on both sides of the equivalence, and the equivalence follows from Propositions~\ref{ConnectedLimit} and \ref{HausdorffLimit}.
    
    Assuming the equivalent conditions, $\Spec{\PP}$ is clearly a continuum, and it is tree-like as every open cover is refined by $\{p^\in: p \in G_n\}$ for some $n$.
    If $(\sqsupset^m_n)$ is monotone, then, by \autoref{LocallyConnected}, $\Spec{\PP}$ is also locally connected and so a dendrite.
\end{proof}

\subsection{Fra\"iss\'e sequences}\label{ss.FraisseSequences}

We are particularly interested in sequences which are `as complicated as possible', in a sense capturing a whole given subcategory of graphs.
In particular, such sequences absorb all morphisms in the following sense.

\begin{dfn}
A sequence $(G_n, \sqsupset_n^m)$ in a subcategory $\mathbf{C}$ of $\mathbf{G}$ is \emph{\textup{(}lax-\textup{)}absorbing} if, whenever ${\dashv}\in\mathbf{C}^{G_m}_G$, there are $n\geq m$ and ${\Dashv} \in\mathbf{C}^G_{G_n}$ with ${\dashv\circ\Dashv} ={\sqsupset^m_n}$ (${\dashv\circ\Dashv} \subseteq{\sqsupset^m_n}$).
More symbolically, these notions are defined as follows.
\begin{align*}
 \tag{Absorbing}{\dashv}\in\mathbf{C}^{G_m}_G\qquad&\Rightarrow\qquad\text{ exists }{\Dashv}\in\mathbf{C}^G_{G_n}\ ({\dashv\circ\Dashv} ={\sqsupset^m_n}),\\
 \tag{lax-Absorbing}\dashv\ \in\mathbf{C}^{G_m}_G\qquad&\Rightarrow\qquad\text{ exists }{\Dashv}\in\mathbf{C}^G_{G_n}\ ({\dashv\circ\Dashv}\subseteq{\sqsupset^m_n}).
\end{align*}

A \emph{\textup{(}lax-\textup{)}Fra\"{i}ss\'e} sequence is a (lax-)absorbing sequence that is also \emph{cofinal}, meaning that for every object $G$ in $\mathbf{C}$ there is a morphism (in $\mathbf{C}$) onto $G$ from a member of the sequence, i.e.
\[ \tag{Cofinal} G \in \mathbf{C} \qquad\Rightarrow\qquad \text{exists }{\dashv} \in \mathbf{C}^G_{G_n}.
\]
\end{dfn}

Fra\"{i}ss\'e sequences play a key role in construction of universal homogeneous objects, see \cite{Kubis2014b}.
Usually, the cofinality does not have to be verified as it is automatic if the category $\mathbf{C}$ is \emph{directed} in that all objects $E,F\in\mathbf{C}$ are the codomains of morphisms with a common domain.

The proof of the following is immediate and left as an exercise.
\begin{prp} \label{DirectedCofinal}
 In a directed subcategory $\mathbf{C} \subseteq \mathbf{G}$, every lax-absorbing sequence ${\sqsupset_n^m}\in\mathbf{C}^{G_m}_{G_n}$ is cofinal and so lax-Fra\"{i}ss\'e.\qed
\end{prp}

Of course, not all directed subcategories $\mathbf{C}$ of $\mathbf{G}$ have (lax-)Fra\"iss\'e sequences. The key property for $\mathbf{C}$ to have here is \emph{amalgamation}, which means that, for any ${\sqsupset}\in\mathbf{C}_H^G$ and ${\sqni}\in\mathbf{C}_I^G$, we have ${\dashv}\in\mathbf{C}_J^H$ and ${\Dashv}\in\mathbf{C}_J^I$ such that ${\sqsupset\circ\dashv}={ \sqni\circ\Dashv}$, i.e.
\[\tag{Amalgamation}{\sqsupset}\in\mathbf{C}_H^G, {\sqni}\in\mathbf{C}_I^G\ \Rightarrow\ \text{exist }{\dashv}\in\mathbf{C}_J^H, {\Dashv}\in\mathbf{C}_J^I\ ({\sqsupset\circ\dashv}= {\sqni\circ\Dashv}).\]
Note that if $\mathbf{C}$ has amalgamation as well as a \emph{weakly terminal object} $T$ (i.e. such that $\mathbf{C}^T_G\neq\emptyset$, for all $G\in\mathbf{C}$) then it is automatically directed. 

In what follows we will use the definition of wide subcategories and ideals, as given in \autoref{dfn:idealswide}.

\begin{prp} \label{FraisseExists}
A subcategory $\mathbf{C}$ of $\mathbf{G}$ has a Fra\"iss\'e sequence whenever $\mathbf{C}$ has amalgamation and is directed.
\end{prp}
\begin{proof}
This follows from \cite[Theorem 3.7]{Kubis2014b} as soon as we note that $\mathbf{G}$ and hence $\mathbf{C}$ is an essentially countable category -- the full subcategory of graphs whose vertices lie in some fixed countable set like $\omega$ forms a countable essentially wide (i.e. wide once we include compositions with isomorphisms) subcategory of $\mathbf{G}$. In particular, $\mathbf{C}$ has a countable dominating family.
\end{proof}

It is easy to see that a subsequence of a (lax-)Fra\"{i}ss\'e sequence is still (lax-)Fra\"{i}ss\'e.
Moreover, we can construct subsequences in wide ideals. 
\begin{prp} \label{IdealFraisseSubsequence}
Let $\mathbf{C} \subseteq \mathbf{G}$ be a subcategory and let $I \subseteq \mathbf{C}$ be an ideal.
\begin{enumerate}
    \item If $I$ is wide, then every Fra\"{i}ss\'e sequence has a subsequence in $I$.
    \item If $I$ is wide and lax-closed, then every lax-Fra\"{i}ss\'e sequence has a subsequence in $I$.
    \item If a single cofinal sequence has a subsequence in $I$, then $I$ is wide.
\end{enumerate}
\end{prp}
\begin{proof}
    Take any (lax-)Fra\"{i}ss\'e sequence $(G_n, \sqsupset_n^m)$ in $\mathbf{C}$.
    Let $k \in \omega$ be arbitrary.
    As $I$ is wide, we have ${\dashv}\in I\cap\mathbf{C}^{G_k}_H$, for some $H\in\mathbf{C}$.
    As $(\sqsupset_n^m)$ is (lax-)Fra\"{i}ss\'e, we have ${\Dashv}\in\mathbf{C}^H_{G_l}$, for some $l > k$, such that ${\sqsupset^k_l} = {\dashv} \circ {\Dashv}$ (or ${\sqsupset^k_l} \supseteq {\dashv} \circ {\Dashv}$).
    As $I$ is a (lax-closed) ideal, ${\sqsupset^k_l} \in I$.
    Finally we obtain a subsequence in $I$ by \autoref{IdealSubsequences}.

    On the other hand, if a cofinal sequence has a subsequence $(H_n, \sqni^m_n)$ in $I$, the subsequence is cofinal as well.
    Then for every $K \in \mathbf{C}$ there is ${\dashv} \in \mathbf{C}^K_{H_m}$ for some $m$, and ${\dashv} \circ {\sqni^m_{m + 1}} \in I$.
\end{proof}

We have seen in \autoref{ss.faithful} that edge-faithful sequences in $\mathbf{B}$ nicely represent their spectra.
For this reason and also in order to avoid trivialities, we shall consider categories of co-bijective morphisms.
While a subcategory $\mathbf{C}$ of $\mathbf{B}$ may have many different (lax-)Fra\"{i}ss\'e sequences, the key point for our work is that they all have homeomorphic spectra. 

\begin{prp}\label{LaxFraisseSpectra}
Assume $\mathbf{C}$ is a subcategory of $\mathbf{B}$, and let $(G_n,\sqsupset_n^m)$ and $(H_n,{\sqni_n^m})$ be sequences in $\mathbf C$, with induced posets $\PP$ and $\QQ$, respectively.
Suppose one of the following.
\begin{enumerate}
    \item The sequences $(\sqsupset^m_n)$ and $(\sqni^m_n)$ are lax-Fra\"{i}ss\'e, and they have edge-witnessing star-refining subsequences.
    \item The sequences $(\sqsupset^m_n)$ and $(\sqni^m_n)$ are Fra\"{i}ss\'e.
\end{enumerate}
Then the spectra $\Spec{\PP}$ and $\Spec{\QQ}$ are homeomorphic.
\end{prp}

\begin{proof}
Let us prove the lax-Fra\"{i}ss\'e case.
First, we can assume that the two lax-Fra\"iss\'e sequences of interest are edge-witnessing and star-refining, as the spectrum given a sequence and the one given by a subsequence are homeomorphic (\autoref{Subsequence}).

As $(\sqni^m_n)$ is cofinal, we can take any $m_0\in\omega$ and find ${\dashv_0}\in\mathbf{C}^{G_{m_0}}_{H_{n_0}}$.
As $(\sqsupset_n^m)$ is lax-absorbing, we have ${\Dashv_0}\in\mathbf{C}^{H_{n_0}}_{G_{m_1}}$ with ${\dashv_0\circ\Dashv_0}\subseteq{\sqsupset^{m_0}_{m_1}}$.
As $(\sqni_n^m)$ is lax-absorbing, we have ${\dashv_1} \in\mathbf{C}^{G_{m_1}}_{H_{n_1}}$ with ${\Dashv_0\circ\dashv_1}\subseteq{\sqni^{n_0}_{n_1}}$.
Continuing in this way, we obtain ${\dashv_k}\in\mathbf{C}^{G_{m_k}}_{H_{n_k}}$ and ${\Dashv_k}\in\mathbf{C}^{H_{n_k}}_{G_{m_{k+1}}}$ such that ${\dashv_k\circ\Dashv_k}\subseteq{\sqsupset^{m_k}_{m_{k+1}}}$ and ${\Dashv_k\circ\dashv_{k+1}}\subseteq{\sqni^{n_k}_{n_{k+1}}}$, for all $k\in\omega$.
As the morphisms $\sqsupset_n^m$ and $\sqni_n^m$ are edge-witnessing and therefore edge-surjective (see \autoref{fig:morphisms}), $\wedge$ and $\sqcap$ agree on levels of $\PP$ and $\QQ$ respectively, by \autoref{OverlappingCoherence}.
Hence the morphisms $\dashv_k$ and $\Dashv_k$ are $\wedge$-preserving.
Since a subrelation of a $\wedge$-preserving relation is again $\wedge$-preserving, each $\dashv_k\circ\Dashv_k$ is $\wedge$-preserving. Likewise, each $\Dashv_k\circ\dashv_{k+1}$ is $\wedge$-preserving.
As the morphisms $\sqsupset_n^m$ and $\sqni_n^m$ are star-refining, the induced posets are regular (see \autoref{prp:Summary1}).
The result now follows immediately from \cite[Proposition 3.11]{BartosBiceV.Compacta}.

To prove the Fra\"{i}ss\'e case, we construct the morphisms $\dashv_k$ and $\Dashv_k$ analogously, now satisfying ${\dashv_k\circ\Dashv_k} = {\sqsupset^{m_k}_{m_{k+1}}}$ and ${\Dashv_k\circ\dashv_{k+1}} = {\sqni^{n_k}_{n_{k+1}}}$, for all $k\in\omega$.
The result follows from \autoref{Subsequence} applied multiple times: $(\dashv_k\circ\Dashv_k)$ is both a subsequence of $(\sqsupset^m_n)$ and of the intertwining sequence $(\dashv_0, \Dashv_0, \dashv_1, \ldots)$, and analogously for $(\Dashv_k \circ \dashv_{k + 1})$.
\end{proof} 

\begin{lemma} \label{SubcategoryOfE}
    Let $\mathbf{C} \subseteq \mathbf{B}$ be a subcategory with amalgamation such that edge-witnessing morphisms are wide.
    Then all morphisms in $\mathbf{C}$ are edge-surjective.
\end{lemma}
\begin{proof}
    Let ${\sqsupset} \in \mathbf{C}^G_H$ and let $g \sqcap g' \in G$.
    There is an edge-witnessing morphism ${\sqni} \in \mathbf{C}^G_I$ and a vertex $i \in I$ with $g, g' \sqni i$.
    By the amalgamation property, there are ${\dashv} \in \mathbf{C}^H_J$ and ${\Dashv} \in \mathbf{C}^I_J$ with ${\sqsupset}\circ{\dashv} = {\sqni}\circ{\Dashv}$.
    Hence, there are vertices $j \in J$ and $h,  h' \in H$ such that $i \Dashv j$ and $g \sqsupset h \dashv j$ and $g' \sqsupset h' \dashv j$.
    We have $h \sqcap h'$ since $\dashv$ is edge-preserving, and we are done.
\end{proof}

The following theorem summarizes a typical situation. (Recall, the category $\mathbf{E}$ is that of edge-surjective morphisms.)
\begin{thm} \label{cor:FraisseLimit}
    Let $\mathbf{C} \subseteq \mathbf{B}$ be a directed subcategory with amalgamation such that edge-witnessing morphisms and star-refining morphisms are wide. Then:
    \begin{enumerate}
        \item\label{itm:exists} There is a Fra\"{i}ss\'e sequence in $\mathbf{C}$.
        \item\label{itm:edge-surjective} We have $\mathbf{C} \subseteq \mathbf{E}$, and edge-witnessing and star-refining morphisms form wide ideals.
        \item\label{itm:all_nice} Every lax-Fra\"{i}ss\'e sequence in $\mathbf{C}$ is edge-faithful and has an edge-witnessing and a star-refining subsequence.
        \item\label{itm:all_same} All lax-Fra\"{i}ss\'e sequences in $\mathbf{C}$ have homeomorphic spectra.
        \item\label{itm:Flim_Hausdroff} The spectrum is metrizable.
    \end{enumerate}
\end{thm}
\begin{proof}
    \ref{itm:edge-surjective} follows from \autoref{SubcategoryOfE} and \autoref{Ideals}.
    Every lax-Fra\"{i}ss\'e sequence has an edge-witnessing and star-refining subsequence by \autoref{LaxClosed} and \autoref{IdealFraisseSubsequence}.
    It follows from \autoref{OverlappingCoherence} that it is also edge-faithful.
    That proves~\ref{itm:all_nice}.
    Then~\ref{itm:all_same} follows from \autoref{LaxFraisseSpectra}, and \ref{itm:Flim_Hausdroff} follows from \autoref{HausdorffLimit}.
\end{proof}

The above theorem calls for the following abuse of terminology:
\begin{dfn}\label{dfn:FraisseLimit}
     Let $\mathbf{C} \subseteq \mathbf{B}$ be a directed subcategory with amalgamation such that edge-witnessing morphisms and star-refining morphisms are wide (and so form wide ideals as also $\mathbf{C} \subseteq \mathbf{E})$.
     The unique spectrum of a lax-Fra\"{i}ss\'e sequence is called the \emph{Fra\"{i}ss\'e limit} of $\mathbf{C}$.
\end{dfn}

\subsection{The clique functor}\label{ss.clique}

To prove that a subcategory $\mathbf{C}$ of $\mathbf{G}$ has amalgamation it will often suffice to consider morphisms that are functions. Indeed, this will be the case whenever $\mathbf{C}$ is closed under the clique functor which we now describe. First let us denote the cliques (i.e. non-empty complete subgraphs) of a graph $(G,\sqcap)$ by
\[\mathsf{X}G=\{C\in\mathsf{P}(G)\setminus\{\emptyset\}:C\times C\subseteq\sqcap\}.
\]
In particular, cliques of triangle-free graphs are really just edges, including loops.

The \emph{clique graph} has vertices $\mathsf{X}G$ and edges given by containment, i.e. cliques are related in $\mathsf{X}G$ if and only if one is contained in the other.

Let $\mathbf{F}$ denote the wide subcategory $\mathbf{G}$ consisting of all morphisms which are functions. Recall that $\mathbf{S}$ is the wide subcategory of $\mathbf{G}$ with co-surjective morphisms.

\begin{prp}\label{CliqueFunctor}
We have a functor $\mathsf{X}:\mathbf{S}\rightarrow\mathbf{F}$ taking a graph $G$ to its clique graph $\mathsf{X}G$ and any ${\sqsupset}\in\mathbf{S}_G^H$ to the function $\mathsf{X}^\sqsupset\in\mathbf{F}_{\mathsf{X}G}^{\mathsf{X}H}$ given by
\[\mathsf{X}^\sqsupset(C)=C^\sqsubset.\]
Moreover, if $H$ is triangle-free and ${\sqsupset}\in\mathbf{B}_G^H\cap\mathbf{E}_G^H$ then $\mathsf{X}^\sqsupset\in\mathbf{B}_{\mathsf{X}G}^{\mathsf{X}H}$ too.
\end{prp}

\begin{proof}
If ${\sqsupset}\in\mathbf{S}_G^H$ and $C$ is a clique of $G$ then $\mathsf{X}^\sqsupset(C)=C^\sqsubset$ is a clique of $H$, as $\sqsupset$ is co-surjective and preserves the edge relation. Moreover, if $D$ is another clique of $G$ then $C\subseteq D$ certainly implies $C^\sqsubset\subseteq D^\sqsubset$, hence if $C$ and $D$ are adjacent so are $\mathsf{X}^\sqsupset(C)$ and $\mathsf{X}^\sqsupset(D)$, i.e. $\mathsf{X}^\sqsupset$ is in $\mathbf{F}_{\mathsf{X}G}^{\mathsf{X}H}$. For any other morphism ${\sqni}\in\mathbf{S}^F_H$, we immediately see that 
\[\mathsf{X}^{{\sqni\circ\sqsupset}}(C)=C^{{\sqsubset\circ\sqin}}=C^{\sqsubset{\sqin}}=\mathsf{X}^{\sqni}(\mathsf{X}^\sqsupset(C))=\mathsf{X}^{\sqni}\circ\mathsf{X}^\sqsupset(C).\]
Thus $\mathsf{X}^{{\sqni\circ\sqsupset}}=\mathsf{X}^{\sqni}\circ\mathsf{X}^\sqsupset$, showing that $\mathsf{X}$ is a functor.

For the last statement, take any ${\sqsupset}\in\mathbf{B}_G^H\cap\mathbf{E}_G^H$ and any clique $D\in\mathsf{X}H$. If $D=\{h\}$ then the co-injectivity of $\sqsupset$ yields $g\in G$ with $g^\sqsubset=\{h\}$ and hence $\mathsf{X}^\sqsupset(\{g\})=\{h\}=D$. On the other hand, if $D=\{h,h'\}$, for distinct $h,h'\in H$, then we can take $g,g'\in G$ with $h\sqsupset g\sqcap g'\sqsubset h'$, as $\sqsupset$ is edge-surjective, and hence $\mathsf{X}^\sqsupset(\{g,g'\})\supseteq\{h,h'\}=D$. If $D$ is triangle-free then this implies $\mathsf{X}^\sqsupset(\{g,g'\})=D$, showing that $\mathsf{X}^\sqsupset$ is surjective and hence $\mathsf{X}^\sqsupset\in\mathbf{B}^{\mathsf{X}H}_{\mathsf{X}G}$.
\end{proof}

\begin{prp}\label{prop:inmorphismclique}
The restriction $\in_G$ of the relation $\in$ to $G\times\mathsf{X}G$, for each $G\in\mathbf{G}$, yields a natural transformation\footnote{Recall that a natural transformation between functors $F, G\maps \mathbf{C} \to \mathbf{C'}$ is a family of morphisms $(\eta_C\maps F(C) \to G(C))_{C \in \mathbf{C}}$ such that $G(f) \circ \eta_C = \eta_D \circ F(f)$ for every $\mathbf{C}$-morphism $f\maps C \to D$.}
from the clique functor $\mathsf{X}$ to the identity functor $\mathsf{I}$.
\end{prp}

\begin{proof}
Say we have adjacent $C,D\in\mathsf{X}G$ and let $c\in C$ and $d\in D$. If $C\subseteq D$ then $c\in D$, while if $D\subseteq C$ then $d\in C$. In either case, $c\sqcap d$, as $C$ and $D$ are cliques, showing that $\in_G$ preserves the edge relation of $\mathsf{X}G$ and is thus a morphism in $\mathbf{S}^G_{\mathsf{X}G}$. Now just note that ${\sqsupset\circ\in_H}={\in_G\circ\ \mathsf{X}^\sqsupset}$, for any ${\sqsupset}\in\mathbf{S}_G^H$, as
\[d \mathrel{(\sqsupset\circ\in_H)} C\ \Leftrightarrow\ \exists c\ (d\sqsupset c\in C)\ \Leftrightarrow \ d\in C^\sqsubset = \mathsf{X}^\sqsupset(C)\ \Leftrightarrow\ d \mathrel{({\in_G} \circ {\mathsf{X}^\sqsupset})} C.\]
Thus $(\in_G)_{G\in\mathbf{G}}$ is a natural transformation from $\mathsf{X}$ to $\mathsf{I}$.
\end{proof}

\begin{dfn}\label{defin:cliqueclosed}
We call a subcategory $\mathbf{C}$ of $\mathbf{S}$ \emph{clique-closed} if ${\in_G}\in\mathbf{C}^G_{\mathsf{X}G}$ (in particular $\mathsf{X}G\in\mathbf{C}$), for all $G\in\mathbf{C}$, and $\mathsf{X}^\sqsupset\in\mathbf{C}_{\mathsf{X}G}^{\mathsf{X}H}$, whenever ${\sqsupset}\in\mathbf{C}_G^H$.
\end{dfn}

In other words, $\mathbf{C}$ is clique-closed if $\mathsf{X}$ restricts to a functor on $\mathbf{C}$ and $(\ni_G)$ restricts to a natural transformation between the restricted endofunctors $\mathsf{X}$ and $\mathsf{I}$.

Given categories $\mathbf{C}$ and $\mathbf{C'}$,  we let $\mathbf{C}\cap\mathbf{C'}$ be the category of all objects and morphisms that are both in $\mathbf{C}$ and $\mathbf{C'}$.

\begin{cor}\label{cor:cliqueclosed}
If $\mathbf{C}$ is clique-closed and $\mathbf{C}\cap\mathbf{F}$ has amalgamation, so does $\mathbf{C}$.
\end{cor}

\begin{proof}
Take ${\sqsupset}\in\mathbf{C}^G_H$ and ${\sqni}\in\mathbf{C}_I^G$. If $\mathbf{C}$ is clique-closed and $\mathbf{C}\cap\mathbf{F}$ has amalgamation then, in particular, we can amalgamate $\mathsf{X}^\sqsupset$ and $\mathsf{X}^{\sqni}$, i.e. we have functions $\phi\in\mathbf{C}_J^{\mathsf{X}H}$ and $\pi\in\mathbf{C}_J^{\mathsf{X}I}$ with ${\mathsf{X}^\sqsupset\circ\phi}={\mathsf{X}^{\sqni}\circ\pi}$ so
\[
{\sqsupset} \circ {\in_H} \circ \phi
\ \ =\ \ {\in_G} \circ \mathsf{X}^\sqsupset \circ \phi
\ \ =\ \ {\in_G} \circ \mathsf{X}^{\sqni} \circ \pi
\ \ =\ \ \sqni\circ\in_I\circ\ \pi.
\]
Thus ${\in_H} \circ \phi$ and ${\in_I} \circ\ \pi$ amalgamate $\sqsupset$ and $\sqni$, i.e. $\mathbf{C}$ has amalgamation.
\end{proof}

Before moving on, let us note some more basic properties of $\in_G$, for any $G\in\mathbf{G}$.

\begin{prp}\label{CliqueTransformation}
$\in_G$ is always co-bijective, edge-witnessing and monotone.
\end{prp}

\begin{proof}
Co-surjectivity of $\in_G$ follows from the fact that cliques are, by definition, non-empty, while co-injectivity follows from the fact that every vertex $g$ is the unique element of the singleton clique $\{g\}$. For any $g,h\in G$ with $g\sqcap h$, we see that $g,h\in_G\{g,h\}\in\mathsf{X}G$, showing that $\in_G$ is also edge-witnessing. In particular, $\in_G$ is edge-surjective so to verify monotonicity it suffices to consider single vertices $g\in G$ (\autoref{MonotoneEdgeSurjective}). But any clique $C\in \mathsf{X}G$ containing $g$ is related to the singleton clique $\{g\}$, so preimages of single vertices are indeed connected and hence $\in_G$ is monotone.
\end{proof}

\begin{cor}\label{XGConnected}
 If a graph $G$ is connected then so is its clique graph $\mathsf{X}G$.
\end{cor}

\begin{proof}
 As $\in_G$ is co-surjective, $G^{\in_G}=\mathsf{X}G$. If $G$ is connected, this implies $\mathsf{X}G$ is also connected, as $\in_G$ is also monotone.
\end{proof}

By showing that our category is clique-closed, besides reducing amalgamation to functions, we get that edge-witnessing morphisms are wide for free.
\begin{cor} \label{CliqueClosedEdgeWitnessing}
    Let $\mathbf{C}$ be a clique-closed category.
    Then edge-witnessing morphisms are wide in $\mathbf{C}$.
\end{cor}
\begin{proof}
    For every $G\in\mathbf{C}$, ${\in_G}\in\mathbf{C}^G_{\mathsf{X}G}$ is edge-witnessing, by \autoref{CliqueTransformation}.
\end{proof}

Let us conclude with an example of how to obtain a clique-closed category. Recall, a subcategory $\mathbf C$ of a category $\mathbf C'$ is full if $\mathbf C$ contains some objects of $\mathbf C'$, but all morphisms between them. 
\begin{cor} \label{TreesCliqueClosed}
    Let $\mathbf{T} \subseteq \mathbf{B}$ be the full subcategory of all trees.
    $\mathbf{T}$ as well as any full $\mathbf{C} \subseteq \mathbf{T}$ such that $\mathsf{X}G \in \mathbf{C}$ for every $G \in \mathbf{C}$ is clique-closed.
\end{cor}
\begin{proof}
    For every tree $T$, the cliques of $T$ are exactly the singletons and edges, i.e. $\mathsf{X}T = \{\{g\}: g \in T\} \cup \{\{g, h\}: g \sim h \in T\}$.
    It is easy to see that $\mathsf{X}T$ is the tree constructed from $T$ by replacing every edge $g \sim h$ by $\{g\} \sim \{g, h\} \sim \{h\}$, similarly to \autoref{EdgeSplitting}.
    Hence, $\mathbf{T}$ itself satisfy the assumptions on $\mathbf{C}$.

    By \autoref{SurjectiveImpliesEdgeSurjective} we have $\mathbf{C} \subseteq \mathbf{E}$, and so by the last part of \autoref{CliqueFunctor}, $\mathsf{X}^{\sqsupset} \in \mathbf{B}^{\mathsf{X}T}_{\mathsf{X}S} = \mathbf{C}^{\mathsf{X}T}_{\mathsf{X}S}$ for every ${\sqsupset} \in \mathbf{C}^T_S$.
    By \autoref{CliqueTransformation}, ${\in_T} \in \mathbf{C}^T_{\mathsf{X}T}$ for every $T \in \mathbf{C}$.
    Altogether, $\mathbf{C}$ is clique-closed.
\end{proof}

\section{Examples}\label{s.Examples}

The general theory we have developed so far leads to the following procedure for constructing compacta from the desired overlap graphs of their finite open covers:
\begin{enumerate}[itemsep=0.5ex]
    \item We define a category $\mathbf{C}$ of graphs of desired shape and suitable relational morphisms.
    As argued in the previous section, it is natural to consider subcategories of $\mathbf{B} \cap \mathbf{E}$, though sometimes this is not assumed a priori and follows from other assumptions (e.g. by \autoref{SubcategoryOfE}).
    
    \item We show that $\mathbf{C}$ is suitable for our Fra\"{i}ss\'e theory, i.e. is directed, has amalgamation and wide families of edge-witnessing and star-refining morphisms.
    Then \autoref{cor:FraisseLimit} applies.
    
    Every category $\mathbf{C}$ we consider will have a terminal object, making directedness an immediate consequence of amalgamation.
    We will often show that $\mathbf{C}$ is clique-closed, and so it is enough to show amalgamation for functions (\autoref{cor:cliqueclosed}).
    That edge-witnessing and star-surjective morphisms are wide often follows from the edge-splitting construction (\autoref{EdgeSplitting}) or from being clique-closed (\autoref{CliqueClosedEdgeWitnessing}).

    \item If possible, we characterize lax-Fra\"{i}ss\'e and Fra\"{i}ss\'e sequences in $\mathbf{C}$.

    By \autoref{cor:FraisseLimit}, every lax-Fra\"{i}ss\'e sequence has an edge-witnessing and a star-refining subsequence.
    Sometimes this is enough, but sometimes we find other (lax-closed) wide ideals $I$ in $\mathbf{C}$, so every (lax-)Fra\"{i}ss\'e sequence has also a subsequence in $I$ by \autoref{IdealFraisseSubsequence}.
    Also note that it is enough to prove lax-absorption by \autoref{DirectedCofinal} -- cofinality follows from directedness.

    \item We characterize the Fra\"{i}ss\'e limit $X$ (\autoref{dfn:FraisseLimit}), and possibly show that every sequence in $\mathbf{C}$ with spectrum $X$ is lax-Fra\"{i}ss\'e.

    To do so we translate between properties of a sequence in $\mathbf{C}$ and topological properties of its spectrum (e.g. we use \autoref{ContinuumTreeLimit}).
    Either known properties of Fra\"{i}ss\'e sequence will translate to a topological characterization of a particular space, or we will a priori have a particular space $X$ in mind and show that every corresponding sequence is lax-Fra\"{i}ss\'e.
    In the latter case it is then enough to find any sequence in $\mathbf{C}$ with spectrum $X$, e.g. with help of \autoref{SequenceForSpace}.
\end{enumerate}

Below, we apply this reasoning to various quite simple categories of graphs and we obtain their Fra\"{i}ss\'e limits, namely we obtain the Cantor space, the arc, the Cantor fan, and the Lelek fan.

\subsection{The Cantor space}\label{ss.Cantor}

Here we start with the simplest possible example, the full subcategory of $\mathbf{B}$ whose objects are \emph{discrete} graphs $D$, where the edge relation is just equality on $D$. We denote this subcategory by $\mathbf{D}$.

Note that any ${\sqsupset}\in\mathbf{S}_G^D$ is automatically a function on $G$ when $D\in\mathbf{D}$ -- then $c,d\sqsupset g$ implies $c\sqcap d$, by edge-preservation, and hence $c=d$, by discreteness. A function is co-injective precisely when it is surjective, so morphisms of $\mathbf{D}$ are precisely surjective functions. 

Recall that a \emph{Stone space} is a $\mathsf{T}_0$ compactum with a clopen basis (which is then $\mathsf{T}_2$).

\begin{prp}\label{Stone}
Let $(G_n,{\sqsupset}_n^m)$ be a sequence in $\mathbf D$  with induced poset $\PP$. Then $\mathsf{S}\mathbb{P}$ is a Stone space.
\end{prp}

\begin{proof}
As noted above, $(G_n)\subseteq\mathbf{D}$ implies that each $\sqsupset^m_n$ is a function and hence that $p^\leq\cap G_m$ is a singleton, for each $p\in G_n$. It follows that each filter and, in particular, each minimal selector in $\mathbb{P}$ contains at most one vertex from $G_m$, and so is a thread. The corresponding basic open sets $(p^\in)_{p\in G_m}$ are thus disjoint and hence form a clopen cover of $\mathsf{S}\mathbb{P}$. As $m$ was arbitrary, it follows that $(p^\in)_{p\in\mathbb{P}}$ is a clopen basis of $\mathsf{S}\mathbb{P}$ and hence $\mathsf{S}\mathbb{P}$ is a Stone space.
\end{proof}

Note that the singleton graph is a terminal object in $\mathbf{D}$, and that $\mathbf{D}$ has pullbacks given by the usual fibre-product, so $\mathbf{D}$ has amalgamation.
$\mathbf{D}$ is also trivially clique-closed as the clique functor is isomorphic to the identity functor.
Finally, note that every function from a discrete graph is trivially star-refining, while every function onto a discrete graph is trivially edge-witnessing. 
Hence, we have a Fra\"{i}ss\'e limit by \autoref{cor:FraisseLimit}.
The key extra ingredient for a sequence to be Fra\"iss\'e is anti-injectivity.

\begin{thm}\label{DFraisseEquivalents}
Let $(G_n,{\sqsupset}_n^m)$ be a sequence in $\mathbf D$  with induced poset $\PP$. The following are equivalent.
\begin{enumerate}
 \item\label{DFraisse} $(\sqsupset^m_n)$ is a Fra\"iss\'e sequence in $\mathbf{D}$;
 \item\label{DLaxFraisse} $(\sqsupset^m_n)$ is a lax-Fra\"iss\'e sequence in $\mathbf{D}$;
 \item\label{DAntiInjective} $(\sqsupset^m_n)$ has an anti-injective subsequence;
 \item\label{DCantor} $\mathsf{S}\mathbb{P}$ is homeomorphic to the Cantor space.
\end{enumerate}
\end{thm}

\begin{proof}
\ref{DFraisse}$\Leftrightarrow$\ref{DLaxFraisse}: Since all morphisms in the category of interest are total surjective functions if ${\sqsupset\circ\sqni}\subseteq{\dashv}$ for ${\sqsupset},{\sqni},{\dashv}\in\mathbf D$, then we have equality. This shows that every lax-Fra\"iss\'e sequence is Fra\"iss\'e.
 
\ref{DFraisse}$\Rightarrow$\ref{DAntiInjective}: For every $G \in \mathbf{D}$ we may consider the projection $\pi$ of  $H = G \times \{0, 1\}$, considered as a discrete graph, onto $G$.
 This is immediately seen to be anti-injective.
 Hence, the anti-injective ideal is wide in $\mathbf{D}$, and every Fra\"{i}ss\'e sequence has an anti-injective subsequence by \autoref{IdealFraisseSubsequence}.

\ref{DAntiInjective}$\Rightarrow$\ref{DFraisse}: Take any $\pi\in\mathbf{D}^{G_j}_E$. If $(\sqsupset^m_n)$ has an anti-injective subsequence then, in particular, we can find $k>j$ such that $\sqsupset^j_k$-preimages always have more vertices than than $\pi$-preimages, i.e. $|d^{\sqsupset^j_k}|\geq|d^\pi|$, for all $d\in G_j$. For each $d\in G_j$, we thus have a map $\phi_d$ from $d^{\sqsupset^j_k}$ onto $d^\pi$. Putting these together, we obtain a map $\phi=\bigcup_{d\in D}\phi_d$ from $G_k$ onto $E$ with ${\sqsupset^j_k}={\pi\circ\phi}$. This shows that $(\sqsupset^j_k)$ is indeed a Fra\"iss\'e sequence in $\mathbf{D}$.

\ref{DAntiInjective}$\Leftrightarrow$\ref{DCantor}: We already know that $\mathsf{S}\mathbb{P}$ is a Stone space, by \autoref{Stone}. As $\mathsf{S}\mathbb{P}$ is also second countable/metrizable, it will be homeomorphic to the Cantor space precisely when it is perfect. This is equivalent to $(\sqsupset^m_n)$ having an anti-injective subsequence, by \autoref{prp:Summary1}\ref{CaseAntiInj1}.
\end{proof}

\subsection{The Arc}\label{ss.arc}

Let $\mathbf{P}$ be the full subcategory of $\mathbf{B}$ whose objects are paths, i.e.
\[
\mathbf{P}=\{P\in\mathbf{G}:P\text{ is a path}\}.
\]
The first thing to observe is that this subcategory is clique-closed. Recall that $\mathsf EG$ is the set of ends (i.e. vertices of degree at most $1$) of a graph $G$.

\begin{prp}\label{PCliqueClosed}
$\mathbf{P}$ is clique-closed.
\end{prp}

\begin{proof}
Any path $P$ is of the form $\{p_0 \sim p_1 \sim \cdots \sim p_n\}$, and therefore
\[
\mathsf XP=\{\{p_0\}\sim \{p_0,p_1\}\sim \{p_1\}\sim\{p_1,p_2\}\sim\cdots\sim \{p_{n-1},p_n\}\sim\{p_n\}\}.
\]
This show that $\mathbf P$ is a full subcategory of the category of trees.  \autoref{TreesCliqueClosed} gives the thesis.
 \end{proof}

Some notation is needed. Every pair of vertices $q$ and $r$ in a path $P$ defines a \emph{subpath} $[q,r]$ defined to be the smallest connected subset of $P$ containing both $q$ and $r$. Equivalently, $[q,r]$ is the unique connected subset of $P$ with $\mathsf{E}[q,r]=\{q,r\}$. Similarly, we let $(q,r)$ denote the smallest subset connecting $\{q,r\}$, i.e. $(q,r)=[q,r]\setminus\{q,r\}$. For a path $P$ and $Q\subseteq P$, we say that $b\in P$ separates $Q\subseteq P$ precisely when $b\in[q,r]\setminus Q$, for some $q,r\in Q$.

\begin{lemma}\label{LemmaMonotoneStrictPreimages}
Let ${\sqsupset}\in\mathbf{P}^Q_P$ be monotone. Then $q_\sqsupset$ is connected for every $q\in Q$.
\end{lemma}

\begin{proof}
 By monotonicity, we know that the non-strict preimage $q^\sqsupset$ is connected. If $q_\sqsupset$ were separated by some $p$, necessarily in $q^\sqsupset\setminus q_\sqsupset$, then we would have $q'\neq q$ with $p^\sqsubset=\{q,q'\}$. As $\sqsupset$ is co-injective, we would then also have $p'\in P$ with $p'^{\sqsubset}=\{q'\}$, which must therefore be separated from $p$ by some vertices of $q_\sqsupset$. But then $q'^\sqsupset$ would not be connected, contradicting the monotonicity of $\sqsupset$.
\end{proof}

\begin{rmk}
Note that both the fact that $\sqsupset$ is co-injective and that we are dealing with paths (and not just trees) are necessary for the above lemma, as the following two examples show.

Let $Q=\{a,b\}$,  $P=\{x\sim y\sim z\}$. Consider ${\sqsupset}\in \mathbf S_{P}^Q$ be the surjective monotone morphism obtained by $x,z\sqsubset a$ and $y\sqsubset a,b$. Then $a_{\sqsupset}=\{x,z\}$ is not connected. When one wants to fix co-bijectivity, one adds a vertex $w$ to $P$ by connecting it only to $y$ (obtaining therefore the claw graph, or the letter `T'). Extending $\sqsupset$ to $\sqni$  by setting $w\sqni b$, one obtains a co-bijective monotone morphism between trees such that $a_{\sqni}$ is the disconnected set $\{x,z\}$.
\end{rmk}

Here we focus on the subcategory $\mathbf{A}$ of monotone  morphisms, for the whole category $\mathbf{P}$, see \autoref{ss.pseudoarc}. Our goal here is to prove that $\mathbf{A}$ has amalgamation and hence Fra\"iss\'e sequences, which can be characterised in an analogous manner to \autoref{DFraisseEquivalents}.  

Next proposition formalizes \autoref{fig:typical}, which describes a typical morphism in $\mathbf A$. 
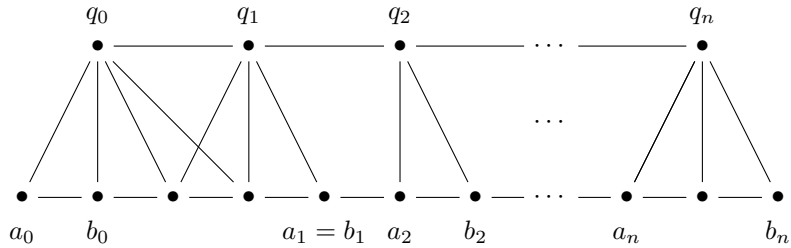
\begin{figure}[!ht]
    \centering
    \begin{tikzpicture}[
            y={(0, 0.5cm)},
            mylabel/.style = {label={[labelstyle]#1}},
            labelstyle/.style = {
                text height = 1.5ex,
                text depth = 0.25ex,
            }
        ]
		\node (q0) at (-4,0) [mylabel=$q_0$]{$\bullet$};
		\node (q1) at (-2,0) [mylabel=$q_1$]{$\bullet$};
		\node (q2) at (0,0) [mylabel=$q_2$]{$\bullet$};
		\node (qx) at (2,0) { $\cdots$ };
		\node (qn) at (4,0) [mylabel=$q_n$]{$\bullet$};
		
		\node (dots) at (2,-2) {$\cdots$};
		\node (p0) at (-5,-4) [mylabel=below:{$a_0$}] {$\bullet$};
		\node (p1) at (-4,-4) [mylabel=below:{$b_0$}] {$\bullet$};
		\node (p2) at (-3,-4) {$\bullet$};
		\node (p3) at (-2,-4) {$\bullet$};
		\node (p4) at (-1,-4) [mylabel=below:{$a_1=b_1$}] {$\bullet$};
		\node (p5) at (0,-4) [mylabel=below:{$a_2$}] {$\bullet$};
		\node (p6) at (1,-4) [mylabel=below:{$b_2$}] {$\bullet$};
		\node (p7) at (3,-4) [mylabel=below:{$a_n$}] {$\bullet$};
		\node (p8) at (4,-4) {$\bullet$};
		\node (p9) at (5,-4) [mylabel=below:{$b_n$}] {$\bullet$};
		\node (px) at (2,-4) {$\cdots$};
		
		\draw (q0)--(q1)--(q2)--(qx)--(qn);
		\draw (p0)--(p1)--(p2)--(p3)--(p4)--(p5)--(p6)--(px)--(p7)--(p8)--(p9);
		
		\draw (p0)--(q0);
		\draw (p1)--(q0);
		\draw (p2)--(q0);
		\draw (p2)--(q1);
		\draw (p3)--(q1);
		\draw (p3)--(q0);
		\draw (p4)--(q1);
		\draw (p5)--(q2);
		\draw (p6)--(q2);
		\draw (p7)--(qn);
		\draw (p8)--(qn);
		\draw (p9)--(qn);
		\draw (p7)--(qn);
    \end{tikzpicture}
    				
    \caption{A typical morphism in $\mathbf A$.}
    \label{fig:typical}
\end{figure}

\begin{prp}\label{prop:MonotoneStructure}
Let $Q=\{q_0\sim q_1\sim\cdots\sim q_n\}$ be a path, and let ${\sqsupset}\in\mathbf A_P^Q$. Then there exists a unique sequence $a_0,b_0, \ldots,a_n,b_n\in P$ such that 
\begin{enumerate}
\item $[a_i,b_i]=(q_i)_{\sqsupset}\neq\emptyset$,
\item $\mathsf EP=\{a_0,b_n\}$, and
\item if $i<j<k$ then $a_j\in [a_i,a_k]$, or equivalently that for every $i$ with $0<i<n$ we have $[a_i,b_i]\subseteq (b_{i-1},a_{i+1})$.
\end{enumerate}
As a consequence for every $i<n$ we have that $(b_i,a_{i+1})=\{q_i,q_{i+1}\}_{\sqsupset}$ (which might be empty).
\end{prp}
\begin{proof}
We will define the points $a_i$ and $b_i$ and then show that they satisfy the thesis. By \autoref{LemmaMonotoneStrictPreimages}, each $(q_i)_{\sqsupset}$ is nonempty and connected, hence it is a subpath $[a_i,b_i]$. 

Let us show that one of $a_0$ and $b_0$ must be an end. Since $Q\setminus \{q_0\}$ is connected, so is $P\setminus (q_0)_\sqsupset=(Q\setminus\{q_0\})^\sqsupset$ by monotonicity. Hence $P$ is the disjoint union of the two subpaths $[a_0,b_0]$ and $P\setminus (q_0)_\sqsupset$, so $[a_0,b_0]$ must contain an end of $P$. We can assume without loss of generality that $a_0$ is an end, and the same proof shows that we can assume that $b_n$ is an end too and that $\{a_0,b_n\}=\mathsf EP$.

Generalising the above reasoning, one can, for every $i$, write $P$ as the disjoint union of the possibly empty subpaths $P=[q_0,q_i)^{\sqsupset}\cup(q_i)_{\sqsupset}\cup(q_i,q_n]^{\sqsupset}$. (Note that by edge-preservation $[q_0,q_i)^{\sqsupset}$ and $(q_i,q_n]^{\sqsupset}$ are disjoint, each contain an end, and therefore $(q_i)_{\sqsupset}=[a_i,b_i]$ is truly in between). If $0<i<n$ then $a_{i-1}\in [q_0,q_i)^{\sqsupset}$ and $b_{i+1}\in (q_i,q_n]^{\sqsupset}$, which implies that $[a_i,b_i]\subseteq (b_{i-1},a_{i+1})$.

For the last statement, we swap $a_i$ and $b_i$ if $b_i\in (a_0,a_i)$. (In this way, the points $a_0,b_0,\ldots,a_n,b_n$ are truly ordered in the path, matching \autoref{fig:typical}.) Fix now $c\in (b_i,a_{i+1})$. If $c\sqsubset q_j$ for some $j<i$, then $b_i$ separates the connected $q_j^\sqsupset \ni a_j,c$, which is a contradiction. Likewise, we cannot have that $i+1<j$. Hence $c^{\sqsubset}\subseteq\{q_i,q_{i+1}\}$, but since $c\notin [a_i,b_i]\cup [a_{i+1},b_{i+1}]$, we must have that $c^{\sqsubset}=\{q_i,q_{i+1}\}$.
\end{proof}
It follows from \autoref{prop:MonotoneStructure} that  ${\sqsupset}\in\mathbf{A}_P^Q$ satisfies that for every $e\in\mathsf EQ$ then  $\mathsf EP\cap e_\sqsupset\neq\emptyset$. In particular $\sqsupset$ is
\begin{align*}
 \tag{End-Preserving}e\in\mathsf{E}P\qquad&\Rightarrow\qquad e^\sqsubset\subseteq\mathsf{E}Q,\text{ and }\\
 \tag{End-Surjective}f\in\mathsf{E}Q\qquad&\Rightarrow\qquad f^\sqsupset\cap\mathsf{E}P\neq\emptyset.
\end{align*}

\begin{obs}\label{MorphismsToSubpaths}
From the structure described in \autoref{prop:MonotoneStructure}, it follows that for ${\sqsupset}\in\mathbf{A}^Q_P$ and $q\sim r\in Q$ there exists uniquely determined $q_r$ and $r_q$ such that $q_r^{\sqsubset}=\{q\}$, $r_q^\sqsubset=\{r\}$ and for every $p\in (q_r,r_q)$ we have $p^\sqsubset=\{q,r\}$. (In the case $q=q_i$ and $r=q_{i+1}$ in a given enumeration, then $q_r=b_i$ and $r_q=a_{i+1}$).
\end{obs}

\begin{prp}\label{ACliqueClosed}
 $\mathbf{A}$ is clique-closed.
\end{prp}

\begin{proof}
We need to check ~\autoref{defin:cliqueclosed}. By Propositions~\ref{PCliqueClosed} and \ref{CliqueTransformation}, it only remains to show that $\mathsf X^{\sqsupset}\in\mathbf A$ whenever ${\sqsupset}\in\mathbf{A}^Q_P$. Since $\mathsf X^\sqsupset$ is a surjective morphism between two connected acyclic graphs, by \autoref{SurjectiveImpliesEdgeSurjective}, it is edge-surjective, therefore to check monotonicity, by \autoref{MonotoneEdgeSurjective}, it is enough to check that $\mathsf{X}^\sqsupset$-preimages of vertices in $\mathsf{X}Q$ are connected.

Take $D\in\mathsf{X}Q$. If $D=\{q\}$ then, $q_\sqsupset$ is connected (\autoref{LemmaMonotoneStrictPreimages}) and hence so is $\mathsf{X}^{\sqsupset-1}\{D\}=\mathsf{X}q_\sqsupset$. If $D=\{q,r\}$, for distinct $q,r\in Q$, then \autoref{MorphismsToSubpaths} yields $q_r\in q_\sqsupset$ and $r_q\in r_\sqsupset$ such that $\{q,r\}_\sqsupset=(q_r,r_q)$. The definition of $\mathsf X$ readily gives that  $\mathsf{X}^{\sqsupset-1}\{q,r\}=\mathsf{X}((q^r,r^q))$, which is connected. Thus $\mathsf{X}^\sqsupset\in\mathbf{A}^{\mathsf{X}Q}_{\mathsf{X}P}$.
\end{proof}

To prove $\mathbf{A}$ has amalgamation, it now suffices to consider functions in $\mathbf{A}$, by \autoref{cor:cliqueclosed}. A more detailed analysis of the morphisms in terms of their `types' will also allow us to prove amalgamation directly, even for relations.

Fix paths $P$ and $Q$. For any ${\sqsupset}\subseteq Q\times P$, we define $\mathsf{t}_\sqsupset:\mathsf{X}Q\rightarrow\omega$ by
\[
\mathsf{t}_\sqsupset(C)=|C_\sqsupset|.
\]
Certain properties of $\sqsupset$ can already be extracted from $\mathsf{t}_\sqsupset$, e.g.
 \begin{itemize}
 \item $\sqsupset$ is co-injective precisely when $\mathsf{t}_\sqsupset(\{q\})\geq1$, for all $q\in Q$.
 \item $\sqsupset$ is strictly anti-injective precisely when $\mathsf{t}_\sqsupset(\{q\})\geq2$, for all $q\in Q$.
 \item $\sqsupset$ is edge-witnessing precisely when $\mathsf{t}_{\sqsupset}(\{q,r\})\geq1$ for all adjacent $q,r\in Q$.
 \item $\sqsupset$ is a partial function precisely when $\mathsf{t}_\sqsupset(C)\neq0$ implies $C=\{q\}$, for some $q\in Q$.
 \end{itemize}
If ${\sqsupset}\in\mathbf{A}^Q_P$ then, moreover,
 \begin{itemize}
 \item $\sqsupset$ is star-refining precisely when $\mathsf{t}_{\sqsupset}(\{q,r\})\geq2$ for all adjacent $q,r\in Q$.
 \end{itemize} 
It could still make sense to consider a function $\mathsf t_{\sqsupset}$ associated to a morphism ${\sqsupset}\subseteq Q\times P$ even outside the category of paths. In this case, some of the above properties are true in general, while others need some additional assumptions (for example, the characterisation of edge-witnessing holds only when $P$ is triangle-free).

As our focus is on co-injective morphisms, we define types as follows.

\begin{dfn}
 We call $t:\mathsf{X}Q\rightarrow\omega$ a \emph{type} if $t(\{q\})\geq1$, for all $q\in Q$.
\end{dfn}

\begin{prp}\label{TypeExistence}
 For each path $Q$ and each type $t:\mathsf{X}Q\rightarrow\omega$, there are a path $P$ and ${\sqsupset}\in\mathbf{A}^Q_P$ with $\mathsf{t}_\sqsupset=t$.
\end{prp}

\begin{proof}
 Take disjoint $P_q\in\mathbf{P}$ with $|P_q|=t(\{q\})$, for all $q\in Q$. Further choose $q_r\in\mathsf{E}P_q$, for all adjacent $q,r\in Q$, such that $P_q=[q_r,q_s]$ when $q \sim r\neq s \sim q$. Then take $P_{q,r}\in\mathbf{P}$ with $\mathsf{E}P=\{q_r,r_q\}$, for all adjacent $q,r\in Q$, further ensuring that $P'_{q,r}=P_{q,r}\setminus\{q_r,r_q\}$ are disjoint, both from each other and the $(P_q)_{q\in Q}$, and that $|P'_{q,r}|=t(\{q,r\})$, for all adjacent $q,r\in Q$. Thus $P=\bigcup_{C\in\mathsf{X}Q}P_C$ is a path on which we can define ${\sqsupset}\in\mathbf{A}^Q_P$ by $p^\sqsubset=\{q\}$, for all $p\in P_q$, and $p^\sqsubset=\{q,r\}$, whenever $p\in P'_{q,r}$. Then $|q_\sqsupset|=|P_q|=t(\{q\})$, for all $q\in Q$, and $|\{q,r\}_\sqsupset|=|P'_{q,r}|=t(\{q,r\})$, for all adjacent $q,r\in Q$, showing that $\mathsf{t}_\sqsupset=t$.
\end{proof}

The natural order on types, given by pointwise domination, corresponds to the factorisation order on morphisms.

\begin{prp}\label{TypeFactorisation}
 If ${\sqsupset}\in\mathbf{A}^R_P$ and ${\sqni}\in\mathbf{A}^R_Q$ then
 \[\mathsf{t}_\sqsupset\leq\mathsf{t}_{\sqni}\qquad\Leftrightarrow\qquad\text{ there is }\,{\dashv}\in\mathbf{A}^P_Q\ ({\sqsupset}\circ{\dashv}={\sqni}).\]
Moreover, if $\mathsf{t}_\sqsupset(\{r, s\}) = 0$ implies $\mathsf{t}_{\sqni}(\{r, s\}) = 0$ for every adjacent $r, s \in R$, then we can choose $\dashv$ to be a function.
\end{prp}

\begin{proof}
 The $\Leftarrow$ part is immediate from the fact that $\mathsf{t}_{\sqsupset\circ\dashv}\geq\mathsf{t}_\sqsupset$, for any ${\dashv}\in\mathbf{I}^P_Q$. Conversely, assume $\mathsf{t}_\sqsupset\leq\mathsf{t}_{\sqni}$. In particular, for each $r\in R$, this means $|r_\sqsupset|\leq|r_{\sqni}|$ so we have a function ${\dashv_r}\in\mathbf{A}^{r_\sqsupset}_{r_{\sqni}}$ with $r_s\dashv_rr'_s$, for all $s\in r^\sim$, where $r_s$ and $r'_s$ are the vertices defined from $\sqsupset$ and $\sqni$ respectively as in \autoref{MorphismsToSubpaths}. For all adjacent $r,s\in R$, this also means $|\{r,s\}_\sqsupset|\leq|\{r,s\}_{\sqni}|$ so we also have ${\dashv_{r,s}}\in\mathbf{A}^{[r_s,s_r]}_{[r'_s,s'_r]}$ again with $r'^{\vdash_{r,s}}_s=\{r_s\}$ and $s'^{\vdash_{r,s}}_r=\{s_r\}$, but this time with $q^{\vdash_{r,s}\sqsubset}=\{r,s\}$, for all $q\in(r'_s,s'_r)$. Indeed, if $(r_s,s_r)=\emptyset$ then $r_s\sim s_r$ and so we may set $q^{\vdash_{r,s}}=\{r_s,s_r\}$, for all $q\in(r'_s,s'_r)$. On the other hand, if $(r_s,s_r)\neq\emptyset$ then we can take $\dashv_{r,s}$ to be a function from $(r'_s,s'_r)$ onto $(r_s,s_r)$. Now ${\dashv}=\bigcup_{C\in\mathsf{X}R}\dashv_C$ satisfies ${\sqsupset\circ\dashv}={\sqni}$.
 
 For the moreover part, when $\mathsf{t}_\sqsupset(\{r, s\}) = 0$ implies $\mathsf{t}_{\sqni}(\{r, s\}) = 0$ for every adjacent $r, s \in R$, the second part of the proof does not apply, and therefore the union of the functions $\dashv_r$, for $r\in R$, does the job.
\end{proof}

The existence of sub-factorisations ${\sqsupset\circ\dashv}\subseteq{\sqni}$ could also be characterised in terms of $\mathsf{t}_\sqsupset$ and $\mathsf{t}_{\sqni}$. For our purposes, the following sufficient condition will suffice.

\begin{prp}\label{TypeSubFactorisation}
 If ${\sqsupset}\in\mathbf{A}^R_P$, ${\sqni}\in\mathbf{A}^R_Q$, $|P|\leq|Q|$ and $|P|\leq|\{r,s\}_{\sqni}|$, for all adjacent $r,s\in R$, then we have $\phi\in(\mathbf{A}\cap\mathbf{F})^P_Q$ with ${\sqsupset}\circ\phi\subseteq{\sqni}$.
\end{prp}

\begin{proof}
 If $|R|=1$ then, as $|P|\leq|Q|$, we have monotone $\phi$ mapping $Q$ onto $P$. Otherwise, for each $r\in R$, pick $r'\in r_\sqsupset$, the only proviso being that $r'\in\mathsf{E}P$ whenever $r\in\mathsf{E}R$. Set $\phi(q)=r'$ whenever $q\in r_{\sqni}$. As $|(r',s')|\leq|P|\leq|\{r,s\}_{\sqni}|$, we can extend $\phi$ to a monotone function mapping the path $[r_s,s_r]$ (where $r_s$ and $s_r$ are chosen by \autoref{MorphismsToSubpaths} for $r$, $s$, and $\sqni$) onto $[r',s']$, for all adjacent $r,s\in R$. The resulting function $\phi\in(\mathbf{A}\cap\mathbf{F})^P_Q$ then satisfies ${\sqsupset}\circ\phi\subseteq{\sqni}$.
\end{proof}

Amalgamation now follows from the fact that types are directed.

\begin{prp}\label{prp:amalgarc}
 $\mathbf{A}$ and $\mathbf{A} \cap \mathbf{F}$ have amalgamation.
\end{prp}

\begin{proof}
 Given any ${\sqsupset}\in\mathbf{A}^P_Q$ and ${\sqni}\in\mathbf{A}^P_R$, we have a type $t\geq\mathsf{t}_\sqsupset,\mathsf{t}_{\sqni}$ (e.g. take $t=\max\{\mathsf{t}_\sqsupset,\mathsf{t}_{\sqni}\}$). By \autoref{TypeExistence}, we have ${\sqnii}\in\mathbf{A}^P_S$ with $\mathsf{t}_{\sqnii}=t$. By \autoref{TypeFactorisation}, we have ${\dashv}\in\mathbf{A}^Q_S$ and ${\Dashv}\in\mathbf{A}^Q_R$ with ${\sqsupset\circ\dashv}={\sqnii}={\sqni\circ\Dashv}$. Again by \autoref{TypeFactorisation}, if ${\sqsupset}$ and ${\sqni}$ were functions, ${\dashv}$ and ${\Dashv}$ can be chosen to be functions too.
\end{proof}

The singleton graph is a terminal object in $\mathbf{A}$, and so amalgamation implies directedness. By \autoref{EdgeSplitting}, edge-witnessing and star-refining morphisms are wide. Altogether, \autoref{cor:FraisseLimit} applies and we have a Fra\"{i}ss\'e limit. Now we shall completely characterise (lax-)Fra\"{i}ss\'e sequences in $\mathbf{A}$.

A sequence $(G_n,{\sqsupset}_n^m)$ in $\mathbf{A}$ is said to be \emph{nontrivial} if $|G_n|\neq1$, for some $n\in\omega$. By co-injectivity, if $|G|>1$ and ${\sqsupset}\in\mathbf P_{H}^{G}$ is edge-witnessing (or star-refining) then $|H|>|G|$.\footnote{This is true more generally if $G$ is not discrete and we have a co-bijective edge-witnessing morphism, or if $H$ is not discrete and we have a co-bijective star-refining morphism.} Hence any nontrivial sequence of edge-witnessing (or star-refining) morphisms gives that $\lim_n |G_n|=\infty$.

\begin{lemma} \label{AMorphisms}
    In $\mathbf{A}$, every star-refining morphism is edge-witnessing, and every composition of two edge-witnessing morphisms is star-refining.
\end{lemma}
\begin{proof}
    The first part follows from \autoref{TreeMorphismsSummary} and the second part from Propositions~\ref{StarSurjectiveCoEdgeWitnessing} and \ref{EdgevsCoEdgeWitnessing}.
\end{proof}

\begin{thm}\label{AFraisseEquivalents}
Let $(G_n,\sqsupset_n^m)$ be a sequence in $\mathbf A$ with induced poset $\PP$. The following are equivalent.
\begin{enumerate}
 \item\label{ALaxFraisse} $(\sqsupset^m_n)$ is a lax-Fra\"iss\'e sequence in $\mathbf{A}$;
 \item\label{AEdgeWitnessing} $(\sqsupset^m_n)$ is nontrivial, and it has an edge-witnessing subsequence;
 \item\label{AStarSurjective} $(\sqsupset^m_n)$ is nontrivial, and it has a star-refining subsequence;
 \item\label{AArc} $\mathsf{S}\mathbb{P}$ is homeomorphic to the arc.
\end{enumerate}
A lax-Fra\"iss\'e sequence is Fra\"iss\'e precisely when it also has a strictly anti-injective subsequence.
\end{thm}

\begin{proof}
    Clearly every lax-Fra\"{i}ss\'e sequence in $\mathbf{A}$ is nontrivial and by \autoref{cor:FraisseLimit} has an edge-witnessing and a star-refining subsequence.
    Similarly, it is easy to see that strictly anti-injective morphisms are wide and so form a wide ideal in $\mathbf{A}$ by \autoref{Ideals}.
    Hence, every Fra\"{i}ss\'e sequence has a strictly anti-injective subsequence by \autoref{IdealFraisseSubsequence}.
    Together with a use of \autoref{AMorphisms} we have \ref{ALaxFraisse}$\Rightarrow$\ref{AEdgeWitnessing}$\Leftrightarrow$\ref{AStarSurjective}.

    To show \ref{AEdgeWitnessing}$\Rightarrow$\ref{ALaxFraisse} take ${\sqsupset}\in\mathbf{A}^{G_m}_P$. If $(\sqsupset^m_n)$ has an edge-witnessing subsequence then we have $n>m$ such that $|P|\leq|G_n|$ (since the sequence is nontrivial) and $\sqsupset^m_n$ can be written as a composition of $|P|$ edge-witnessing morphisms so $|P|\leq|\{g,h\}_{\sqsupset^m_n}|$, for all adjacent $g,h\in G_m$. These are precisely the hypotheses of \autoref{TypeSubFactorisation}, which then applies, giving $\phi\in(\mathbf{A}\cap\mathbf{F})^P_Q$ with ${\sqsupset}\circ\phi\subseteq{\sqsupset^m_n}$. Hence $(\sqsupset^m_n)$ is lax-Fra\"iss\'e.
    
    Similarly, if $(\sqsupset_n^m)$ has both nontrivial edge-witnessing and strictly anti-injective subsequences, then for any ${\sqsupset}\in\mathbf{A}^{G_m}_P$ there is $n > m$ such that $\mathsf{t}_{\sqsupset^m_n}\geq\mathsf{t}_\sqsupset$.
    By \autoref{TypeFactorisation}, we then have ${\dashv}\in\mathbf{A}^P_Q$ with ${\sqsupset}\circ{\dashv}={\sqni}$, showing $(\sqsupset_n^m)$ is Fra\"iss\'e.

    The implication \ref{AArc}$\Rightarrow$\ref{AEdgeWitnessing} (and \ref{AStarSurjective}) follows directly from \autoref{ContinuumTreeLimit}.
    On the other hand, as all lax-Fra\"{i}ss\'e sequences have homeomorphic spectra by \autoref{cor:FraisseLimit}, it suffices to find a single lax-Fra\"{i}ss\'e sequence whose spectrum is the arc.
    By \autoref{SequenceForSpace} it is enough to find a suitable sequence of minimal open covers of $[0, 1]$.
    Such sequence was described in \cite[Example 2.16]{BartosBiceV.Compacta}: consider 
     \[
     C_n = \{\mathrm{int}([(k - 1)/2^{n+1}, (k + 1)/2^{n + 1}]): 1 \leq k \leq 2^{n + 1} - 1\}
     \]
    (see also \autoref{fig:comparison}). Then each $C_n$ with the overlap relation is a path, and the inclusion relation restricted to $C_n\times C_{n+1}$ is monotone and edge-witnessing.
    Hence the induced sequence in $\mathbf{A}$ is lax-Fra\"{i}ss\'e by \ref{AEdgeWitnessing}$\Rightarrow$\ref{ALaxFraisse}.
\end{proof}

\begin{xpl}
Let us give a concrete example of a lax-Fra\"isse sequence (therefore leading to the arc) which is not Fra\"iss\'e. Let $G_0$ be a two-point path, and for every $n \in \omega$, let $G_{n + 1}$ be $\mathsf{X}{G_n}$ and let $\sqsupset^n_{n+1}$ be the natural map $\ni_{G_n}$.
 By \autoref{ACliqueClosed}, $(\sqsupset^m_n)$ is a sequence in $\mathbf{A}$.
 It is easy to see that $\mathsf{t}_{\ni_{G_n}}$ is a constant function taking the value $1$, and so $(\sqsupset^m_n)$ is edge-witnessing and anti-injective, but not star-refining.
 Also, $g_{\sqsupset^m_n}$ is a singleton for every $g \in G_m$ and $n \geq m$, so $(\sqsupset^m_n)$ does not have a strictly anti-injective subsequence.
 Hence, $(\sqsupset^m_n)$ is lax-Fra\"{i}ss\'e in $\mathbf{A}$ but not Fra\"{i}ss\'e.
\end{xpl}

\subsubsection{More sequences leading to the arc} \label{s.MoreArcs}

As the following example shows, we do not have a Fra\"{i}ss\'e sequence in the category of paths and monotone surjective relations in $\mathbf{S}$, as the amalgamation property fails.
However, we show sequences in this category admin a \emph{co-bijective modification}, and often have the arc as the spectrum of the induced poset.
This will be used later when analyzing certain categories of fans.

\begin{xpl}
   Let $F = \{f, f'\}$, $G = \{g\}$, and $H = \{h, h', h''\}$ be paths.
   Let ${\sqsupset} \in \mathbf{S}^F_G$ be the only surjective morphism, and let ${\sqni} \in \mathbf{S}^F_H$ be the function satisfying $f \sqni h, h'$ and $f' \sqni h''$.
   Note that $\sqsupset$ and $\sqni$ are monotone and surjective.
   There is no path $P$ with surjective morphisms ${\dashv} \in \mathbf{S}^G_P$, ${\Dashv} \in \mathbf{S}^H_P$ such that ${\sqsupset} \circ {\dashv} = {\sqni} \circ {\Dashv}$.
   Otherwise, there is $p \in P$ such that $p \vDash h$, and hence $p^\vDash \subseteq \{h, h'\}$ and $p^{\vDash\sqin} = \{f\}$.
   On the other hand, $p^{\vdash\sqsubset} = \{f, f'\}$ for every $p \in P$.
\end{xpl}

By the following lemma, monotone surjective morphisms between paths are automatically co-bijective at non-ends.

\begin{lemma} \label{OnlyEndpointRedundantLemma}
    Let $F$ and $G$ be paths and let ${\sqsupset} \in \mathbf{S}^G_F$ be monotone and surjective.
    Then for every $g \in G \setminus \mathsf{E} G$ there is $f \in F \setminus \mathsf{E} F$ such that $f^\sqsubset = \{g\}$.
\end{lemma}
\begin{proof}
    Let $g_-$ and $g_+$ be the two neighbours of $g$ and pick $e_-,e_+\in\mathsf EG$ such that $g_-\in [e_-,g]$ and $g_+ \in [g, e_+]$.
    Let $X^- = [e_-,g_-]^{\sqsupset}$ and $X^+ = [g_+, e_+]^{\sqsupset}$.
    Since $\sqsupset$ is monotone, $X_-$ and $X_+$ are connected, and hence they are (non-empty) subpaths of $G$.
    Note that no vertex of $[e_-,g_-]$ is adjacent to a vertex of $[g_+,e_+]$, and so no vertex of $X_-$ is adjacent to a vertex of $X_+$.
    Hence, there is a vertex $f \in G$ strictly between the paths $X_-$ and $X_+$.
    Clearly, $f \notin \mathsf{E} F$, and the only possibility is that $f^{\sqsubset} = \{g\}$.
\end{proof}

For $e$ an end of a path $P$, let $e^{\op}$ denote the opposite end.
\begin{prp} \label{PathCoBijectiveModification}
    Let $(G_n, \sqsupset^m_n)$ be a monotone surjective sequence of paths in $\mathbf S$. Suppose that there is a thread of ends $(e_n)$ of $(G_n,\sqsupset_n^m)$ such that $e_{n+1}{}^{\sqsubset_{n+1}^n}=\{e_n\}$ for every $n$.
    Then there is the largest co-bijective restriction $(H_n,\sqni_n^m)$ of $(G_n,\sqsupset_n^m)$, satisfying
    \begin{enumerate}
    \item\label{lemmacobijarc1} $H_n=[e_n,f_n]$ for some $f_n$ with $f_n\sqcap e_n^{\op}$,
    \item\label{lemmacobijarc2} for every $m$ there is $n\geq m$ such that $f_m\sqsupset_n^m e_n^{\op}$,
    \item\label{lemmacobijarc3} $(\sqni^m_n)$ is a sequence in $\mathbf{A}$, and 
    \item\label{lemmacobijarc4} $(G_n,\sqsupset_n^m)$ and $(H_n,\sqni_n^m)$ have canonically homeomorphic spectra.
    \end{enumerate}
\end{prp}
\begin{proof} 
We construct the sets $H_n$, and then prove that $(H_n,\sqni_n^m)$, where ${\sqni_n^m}={\sqsupset_n^m}\restriction H_n\times H_m$, satisfies the required properties. 

This construction is done by removing `redundant' points from $G_n$. First, declare  $g\in G_n$ to be `red' if $\sqsupset_{n+1}^n$ is not co-bijective at $g$, i.e. there is no $h\in G_{n+1}$ such that $h^{\sqsubset_{n+1}^n}=\{g\}$. These points have to be removed to ensure co-bijectivity. Removing red points may not be enough. In general, if $h\in G_{n+1}$ needs to removed and  $g\in G_n$ is such that co-bijectivity at $g$ is witnessed only by $h$ (that is, the strict preimage $g_{\sqsupset}$ equals $\{h\}$), then $g$ must be removed too. We apply this reasoning inductively: we say that $g\in G_n$ is a `blue' point if the only point in $G_{n+1}$ witnessing co-bijectivity of $\sqsupset_{n+1}^n$ at $g$ is either blue or red. In particular, every blue point $g\in G_m$ is blue because of a unique red point belonging to some later $G_n$.

By \autoref{OnlyEndpointRedundantLemma} and our hypotheses, if $g$ is red then $g=e_n^{\op}$. Furthermore by the same reasons, if $g\in G_{n}$ is blue because of $h\in G_{n+1}$ and $h=e_{n+1}^{\op}$, then $g=e_n^{\op}$, as co-bijectivity of non-ends is witnessed at non-ends and no original end $e_n$ is red or blue.

Let $H_n=[e_n,f_n]$ where $f_n$ is the last point which neither red nor blue. By the above, $f_n\sqcap e_n^{\op}$. We are now going to show that $(H_n,\sqni_n^m)$ is a co-bijective sequence which satisfies \ref{lemmacobijarc2}--\ref{lemmacobijarc4} (\ref{lemmacobijarc1} follows from the definition of $H_n$). 

To show that $\sqni_{n+1}^n$ is co-bijective, pick $h\in H_n$. If $h\neq e_n^{\op}$, then co-bijectivity of $\sqsupset_{n+1}^n$ at $h$ is witnessed by some $g\in G_{n+1}\setminus \{e_{n+1}^{\op}\}\subseteq H_{n+1}$ (\autoref{OnlyEndpointRedundantLemma}). If $h=e_{n}^{\op}$ then, since $e_n^{\op}$ is not red, there is $g\in G_{n+1}$ such that $g^{\sqsubset_{n+1}^n}=\{h\}$. Since $e_n^{\op}$ is not blue, such $g$ can be chosen to be neither blue nor red, i.e. in $H_{n+1}$. Notice that $\sqni_{n+1}^n$ is a co-bijective morphism between paths that is monotone by \autoref{RestrictionsOfMonotone}, i.e. it belongs to $\mathbf A$, and therefore it is end-preserving. This entails that $f_{n+1}{}^{\sqin_{n+1}^n}=\{f_n\}$, and in particular $(f_n)$ that is a thread.

Let us now show \ref{lemmacobijarc2}. Fix $m$. If $f_m=e_m^{\op}$, there is nothing to prove. Else, let $n\geq m$ be minimal such that $e_n^{\op}$ is red. Hence $e_{n+1}^{\op}\sqsubset_{n+1}^n f_n\sqsubset_n^m f_m$. 

For \ref{lemmacobijarc3}, we now show that $(H_n,\sqni_n^m)$ is indeed a sequence. Otherwise (by induction), there are $m,n$ with $\ell = n - m \geq 2$, $g\in H_m$ and $g'\in H_n$ such that $g\sqni_n^m g'$ yet there is no $k$ with $m<k<n$ and $h\in H_k$ such that $g\sqni^m_k h \sqni^k_n g'$. Since $(G_n,\sqsupset_n^m)$ is a sequence and $g\sqsupset^m_n g'$, there are $g_0,\ldots,g_\ell$ such that 
\[
g=g_0\sqsupset_{m+1}^m g_1\sqsupset_{m+2}^{m+1}\cdots\sqsupset_{n-1}^{n-2} g_{\ell-1} \sqsupset_{n}^{n-1} g_\ell=g'.
\]
By our choice of $m$ and $n$, we have that $g_i\notin H_{m+i}$ for any positive $i<\ell$, and hence $g_i=e_{m+i}^{\op}$ for each such $i$. We aim to show $g=f_m$ and $f_{n-1}\sqsupset_{n}^{n-1}g'$. If this is the case, the sequence $g=f_m,f_{m+1},\ldots,f_{n-1},g'$ witnesses the property of $(H_n,\sqni_n^m)$ being is a sequence is not broken by $g$ and $g'$.
 
Suppose $g\neq f_m$. Note that ${\sqni}_{m+1}^m\in\mathbf A$ and $f_{m + 1}{}^{\sqsubset^m_{m+1}} = \{f_m\}$, and therefore there is $h\in [e_{m+1},f_{m+1})$ with $h^{\sqin^m_{m+1}} = \{g\}$. Hence, the set $g^{\sqsupset_{m+1}^m}$ contains both $h$ and $e_{m+1}^{\op}$, but does not contain $f_{m+1}\in (h,e_{m+1}^{\op})$, and so is disconnected. This contradicts the monotonicity of $\sqsupset_{m+1}^m$.

Next, since $e_{n-1}^{\op} \sqsupset_{n}^{n-1} g'$, and $e_n^{\op}$ was either blue or red, the image of $g'$ cannot be just the point $e_{n-1}^{\op}$, and so $g'{}^{\sqsubset^{n-1}_n} = \{e_{n-1}^\op, h\}$ for some $h \sim e_{n-1}^\op$. Clearly, $h = f_{n-1}$. This concludes the proof that $(H_n,\sqni_n^m)$ forms a sequence in $\mathbf{A}$.

We are left to show \ref{lemmacobijarc4}. This will follow from \autoref{DenseRestrictionHomeomorphic} once we show that each $H_n$ is dense. If $H_n=G_n$ there is nothing to prove. Else, if $e_n^{\op}$ is red, then $\sqsupset_{n+1}^n$ is not co-bijective at $e_n^{\op}$, meaning that $H_n{}^{\sqsupset_{n+1}^n}=G_{n+1}$. If $e_n^{\op}$ is blue as witnessed by the red point $e_{k}^{\op}\in G_k$ for some $k>n$, then $G_{k+1}=H_k{}^{\sqsupset_{k+1}^k}\subseteq H_n{}^{\sqsupset_{k+1}^n}$ as $H_k \subseteq H_{n}{}^{\sqsupset_k^n}$.
\end{proof}

We call the restriction $(H_n, \sqni^m_n)$ obtained in the previous proposition the \emph{co-bijective modification} of $(G_n, \sqsupset^m_n)$.

\begin{rmk}
    Note that co-bijective modification can be almost uniquely defined for every monotone surjective sequence of paths $([a_n, b_n], \sqsupset^m_n)$ in $\mathbf{S}$.    
    If $|[a_0, b_0]| \geq 3$, there is a thread $(g_n)$ with $g_{n + 1}{}^{\sqsubset^n_{n+1}} = \{g_n\}$ such that $([a_n, g_n])$ and $([g_n, b_n])$ induce upper restrictions of $([a_n, b_n], \sqsupset^m_n)$, for which \autoref{PathCoBijectiveModification} can be applied separately, and the resulting modifications can be combined.

    If an initial segment or the whole sequence consists of short paths, we use the following observation: for a path $G$ with $\leq 2$ points and ${\sqsupset} \in \mathbf{S}^G_H$ for some $H$, the relation $\sqsupset$ is co-bijective or there is $g \in G$ with $g^{\sqsupset} \supseteq H$.
\end{rmk}

\begin{cor} \label{MoreArcs}
    Let $(G_n, \sqsupset^m_n)$ be a monotone surjective sequence of paths in $\mathbf{S}$ and let $\PP$ be the induced poset.
    If $(\sqsupset^m_n)$ has a star-refining subsequence, then $\Spec{\PP}$ is an arc or a point.
\end{cor}
\begin{proof}
    Let $(H_n, \sqni^m_n)$ be a co-bijective modification of $(G_n, \sqsupset^m_n)$, which exists and is in $\mathbf{A}$ by \autoref{PathCoBijectiveModification} and by the previous remark, and let $\QQ$ be the induced poset, whose spectrum is canonically homeomorphic to $\Spec{\PP}$.
    
    Let $m \leq n$ and suppose $\sqsupset^m_n$ is star-refining.
    We show that $\sqni^m_n$ is also star-refining, and so $(\sqni^m_n)$ has a star-refining subsequence as well.
    For every $h \in H_n$ there is $g \in G_m$ such that $h^\sqcap \sqsubset^m_n g$.
    If $g \in H_m$, we are done.
    Otherwise, $g$ is an end of $G_m$ and $h^\sqcap \cap H_n \sqin^m_n g_+$ where $g_+$ is the unique vertex adjacent to $g$.
    This is because $x^{\sqsubset^m_n} \subseteq \{g, g_+\}$ and so $x^{\sqin^m_n} = \{g_+\}$ for every $x \in h^\sqcap \cap H_n$.
    
    If $(H_n,\sqni^m_n)$ is trivial, its spectrum is a point. Otherwise, by \autoref{AFraisseEquivalents}, $\Spec{\QQ}$ is the arc.
\end{proof}

\subsection{Fans}\label{ss.Fans}

In this section, we work with both graph theoretic and topological fans. Even though the terminology tends to converge, there are some slight differences.

\subsubsection{Topological and graph theoretic fans}
Most of our terminology is taken from \cite{Chara.Fans}.
Recall that a \emph{dendroid} is a continuum $X$ that is \emph{arcwise connected}, i.e. for every $x \neq y \in X$ there is an arc $A \subseteq X$ with endpoints $x, y$, and hereditarily unicoherent (the intersection of every two subcontinua is connected).
It follows that for every two points $x \neq y$ in a dendroid $X$ there is a unique arc in $X$ with endpoints $x, y$, sometimes denoted by $[x, y]$.
Note that a dendrite is exactly a locally connected dendroid (as every Peano continuum is arcwise connected).

A point $x \in X$ of a dendroid is called an \emph{endpoint} if for every arc $A \subseteq X$ such that $x \in A$ we have that $x$ is an endpoint of $A$.
Let $\mathsf{E}X$ denote the set of endpoints of $X$.
A point $x \in X$ of a dendroid is called a \emph{ramification point} if there is a simple triod $T \subseteq X$ such that $x$ is the ramification point of $T$, i.e. there are three arcs $A_i \subseteq X$, $i < 3$, such that $A_i \cap A_j = \{x\}$ for every $i \neq j < 3$.

A \emph{(topological) fan} is a dendroid $X$ that has exactly one ramification point $0$, called the \emph{root}.
The arcs $[0, e]$, $e \in \mathsf{E}X$, are called \emph{spokes}, and they form the unique family of arcs $\mathcal{A}$ such that $\bigcup\mathcal{A} = X$ and $A \cap A' = \{0\}$ for every $A, A' \in \mathcal{A}$ -- the \emph{spoke decomposition} of a (topological) fan.
A dendroid with a spoke decomposition is already a fan.

There are two fans we will mainly focus our attention on:
\begin{itemize}
\item The \emph{Cantor fan} is the space $X = (C\times[0,1])/{\sim}$ where $(x,t)\sim(y,s)$ if and only if the pairs are equal or $s=t=0$, where $C$ is the Cantor space. This can be viewed as a subset of $\mathbb R^2$ obtained by drawing $C$ as a subset of $[0,1]\times\{1\}$ and adding the segment between $(1/2,0)$ and each point of $C$. 

\item A \emph{Lelek fan} is a nondegenerate subcontinuum of the Cantor fan whose endpoints are dense. All Lelek fans are homeomorphic, and it makes therefore sense to speak of \emph{the} Lelek fan (\cite{BulaOver.CantorFans} or \cite{Chara.Lelek}).
\end{itemize}
Let's now focus on graphs.
A (nontrivial) \emph{fan} $F$ is a tree with exactly one node of degree $\geq 3$, the \emph{root} of the fan, denoted by $0_F$.
A finite fan can be obtained by gluing together a certain amount of finite paths by one of their ends.
The corresponding subpaths of $F$, i.e. the subpaths whose ends are the root and an end of $F$, are called \emph{spokes}.
We denote the set of all spokes of $F$ by $*F$.

Note that the singleton graph is not a fan, but the \emph{claw} (i.e. the graph $\{0, a,b,c\}$ where $0$ is connected to $a$, $b$, and $c$) is the weakly terminal object in our category. 

\subsubsection{The fan categories}
We focus on the category of nontrivial fans where morphisms are co-bijective edge-preserving morphisms, i.e. elements of the category $\mathbf B$. By \autoref{SurjectiveImpliesEdgeSurjective}, such morphisms are edge-surjective, i.e. are in $\mathbf E$.

For fans $F$ and $G$, we call ${\sqsupset}\subseteq G\times F$ \emph{spoke-monotone} if it is root-preserving (i.e. $0_F^{\sqsubset}=\{0_G\}$), and the restriction to any spoke $S\subseteq F$ is a monotone relation between $S$ and another spoke $T\subseteq G$ i.e.
\[
S\in\ast F\quad\Rightarrow\quad\exists T\in*G\ (S^\sqsubset\subseteq T\text{ and }({\sqsupset}\restriction T\times S)\text{ is monotone}).
\]
The spoke $T$ does not have to be unique ($S$ can be collapsed to the root).

Let $\mathbf L$
 be the category\footnote{The name $\mathbf L$ comes from that, as we will see below, the Fra\"iss\'e limit of $\mathbf L$ is the Lelek fan.} whose objects are nontrivial fans and whose morphisms are spoke-monotone morphisms in $\mathbf B$.  Further, we let $\mathbf X$ be the subcategory of $\mathbf L$ with the same objects but where the morphisms are \emph{end-preserving}, meaning the image of an end consists only of ends. In formulas, ${\sqsupset}\in\mathbf G_F^G$ is end-preserving if for every $e\in\mathsf EF$ we have that $e^{\sqsubset}\subseteq\mathsf EG$. 
In formulas, we have
\begin{gather*}
    \mathsf{ob}(\mathbf{L}) = \mathsf{ob}(\mathbf{X})=\{\text{nontrivial fans}\}, \\
    \mathsf{mor}(\mathbf{L}) = \{{\sqsupset} \in \mathsf{mor}(\mathbf B) : {\sqsupset}\text{ is spoke-monotone}\}, \\
    \mathsf{mor}(\mathbf X) = \{{\sqsupset}\in\mathsf{mor}(\mathbf{L}): {\sqsupset}\text{ is end-preserving}\}.
\end{gather*}
In particular, if $F$ and $G$ are fans and ${\sqsupset}\in\mathbf X_F^G$, then for every pair of spokes $S\in\ast F$ and $T\in\ast G$ such that $S^{\sqsubset} \subseteq T$, we have that ${\sqsupset}\restriction T\times S$ is a monotone surjective morphism between the paths $S$ and $T$ which sends ends to ends. Since co-injectivity for monotone surjective maps between paths can only fail at ends (\autoref{OnlyEndpointRedundantLemma}) and ends get sent to ends, ${\sqsupset}\restriction T\times S$ is co-injective and therefore belongs to $\mathbf A$. This also shows that $\mathbf L\subseteq\mathbf E$, i.e., every morphism is edge-surjective.

The difference between $\mathbf L$ and $\mathbf X$ is that for a morphism ${\sqsupset}\in\mathbf L_F^G$ we do not require that the image of each spoke of $F$ is a spoke, yet it may be a proper subset of one. Nevertheless, co-bijectivity of $\sqsupset$ ensures the following.

\begin{lemma} \label{lem:locally_cobijective}
    Let $G$ and $H$ be fans and let ${\sqsupset} \in \mathbf{L}_G^H$.
    For every $T \in *H$ there is $S \in *G$ and $e \in \mathsf{E}G \cap S$ such that $e^\sqsubset = \mathsf{E}H \cap T$.
    Hence, ${\sqsupset}\restriction T \times S$ belongs to $\mathbf{A}$, and $\sqsupset$ is end-surjective.
\end{lemma}
\begin{proof}
    Let $\{f\} = \mathsf{E}H \cap T$.
    Since $\sqsupset$ is co-bijective, there is $e \in G$ with $e^\sqsubset = \{f\}$.
    Since $f \neq 0_H$, we have $e \neq 0_H$, and there is a unique spoke $S \in *G$ with $e \in S$.
    Without loss of generality, $e$ is the vertex of $S$ farthest from $0_G$ such that $e^\sqsubset = \{f\}$.
    The restriction ${\sqsupset} \restriction T \times [0_G, e]$ is surjective since $[0_G, e]^\sqsubset$ is connected and contains $f$, is monotone by \autoref{RestrictionsOfMonotone}, and is co-bijective by \autoref{OnlyEndpointRedundantLemma}, and so is in $\mathbf{A}$.
    We want to show that $e$ is an end of $S$.
    Otherwise, there is a vertex $e_+ \in S \setminus [0_G, e]$ adjacent to $e$.
    But $e_+^\sqsubset \ni f_- \sim f$ would contradict the connectedness of $f_-^\sqsupset$, and so $e_+^\sqsubset = \{f\}$, which contradicts the maximality of $[0_G, e]$.
\end{proof}

As every ${\sqsupset} \in \mathbf{X}^G_F$ is end-preserving, $\mathsf{E}(\sqsupset) = {\sqsupset}\restriction \mathsf{E}F^\sqsubset \times \mathsf{E}F$ is in $\mathbf{S}^{\mathsf{E}G}_{\mathsf{E}F}$.
As $\sqsupset$ is end-surjective by the previous lemma, $\mathsf{E}(\sqsupset)$ is surjective.
Since $\mathsf{E}G$ is a discrete graph, $\mathsf{E}(\sqsupset)$ is a (surjective) function.
Altogether, we obtain the \emph{end functor} $\mathsf{E}\maps \mathbf{X} \to \mathbf{D}$.
Hence, to every sequence in $\mathbf{X}$ there is an associated sequence in $\mathbf{D}$, i.e. an inverse sequences of surjections between finite sets, which corresponds to a finitely branching $\omega$-tree.
The situation in $\mathbf{L}$ is more delicate, but still can be captured.

For $(G_n,\sqsupset_n^m)$ a sequence in $\mathbf L$, we write $0_n$ (instead of $0_n)$ for the root of $G_n$.
\begin{dfn}\label{Def:treeofspokes}
    The \emph{tree of spokes} of a sequence $(G_n, \sqsupset^m_n)$ in $\mathbf{L}$ is the following structure consisting of $\omega$-many levels: the $n$\textsuperscript{th} level is the finite set of spokes $*G_n$.
    Moreover, we declare a spoke $S \in *G_{n + 1}$ to be an immediate successor of a spoke $T \in *G_n$ if $\{0_n\} \subsetneq S^{\sqsubset^n_{n + 1}} \subseteq T$.
    As it may happen that $S^{\sqsubset^n_{n + 1}} = \{0_n\}$, there may be nodes at level $n + 1$ without a predecessor.

    By a \emph{branch of spokes} we mean a sequence $\bar{S} = (S_n)$ such that there exists $n_0$ such that for every $n \geq n_0$ we have $S_n \in *G_n$ and $S_{n + 1}$ is an immediate successor of $S_n$, while $S_m = \{0_m\} = S_{n_0}{}^{\sqsubset^m_{n_0}}$ for every $m \leq n_0$.
    We identify $\bar{S}$ with $(S_n, {\sqsupset_n^m}\restriction S_m\times S_n)$, which is a sequence in $\mathbf{S}$ and an upper restriction of $(G_n, \sqsupset^m_n)$.
\end{dfn}

Note that $n_0$ in the definition above is unique for $\bar{S}$, and that if for two branches of spokes $\bar{S}$ and $\bar{S}'$ we have $S_n = S'_n$, then also $S_m = S'_m$ for every $m \leq n$.
Also note that if the sequence $(\sqsupset^m_n)$ is in $\mathbf{X}$, then the tree of spokes is just the $\omega$-tree corresponding to $\mathsf{E}(\sqsupset^m_n)$, and branches of spokes are all sequences of spokes $(S_n)$ such that $S_n{}^{\sqsubset^m_n}= S_m$ for every $m \leq n$.

Given a sequence in $\mathbf{L}$, we would like to analyze the subspace of its spectrum associated with a branch of spokes.
Since a branch of spokes may not be surjective, we need to pass to its \emph{surjective core}.

\begin{obs} \label{SurjectiveCore}
    Let $(G_n, \sqsupset^m_n)$ be any sequence in $\mathbf{S}$.
    By putting $H_n = \bigcap_{k \geq n} G_k{}^{\sqsubset^n_k}$ and ${\sqni}^m_n = {\sqsupset}^m_n \restriction H_m \times H_n$ for every $m \leq n$ we obtain the largest surjective upper restriction $(H_n, \sqni^m_n)$ of $(G_n, \sqsupset^m_n)$, which we call the \emph{surjective core}.

    There is $k_0 \geq n$ such that $H_n = G_k{}^{\sqsubset^n_k}$ for every $k \geq k_0$.
\end{obs}
\begin{proof}
    The last part follows from the fact that $(G_k{}^{\sqsubset^n_k})_{k \geq n}$ is a decreasing sequence of finite sets, and so it stabilises.
    For every surjective restriction $(H'_n, \sqnii^m_n)$ of $(G_n, \sqsupset^m_n)$ we have $H'_n = H'_k{}^{\sqiin^n_k} \subseteq G_k{}^{\sqsubset^n_k}$ for every $n \leq k$, and so $H'_n \subseteq H_n$.
    The rest follows from the equalities \[\textstyle
        H_n{}^{\sqsubset^m_n} = (\bigcap_{k \geq n} G_k{}^{\sqsubset^n_k}) ^{\sqsubset^m_n} = \bigcap_{k \geq n} G_k{}^{\sqsubset^m_k} = H_m.
        \qedhere
    \]
\end{proof}

We say that a branch of spokes is \emph{nondegenerate} if the spectrum of its surjective core is not a singleton, equivalently if the co-bijective modification of the surjective core is nontrivial.
Note that a nondegenerate branch of spokes $\bar{S}$ is uniquely determined by its surjective core $\bar{T} = (T_n, \sqni^m_n)$ and also by the co-bijective modification $\bar{T}'$ of $\bar{T}$: as $\bar{T}'$ is nontrivial, we have $|T'_n| \geq 2$ for every $n \geq n_0$, and so $T'_n$ and $T_n$ are contained in a unique spoke $S_n \in *G_n$, and $\bar{S}$ is uniquely determined by its tail.

\begin{lemma} \label{EverySpoke}
    Let $(G_n, \sqsupset^m_n)$ be a sequence in $\mathbf{L}$.
    For every $k \in \omega$ and every spoke $S \in *G_k$ there is a nondegenerate branch of spokes $(S_n, \sqni^m_n)$ with $S_k = S$ such that $(\sqni^m_n)_{n \geq m \geq k}$ is in $\mathbf{A}$.
\end{lemma}
\begin{proof}
We put $S_k = S$. We recursively choose $S_n$ for $n > k$ by an application of \autoref{lem:locally_cobijective}. For $n<k$,  we recursively define $S_n$ to be the unique spoke $T \in *G_n$ such that $\{0_n\} \subsetneq S_{n + 1}{}^{\sqsubset^n_{n + 1}} \subseteq T$ if the image is nondegenerate, and  $S_n = \{0_n\}$ otherwise.
\end{proof}

If $(G_n,\sqsupset_n^m)$ is a sequence in $\mathbf L$ with induced poset $\PP$, we denote by $\bar 0\in\Spec{\PP}$ the thread of ends $(0_n)$, viewed also as a minimal selector.
Moreover, for a branch of spokes $\bar{S}$, let $\QQ_{\bar{S}}$ denote the poset induced by the surjective core of $\bar{S}$,
and let $(e_{\bar{S}, n})$ denote the thread of ends of the co-bijective modification of the surjective core of $\bar{S}$ different from $\bar{0}$. 

Putting it all together, we describe the spoke decomposition of the resulting spectrum, which turns out to be a (topological) fan, coming from nondegenerate branches of spokes of the starting sequence.

\begin{prp} \label{FanCorrespondence}
    Let $(G_n, \sqsupset^m_n)$ be a sequence in $\mathbf L$  with spectrum $X$, and suppose that $X$ is connected and Hausdorff. Then:
    \begin{enumerate}
        \item \label{itm:fan} $X$ is a fan with root $\bar 0$.
        \item\label{itm:spokes} The map $\bar{S} \mapsto \Spec{\QQ_{\bar{S}}}\subseteq X$ is a one-to-one correspondence between nondegenerate branches of spokes in $(\sqsupset^m_n)$ and spokes in $X$,
        with $(e_{\bar{S}, n})^\leq \in \mathsf{E}X \cap \Spec{\QQ_{\bar{S}}}$ being the corresponding endpoint.
    \end{enumerate}
    If in addition the sequence $(\sqsupset^m_n)$ belongs to $\mathbf X$, then
    \begin{enumerate}[resume]
        \item\label{itm:Xends} 
        $\mathsf{E}X$ is the closed subspace $\Spec{\RR} \subseteq X$ where $\RR$ is the poset induced by the sequence $\mathsf{E}(\sqsupset^m_n)$ in $\mathbf{D}$.
    \end{enumerate}
\end{prp}
\begin{proof}
    By \autoref{ContinuumTreeLimit}, $X$ is a hereditarily unicoherent continuum, and $(\sqsupset_n^m)$ is edge-faithful and has a star-refining subsequence. 
    We will show \ref{itm:fan} and \ref{itm:spokes} at once.
    To show that $X$ is a fan, we need to show that it is arcwise connected and that $\bar{0}$ is its only ramification point.   

    For every nondegenerate branch of spokes $\bar{S}$ let $A_{\bar{S}}$ denote the spectrum $\Spec{\QQ_{\bar{S}}}$. 
    We show that $A_{\bar{S}} \subseteq X$ is an arc with endpoints $\bar{0}$ and $(e_{\bar{S}, n})^\leq$.
    Let $\bar{T}$ be the surjective core of $\bar{S}$ and let $\bar{T'}$ be the co-bijective modification of $\bar{T}$ obtained by \autoref{PathCoBijectiveModification}, with induced poset $\QQ'_{\bar{S}}$.
    We have that $S \mapsto S^\leq$ is an embedding $\Spec{\QQ'_{\bar{S}}} \to X$ onto $A_{\bar{S}}$ by \autoref{Subspaces} and \autoref{DenseRestrictionHomeomorphic}.
    Since $(\sqsupset^m_n)$ has a star-refining subsequence, so does $\bar{T}$ 
    as for every star-refining ${\sqsupset} \in \mathbf{S}^H_G$ and $T \subseteq G$, the restriction ${\sqsupset}\restriction T^\sqsubset \times T$ is star-refining as well.
    Hence, $A_{\bar{S}}$ is a singleton or an arc by \autoref{MoreArcs}.
    Since $\bar{S}$ is nondegenerate, $\Spec{\QQ'_{\bar{S}}}$ is an arc with endpoints $\bar{0}$ and $(e_{\bar{S}, n})$, and so the claim follows.
    
    For nondegenerate branches of spokes $\bar{S} \neq \bar{S}'$ there is $n \in \omega$ such that $S_n \neq S'_n$.
    Then for every $x \in A_{\bar{S}} \cap A_{\bar{S}'}$ we have $[\{x\}]_n \subseteq [A_{\bar{S}}]_n \cap [A_{\bar{S}'}]_n \subseteq S_n \cap S'_n = \{0_n\}$.
    Hence, $x = \bar 0$ and $A_{\bar{S}} \cap A_{\bar{S}'} = \{\bar 0\}$.
    
    For every $x \in X$ there is a nondegenerate branch of spokes $\bar{S}$ such that $x \in A_{\bar{S}}$.
    If $x = \bar{0}$, this is clear.
    If for some $n$ (and so for almost every $n$) the trace $[\{x\}]_n$ intersects two spokes of $G_n$, then the trace contains $0_n$ as it is connected by \autoref{ConnectedSubsets}. Hence, $\bar 0 \subseteq x$, and so $x = \bar 0$.
    Otherwise, $[\{x\}]_n \subseteq T_n \subseteq S_n$ for a branch of spokes $\bar{S}$ and its surjective core $\bar{T}$. If the co-bijective modification of $\bar{T}$ is nontrivial, then $\bar{S}$ is nondegenerate with $x \in A_{\bar{S}}$.
    Otherwise, again $x = \bar 0$.
    Altogether, $\Spec{\PP}$ is arcwise connected, and so is a dendroid.
    
    Since, by \autoref{EverySpoke}, every spoke of $G_0$ can be extended to a nondegenerate branch of spokes and since $G_0$ is a nontrivial fan with at least three spokes, $\bar{0}$ is indeed a ramification point.
    Altogether, $\{A_{\bar{S}}: \bar{S}\text{ is a nondegenerate branch of spokes}\}$
    is a spoke decomposition of $X$, and so $X$ is a fan with ramification point $\bar 0$, and $(e_{\bar{S}, n})^\leq$ are the endpoints of $X$.
    This proves \ref{itm:fan} and \ref{itm:spokes}.

    To prove \ref{itm:Xends} suppose that the sequence $(\sqsupset^m_n)$ is end-preserving.
    Then all sequences $\bar{S} = (S_n)$ of spokes with $S_n{}^{\sqsubset^m_n} \subseteq S_m$ for all $m \leq n$ induce nondegenerate branches of spokes, which are already sequences in $\mathbf{A}$.
    Their threads of ends $(e_{\bar{S}, n}) = (e_{\bar{S}, n})^\leq$, which form the endpoints of $X$, are exactly all threads of $\mathsf{E}(\sqsupset^m_n)$, which form the spectrum $\Spec{\RR}$.
    The spectrum $\Spec{\RR}$ is a closed subspace of $X$ by \autoref{Subspaces} since $\mathsf{E}(\sqsupset^m_n)$ is a sequence in $\mathbf{D}$ and an upper restriction of $(\sqsupset^m_n)$.
\end{proof}

Before moving on to study $\mathbf X$ and $\mathbf L$ individually we  show that these two categories are Fra\"iss\'e.

\begin{lemma}\label{lem:starreffans}
    Star-refining morphisms are wide in $\mathbf X$, and therefore in $\mathbf L$.
\end{lemma}
\begin{proof}
Pick a fan $F$. For every spoke $S\in\ast F$, we can find a path graph $S'$ and a star-refining morphism in $\mathbf A_{S'}^S$.  Since gluing together star-refining morphisms at the root results in a star-refining morphism between the corresponding fans, we have that star-refining morphisms are wide in $\mathbf X$. As $\mathbf L$ has the same objects as $\mathbf X$, the thesis follows.
\end{proof}

\begin{prp}\label{prop:fanscliqueclosed}
Both $\mathbf{L}$ and $\mathbf{X}$ are clique-closed.
\end{prp}
\begin{proof}
Let $G$ be a fan. As there are no triangles in fans, cliques are either singletons or pairs of adjacent vertices. The graph $\mathsf XG$ has a root given by $C=\{0_G\}$. Further, since every point in $G\setminus\{0_G\}$ is related to at most two elements, every element of $\mathsf XG\setminus\{C\}$ is only related to two elements. Hence $\mathsf XG$ is a fan. Note that $\mathsf XG$ has exactly as many spokes as $G$ does (yet they are longer). 

Let ${\in_G} \subseteq G\times\mathsf XG$ be the morphism associating to each clique $C\in\mathsf XG$ its elements. $\in_G$ maps the root to the root. Furthermore, if $S\in\ast G$, then ${\in_G}\restriction S\times\mathsf XS$ is end-preserving, monotone, and co-bijective, and therefore belongs to $\mathbf A$. Hence ${\in_G}\in\mathbf{X}\subseteq\mathbf L$.

We are left to show that if ${\sqsupset}\in \mathbf L_G^H$, for fans $G$ and $H$, then $\mathsf X^{\sqsupset}\in \mathbf{L}$. We further prove that if $\sqsupset$ is end-preserving so is $\mathsf X^\sqsupset$; this will give that both $\mathbf L$ and $\mathbf X$ are clique-closed. Since $\mathbf L\subseteq\mathbf E$ and fans are triangle-free, \autoref{CliqueFunctor} implies that $\mathsf X^{\sqsupset}$ belongs to $\mathbf B$. Since $\sqsupset$ is spoke-monotone, we can find $i'$ such that $(G^i)^\sqsubset\subseteq H^{i'}$, hence $\mathsf X^\sqsupset(\mathsf XG^i)\subseteq\mathsf XH^{i'}$, and the induced morphism is monotone and sends the root to the root. If further each inclusion as above is not strict, each inclusion of the form $\mathsf X^\sqsupset(\mathsf XG^i)\subseteq\mathsf XH^{i'}$ cannot be strict, therefore if $\sqsupset$ is end-preserving, so is $\mathsf X^\sqsupset$. This finishes the proof.
\end{proof}

\begin{prp}\label{prp:amalgfans}
The categories $\mathbf{X}$ and $\mathbf L$ are directed and have amalgamation.
\end{prp}
\begin{proof}
The idea is to amalgamate spoke-by-spoke, and then to glue all amalgamations.
Since we have a weakly terminal object (the claw), it is enough to show amalgamation and we will get directedness. Fix then nontrivial fans $F$, $G$, and $H$, and let ${\sqsupset}\in\mathbf L_{G}^F$ and ${\sqni}\in\mathbf L_{H}^F$ be two morphisms. (For amalgamation in $\mathbf X$, we will just show that if $\sqsupset$ and $\sqni$ belong to $\mathbf X$, so do the amalgamating morphisms). Since $\mathbf L$ (and $\mathbf X$) are clique-closed (\autoref{prop:fanscliqueclosed}), we can assume that both $\sqsupset$ and $\sqni$ are functions.

Define 
\[
\Gamma=\{(S,S')\in\ast G\times\ast H : S^{\sqsubset}\subseteq (S')^{\sqin}\text{ or }(S')^{\sqin}\subseteq S^{\sqsubset}\}.
\]
Fix $(S,S')\in\Gamma$ such that $S^{\sqsubset}\subseteq (S')^{\sqin}$ and let $R=S^\sqsubset$. By spoke-monotonicity, $R$ is a subpath of $F$ which has the root $0_F$ as one of its ends. In particular, if $R\neq\{0_F\}$, it is completely contained in a spoke of $F$.
Let 
\[
{\sqsupset_{S,S'}}:={\sqsupset}\restriction R\times S \text{ and }{\sqni_{S,S'}}:={\sqni}\restriction R\times (R^{\sqni}\cap S').
\]
Both $\sqsupset_{S,S'}$ and $\sqni_{S,S'}$ are monotone surjective morphisms which are furthermore induced by functions, i.e. they belong to $\mathbf A$. Therefore, by \autoref{prp:amalgarc}, there are a path $K=K(S,S')$ and morphisms ${\dashv_{S,S'}}\in \mathbf A_K^{S}\cap\mathbf F$  and ${\Dashv_{S,S'}}\in \mathbf A_K^{R^{\sqni} \cap S'}\cap\mathbf F$
such that
\[
{\sqsupset_{S,S'}\circ\dashv_{S,S'}}={\sqni_{S,S'}\circ\Dashv_{S,S'}}.
\]
If $(S,S')\in\Gamma$ and $(S')^{\sqin}\subseteq S^{\sqsubset}$, we define symmetrically $T=(S')^{\sqin}$ and amalgamate
\[
{\sqsupset_{S,S'}}:={\sqsupset}\restriction T\times (T^{\sqsupset}\cap S) \text{ and }{\sqni_{S,S'}}:={\sqni}\restriction (T\times S')
\]
to obtain a path $K=K(S,S')$ and amalgamating functions ${\dashv_{S,S'}}\in \mathbf A_K^{T^\sqsupset\cap S}\cap\mathbf F\text{ and }{\Dashv_{S,S'}}\in \mathbf A_K^{S'}\cap\mathbf F
$.

Let $e_{S,S'}\in\mathsf EK(S,S')$ be such that 
\[
e_{S,S'}^{\vdash_{S,S'}\circ\sqsubset_{S,S'}}=e_{S,S'}^{\vDash_{S,S'}\circ\sqin_{S,S'}}=\{0_F\}.
\]
(If $S^{\sqsubset}=\{0_F\}=(S')^{\sqin}$, any end of $K(S,S')$ would do, otherwise, $e_{S,S'}$ is uniquely determined, as all morphisms are monotone.)

Let $K$ be the fan whose spokes are the $K(S,S')$, where the ends $e_{S,S'}$ are glued together. Let
\[
{\dashv}=\bigcup_{(S,S')\in\Gamma}\dashv_{S,S'}\text{ and }{\Dashv}=\bigcup_{(S,S')\in\Gamma}\Dashv_{S,S'}.
\]
It is routine to check that $\dashv$ and $\Dashv$ give the required amalgamation, and that they are edge-preserving morphisms between fans. Further, as $\dashv_{S,S'}$ and $\Dashv_{S,S'}$ are monotone, $\dashv$ and $\Dashv$ are spoke-monotone. 

We just have to show these morphisms belong to $\mathbf L$ and, if $\sqsupset$ and $\sqni$ are in $\mathbf X$, then $\dashv$ and $\Dashv$ are in $\mathbf X$ too.
\begin{itemize}
\item Since both $\sqsupset$ and $\sqni$ are co-bijective, for every $S$ there is $S'$ such that $(S,S')\in\Gamma$ and therefore $\dashv_{S,S'}$ surjects onto $S$. This shows that $\dashv$ is surjective, and since $\dashv$ is a function, it belongs to $\mathbf B$. Symmetry ensures that $\Dashv$ is surjective. This gives that $\mathbf L$ has amalgamation.
\item If both $\sqsupset$ and $\sqni$ are in $\mathbf X$, then by end-preservation, we have that $R=S^{\sqsubset}=(S')^{\sqin}$ is a spoke whenever $(S,S')\in\Gamma$. In this case, $R^{\sqni}\cap S'=S'$. This implies that each $\dashv_{S,S'}$ is surjecting onto $S$ and each $\Dashv_{S,S'}$ is surjecting onto $S'$. Hence if $f$ is an end of $K(S,S')$ then $f^{\vdash_{S,S'}}$ is an end in $S$ (likewise, images of ends of $K(S,S')$ are ends in $S'$), meaning that $\dashv$ and $\Dashv$ are end-preserving, i.e. they belong to $\mathbf X$. This gives that $\mathbf X$ has amalgamation.
\end{itemize}
This concludes the proof.
\end{proof}

By Propositions~\ref{prp:amalgfans} and \ref{FraisseExists}, both $\mathbf X$ and $\mathbf L$ have a Fra\"iss\'e sequence.

\subsubsection{Smooth fans as spectra}

In this subsection we prove that every \emph{smooth fan}, i.e. a subfan of the Cantor fan, is the spectrum of a sequence in $\mathbf{L}$.
For that we carefully define a particular sequence of covers, and use \autoref{SequenceForSpace}.

For every $n, k \in \omega$ let $A_{n, k}$ be the open interval $(k 2^{-n} - 2^{-n}, k 2^{-n} + 2^{-n})$ in the real line, i.e. the open ball with center $x_{n, k} = k 2^{-n}$ and radius $2^{-n}$.
Note that the sequence of covers $(\{A_{n, k}: 0 \leq k \leq 2^n\})_n$, when intersected with the unit interval $[0, 1]$, yields an $\omega$-poset similar to \cite[Example~2.16]{BartosBiceV.Compacta} and \autoref{fig:comparison}.

\begin{obs} \label{IntervalCovers}
    Our sequence of covers has the following structure.
    \begin{enumerate}
        \item We have $A_{n, k} \setminus \bigcup_{k' \neq k} A_{n, k'} = \{x_{n, k}\}$, i.e. the center is the unique point not covered by the other sets at the same level.
        Moreover, if $x_{n, k} \in A_{m, j}$ for some $m \leq n$, then $A_{n, k} \subseteq A_{m, j}$.
        \item We have $A_{n, k} \supseteq A_{n + 1, j}$ if and only if $j \in \{2k - 1, 2k, 2k + 1\}$ and in fact $A_{n, k} = A_{n + 1, 2k - 1} \cup A_{n + 1, 2k} \cup A_{n + 1, 2k + 1}$.
        \item Given $A_{n + 1, k} \subseteq A_{n, j}$, either $k$ is even ($k = 2k'$) and $j = k'$ and the two sets share the center $x_{n + 1, 2k} = x_{n, k}$, or $k$ is odd ($k = 2k' + 1$) and $j \in \{k, k + 1\}$ and $A_{n + 1, k} = A_{n, k'} \cap A_{n, k' + 1}$.
    \end{enumerate}
\end{obs}

We represent the Cantor space as $2^\omega$ with the canonical basis $\{B_s: s \in 2^{<\omega}\}$, and we represent the Cantor fan as the quotient $X = ([0, 1] \times 2^\omega) / (\{0\} \times 2^\omega)$.
Let $\pi_{[0, 1]}\maps X \to [0, 1]$ denote the canonical projection.
Let $\alpha, \beta\maps \omega \to \omega$ be two strictly increasing functions, parametrizing our construction.
Finally, for every $n$ we put
\[
    G_n = \{(A_{\alpha(n), 0} \cap [0, 1]) \times B_\emptyset\} \cup \{(A_{\alpha(n), k} \cap [0, 1]) \times B_s: 0 < k \leq 2^{\alpha(n)}, s \in 2^{\beta(n)}\},
\]
i.e. we take the elementwise product of the $\alpha(n)$\textsuperscript{th} cover of the unit interval with the $\beta(n)$\textsuperscript{th} cover of the Cantor space, and we obtain a basic cover of the product $[0, 1] \times 2^\omega$. Then we replace the sets $(A_{\alpha(n), 0} \cap [0, 1]) \times B_s$, $s \in 2^{\beta(n)}$, with their union $(A_{\alpha(n), 0} \cap [0, 1]) \times B_\emptyset$, which corresponds to gluing the endpoints and obtaining the Cantor fan.

We turn the sequence $(G_n)$ of sets into a lax-sequence of graphs and relational morphisms by putting $g \sqcap g'$ if and only if $g \cap g' \neq \emptyset$ in every $G_n$, and by putting $g \sqsupset^m_n h$ if and only if $g^\in \supseteq h^\in$ for every $m \leq n$, $g \in G_m$, $h \in G_n$.

\begin{lemma}\label{lem:Cantorfansequence}
    $(G_n, \sqsupset^m_n)$ is eventually a sequence in $\mathbf{X}$ whose spectrum is the Cantor fan.
\end{lemma}
\begin{proof}
    It is easy to see that for $n \geq 2$, $G_n$ is a fan with root $(A_{\alpha(n), 0} \times B_{\emptyset}) \cap X$ and with $2^{\beta(n)}$ spokes.
    We use \autoref{SequenceForSpace} to obtain that $(G_n, \sqsupset^m_n)$ is a sequence in $\mathbf{B}$ with spectrum $X$, but we need to verify its hypotheses.

    Every $G_n$ is a minimal cover of $X$ since $(A_{\alpha(n), k} \times B_s) \setminus \bigcup_{(k', s') \neq (k, s)} (A_{\alpha(n), k'} \times B_{s'}) = \{x_{\alpha(n), k}\} \times B_s$ (see \autoref{IntervalCovers}), and that $x_{\alpha(n), k} \in [0, 1]$ if and only if $0 \leq k \leq 2^{\alpha(n)}$.

    $G_n$ consolidates $G_{n + 1}$ essentially because the cover $\{A_{m, k} \cap [0, 1]\}_{0 \leq k \leq 2^m}$ consolidates $\{A_{m + 1, k} \cap [0, 1]\}_{0 \leq k \leq 2^{m + 1}}$, and $\{B_s\}_{s \in 2^m}$ consolidates $\{B_s\}_{s \in 2^{m + 1}}$.

    It is easy to see that $G_n \cap G_{n + 1} = \emptyset$.
    Next, clearly, $\diam(A_{\alpha(n), k})$ and $\diam(B_s)$ for $s \in 2^{\beta(n)}$ depend only on $n$ and go to $0$ as $n \to \infty$, so every open cover of $[0, 1] \times 2^\omega$ is refined by a cover consisting of the products of the members of the $\alpha(n)$\textsuperscript{th} and $\beta(n)$\textsuperscript{th} levels of the respective sequences of covers.
    It follows that every open cover of the quotient $X$ is refined by some $G_n$.
    Altogether we have verified all hypotheses of \autoref{SequenceForSpace}.
    
    Finally, the every $\sqsupset^n_{n + 1}$ is spoke-monotone and end-preserving, and so $(\sqsupset^m_n)$ is a sequence in $\mathbf{X}$.
    Every spoke of $G_n$ is of the form 
    \[
        S_{n, s} = \{(A_{\alpha(n), 0} \times B_{\emptyset}) \cap X\} \cup \{(A_{\alpha(n), k} \times B_s) \cap X: 0 < k \leq 2^{\alpha(n)}\}
    \]
    for some $s \in 2^{\beta(n)}$.
    Hence, $S_{n + 1, s}{}^{\sqsubset^n_{n + 1}} = S_{n, s\restriction \beta(n)}$, and $\sqsupset^n_{n+1}\restriction (S_{n + 1, s} \times S_{n, s\restriction \beta(n)})$ is monotone -- it is essentially the composition of the refinement relations $\dashv^m_{m + 1}$ on $\{A_{m, k} \cap [0, 1]\}_{k \leq 2^m} \times \{A_{m + 1, k} \cap [0, 1]\}_{k \leq 2^{m + 1}}$, which are monotone as every $\dashv^m_{m + 1}$ is the result of a variant of the edge-splitting construction (\autoref{EdgeSplitting}).
\end{proof}

Next, for every subfan $Y \subseteq X$ of the Cantor fan and $n \in \omega$, we put $\beta(n) = n$ for every $n$ and we pick suitable $\alpha$ (below).
Then we put
\[
    H_n = \{g \cap Y: g \in G_n,\ x_g \in \pi_{[0, 1]}(g \cap Y)\} 
\]
where $x_g$ denotes the center $x_{n, k}$ of the unique set $A_{n, k}$ such that $g = (A_{n, k} \times B_s) \cap X$, and define the corresponding lax-sequence of graphs $(H_n, \sqni^m_n)$.
Without loss of generality, we may identify every element $h \in H_n$ with the corresponding element $g \in G_n$ such that $h = g \cap X$, so that $H_n\subseteq G_n$.

We inductively define $\alpha$ depending on $Y$ as follows.
We put $\alpha(0) = 0$ and given $\alpha(n)$ we pick $\alpha(n + 1) > \alpha(n)$ large enough so that whenever $g \cap Y \in H_n$ for some $g = (A_{\alpha(n), k} \times B_s) \cap X \in G_n$, then for both $i = 0, 1$, if there is a spoke $[0, y] \times \{f\}$ of $Y$ for some $y \leq 1$ and $f \in 2^\omega$ with $f \supseteq s^{\smallfrown}i$, then there is such a spoke with $y \geq x_g - 2^{\alpha(n)} + 2^{\alpha(n + 1)}$, so for the leftmost set $g' = (A_{\alpha(n + 1), k'} \times B_{s^\smallfrown i}) \cap X \in G_{n + 1}$ contained in $g$ we have $g' \cap Y \in H_{n + 1}$.

\begin{lemma}
    There is $n_0$ such that above $n_0$ we have the following.
    \begin{enumerate}
        \item\label{itm:cov_fan} $H_n$ is a fan with $0_n \cap Y = 0_n$.
        \item\label{itm:cov_cov} For every $g \in G_n$ there is $h \in H_n$ such that $g \cap Y \subseteq h$, and so $H_n$ is a minimal cover of $Y$.
        \item\label{itm:cov_faithful} For every $n \geq m$, $g \in G_n$, $g' \in G_m$ such that $g \cap Y \in H_n$ and $g' \cap Y \in H_m$ we have $g \subseteq g'$ if and only if $g \cap Y \subseteq g' \cap Y$.
        If $n = m$, we also have $g \cap g' \neq \emptyset$ if and only if $g \cap g' \cap Y \neq \emptyset$.
        Altogether, $(H_n, \sqni^m_n)$ is a lax-restriction of $(G_n, \sqsupset^m_n)$.
        \item\label{itm:cov_cons} The sequence of covers $(H_n)$ is consolidating.
        \item\label{itm:cov_seq} $(H_n, \sqni^m_n)$ is a sequence in $\mathbf{L}$ with spectrum $Y$, and it is the unique co-bijective modification of the prime restriction of $(G_n, \sqsupset^m_n)$ induced by the traces $[Y]_n$.
    \end{enumerate}
\end{lemma}
\begin{proof}
    Claim~\ref{itm:cov_fan}.
    Clearly, $0_n \cap Y \in H_n$.
    If $(A_{\alpha(n), k} \times B_s) \cap Y \in H_n$, then also $(A_{\alpha(n), k'} \times B_s) \cap Y \in H_n$ for every $0 < k' < k$ since $\pi_{[0, 1]}(([0, 1] \times B_s) \cap Y)$ is connected.
    Hence, $H_n$ is a fan as long as $\{s \in 2^n: (\{x_{\alpha(n), 1}\} \times B_s) \cap Y \neq \emptyset\}$ has at least three elements, which happens for every $n \geq n_0$ for suitable $n_0$.

    For \ref{itm:cov_cov} let $g \in G_n$.
    Without loss of generality, $g = (A_{\alpha(n), k} \times B_s) \cap X$ for some $k > 0$ (as $(A_{\alpha(n), 0} \times B_{\emptyset}) \cap Y \in H_n$) and $s \in 2^n$.
    If $g \cap Y \neq \emptyset$ but $x_g \notin \pi_{[0, 1]}(g \cap Y)$, then $(A_{\alpha(n), k} \times B_s) \cap Y \subseteq (A_{\alpha(n), k - 1} \times B_s)$ and $x_{\alpha(n), k - 1} \in \pi_{[0, 1]}((A_{\alpha(n), k - 1} \times B_s) \cap Y)$ as $Y$ is connected.
    Hence, $g \cap Y$ is covered by $g' \cap Y \in H_n$ where $g'$ is the unique vertex of $G_n$ adjacent to $g$.
    It follows that $H_n$ is a cover of $Y$.
    The minimality follows from the definition of the covers $H_n$ as $\{x_{\alpha(n), k}\} \times B_s$ is the part of $A_{\alpha(n), k} \times B_s$ not covered by the other sets $A_{\alpha(n), k'} \times B_{s'}$.

    In \ref{itm:cov_faithful}, we have $g = (A_{\alpha(n), k} \times B_s) \cap X$ and $g' = (A_{\alpha(m), k'} \times B_{s'}) \cap X$ for suitable $k, k', s, s'$.
    Suppose $g \cap Y \subseteq g' \cap Y$.
    Then clearly, $s \supseteq s'$.
    We also have $x_{\alpha(n), k} \in \pi_{[0, 1]}(g \cap Y) \subseteq A_{m, k'}$, and so $A_{\alpha(n), k} \subseteq A_{\alpha(m), k'}$ by \autoref{IntervalCovers}.
    Hence, $g \subseteq g'$.

    Next, suppose $n = m$, $k < k'$, and $g \cap g' \neq \emptyset$.
    Then $s \in \{s', \emptyset\}$, $k' = k + 1$, and $Y$ contains a spoke $[0, y] \times \{f\}$ for some $f \in 2^\omega$ such that $f \supseteq s'$ and $y \geq x_{\alpha(n), k'} > x_{\alpha(n), k}$.
    Hence, $g \cap g' \cap Y \supseteq (A_{\alpha(n), k} \cap A_{\alpha(n), k + 1} \cap [0, y]) \times \{f\} = (x_{\alpha(n), k}, x_{\alpha(n), k + 1}) \times \{f\} \neq \emptyset$.
    Therefore, we may identify $H_n$ with an induced subgraph of $G_n$.

    For \ref{itm:cov_cons}, let $(y, f) \in [0, 1] \times 2^\omega$ be a point of $(A_{\alpha(n), k} \times B_s) \cap Y \in H_n$.
    Let $i \in \{0, 1\}$ be such that $f \supseteq s^\smallfrown i$, and let $k' \in \omega$ be the greatest such that $x_{\alpha(n + 1), k'} \leq y$.
    We put $h' = (A_{\alpha(n + 1), k'} \times B_{s^\smallfrown i}) \cap Y \in H_{n + 1}$.
    Then we have $(y, f) \in h'$.
    Let $k'' \in \omega$ be the smallest such that $A_{\alpha(n + 1), k''} \subseteq A_{\alpha(n), k}$, and put $h'' = (A_{\alpha(n + 1), k''} \times B_{s^\smallfrown i}) \cap Y$.
    From the defining property of $\alpha$ we have $\alpha$ we have $h'' \in H_{n + 1}$.
    If $k' \geq k''$, then we have $(y, f) \in h' \subseteq g$.
    Otherwise, we have $k'' = k' + 1$ and $(y, f) \in h'' \subseteq g$.

    Claim~\ref{itm:cov_seq}.
    Let us first show that $H_n \cap H_{n + 1} = \emptyset$.
    Every $h = (A_{\alpha(n), k} \times B_s) \cap Y \in H_n$ contains a point $(x_h, f)$ for some $f \in 2^\omega$ with $f \supseteq s ^\smallfrown i$.
    Then $h' = (A_{\alpha(n + 1), 2^{\alpha(n + 1) - \alpha(n)} k} \times B_{s^\smallfrown i}) \cap Y$ is the only member of $H_{n + 1}$ containing the point $(x_h, f) = (x_{h'}, f)$.
    But $(x_{\alpha(n + 1), 2^{\alpha(n + 1) - \alpha(n)} k - 1}, f) \in h \setminus h'$.

    Next, $\max\{\diam(h): h \in H_n\} \to 0$ as $n \to 0$ since $H_n$ refines $G_n$.
    This together with \ref{itm:cov_cov}, \ref{itm:cov_cons}, and the previous paragraph are the hypotheses of \autoref{SequenceForSpace}, and so $(H_n, \sqni^m_n)$ is a sequence, it is in $\mathbf{B}$, and $Y$ is its spectrum.
    By \ref{itm:cov_faithful}, $(H_n, \sqni^m_n)$ is a restriction of $(G_n, \sqsupset^m_n)$.
    By \ref{itm:cov_fan}, every $H_n$ is a fan.
    It is easy to see that every $\sqni^m_n$ maps spokes into spokes, and so it is spoke-monotone by \autoref{RestrictionsOfMonotone}.
    Hence, $(\sqni^m_n)$ is in $\mathbf{L}$.

    Let $(\sqnii^m_n)$ be the prime restriction of $(\sqsupset^m_n)$ induced by the traces $[Y]_n$.
    For every $n$, the trace $[Y]_n$ is the cover $\{g \cap Y: g \in G_n, g \cap Y \neq \emptyset\}$.
    By \ref{itm:cov_cov}, $[Y]_{n + 1}$ refines $H_{n + 1}$ and so $H_n$. It follows that $(\sqni^m_n)$ is a dense restriction of $(\sqnii^m_n)$.
    Since for every $g \cap Y \in H_n$, $g$ is the unique element of $G_n$ containing the point $(x_g, f)$ of $Y$ for suitable $f \in 2^\omega$, $(\sqni^m_n)$ is the smallest dense restriction of $(\sqnii^m_n)$ and so its unique co-bijective modification.
\end{proof}

\begin{cor} \label{SmoothFansRealizable}
    The Cantor fan is the spectrum of a sequence in $\mathbf{X}$.
    Every smooth fan is the spectrum of a sequence in $\mathbf{L}$.
    \qed
\end{cor}

\subsubsection{The Cantor fan}

As in \autoref{ss.Cantor} and \autoref{ss.arc}, we intend to characterise Fra\"iss\'e and lax-Fra\"iss\'e sequences in $\mathbf X$. 

First of all, by \autoref{cor:FraisseLimit}, a lax-Fra\"iss\'e sequence in $\mathbf X$ has an edge-witnessing subsequence. Differently from the case of the interval, this does not suffice:

\begin{xpl}\label{xpl:cantorfannasty}
For $n\geq 2$, let $G_n$ be a fans with $2^n$ spokes, such that the first $2^n-1$ spokes have length $2^n$, and the last spoke has length $2$. Viewing each set $\{0,\ldots,2^n-1\}$ as a path (where $i\sqcap j$ if and only if $|i-j|\leq 1$), let ${\sqni_n}\subseteq \{0,\ldots,2^n-1\}\times \{0,\ldots,2^{n+1}-1\}$ be a monotone edge-witnessing morphism in $\mathbf A$ such that $0^{\sqin_n}=\{0\}$. For each spoke of $G_n$ of length $2^n$, we associate two spokes of $G_{n+1}$ of length $2^{n+1}$ via $\sqni_n$, while to the spoke of $G_n$ of length $2$, we associate one spoke of the $G_{n+1}$ of length $2^{n+1}$ via a monotone edge-witnessing morphism in $\mathbf A$, and the spoke of $G_{n+1}$ of length $2$ via the identity. Gluing appropriately such morphisms gives an edge-witnessing ${\sqsupset_{n+1}^n}\in\mathbf X$. The sequence $(\sqsupset_n^m)$ cannot be lax-Fra\"iss\'e, as it is not even cofinal:
 if $F$ is the fan with $3$ spokes each of length $3$, there is no $n$ such that $\mathbf X^F_{G_n}\neq\emptyset$.
\end{xpl}

Recall that, in $\mathbf X$, branches of spokes coincide with their surjective cores, as morphisms are end-preserving.  \autoref{xpl:cantorfannasty} shows that, for a sequence in $\mathbf X$, having an edge-witnessing subsequence and the fact that every branch of spokes has an edge-witnessing subsequence are not the same concept. This motivates the following:

\begin{dfn}
A morphism ${\sqsupset}\in\mathbf X_F^G$ is \emph{spokewise edge-witnessing} if for every $S\in\ast F$ the restriction ${\sqsupset}\restriction S^\sqsubset \times S$ is edge-witnessing. We define  \emph{spokewise star-refining} analogously.
\end{dfn}

\begin{prp}\label{propspokewisestuff}
In $\mathbf{X}$ we have the following.
\begin{enumerate}
\item Every spokewise star-refining/edge-witnessing morphism is star-refining/edge-witnessing;
\item Every star-refining morphism is already spokewise star-refining;
\item Every spokewise star-refining morphism is spokewise edge-witnessing;
\item The composition of two spokewise edge-witnessing morphisms is (spokewise) star-refining.
\item Spokewise edge-witnessing/star-refining morphism form lax-closed ideals.
\end{enumerate}
\end{prp}

\begin{proof}
The first implications are obvious. For the second one, note that the star of any non root point is contained in a single spoke, and so is its image, so the only point which may cause problems is the root. On the other hand, the image of the root is only the root.

To show that spokewise star-refining morphisms are spokewise edge-witnessing, note that this is the case when restricting to pairs of spokes, as one is dealing with maps in $\mathbf A$ (\autoref{StarEdgeSurjectiveEdgeWitnessing}). The same argument, together with \autoref{StarSurjectiveCoEdgeWitnessing}, gives that composing spokewise edge-witnessing morphisms gives (spokewise) star-refining ones. Once again, the last statements follows from that we are truly handling glued morphisms in $\mathbf A$, and the fact that the original properties form lax-closed ideals (\autoref{LaxClosed} and \autoref{Ideals}).
\end{proof}

\begin{prp}\label{CantorEquivalences}
Let $(G_n,\sqsupset_n^m)$ be a sequence in $\mathbf X$. The following are equivalent.
\begin{enumerate}
    \item \label{PCF1} $(\sqsupset_n^m)$ has a (spokewise) star-refining subsequence;
 \item\label{PCF2} $(\sqsupset_n^m)$ has a spokewise edge-witnessing subsequence;
 \item \label{PCF3} every branch of spokes is lax-Fra\"iss\'e in $\mathbf A$.
 \end{enumerate}
\end{prp}
\begin{proof}
The equivalence between \ref{PCF1} and \ref{PCF2} follows directly from \autoref{propspokewisestuff}. 
Since spokes are nondegenerate paths, all branches of spokes are nontrivial, and hence a branch of spokes is lax-Fra\"iss\'e if and only if it has a star-refining subsequence, by \autoref{AFraisseEquivalents}.
Clearly, if $(G_{n_k})$ is a spokewise star-refining subsequence, then $(S_{n_k})$ is a star-refining subsequence of $\bar{S}$ for every branch of spokes $\bar{S}$.
To prove the equivalence of \ref{PCF1} and \ref{PCF3}, we are left to show that if all branches of spokes have star-refining subsequences, then there exists a spokewise star-refining subsequence.
Suppose former.
We will show that for every $m \in \omega$ there is $n \geq m$ such that $\sqsupset^m_n$ is spokewise star-refining.
Otherwise there is $m$ such that the family
\[\textstyle
T = \{S \in \bigcup_{n \geq m} \ast G_n: {\sqsupset}\restriction S^{\sqsubset^m_n} \times S \text{ is not star-refining}\}
\]
intersects every $\ast G_n$, $n \geq m$.
Note that as star-refining morphisms form an ideal, if a spoke $S \in \ast G_n$ is a member of $T$, then so is every spoke $S^{\sqsubset^k_n}$ for $m \leq k \leq n$.
Since the tree of spokes is finitely branching, it follows that there is a branch of spokes $\bar{S}$ such that $S_n \in T$ for every $n \geq m$, and so $\bar{S}$ has no star-refining subsequence.
\end{proof}

Recall, for a morphism ${\sqsupset}\in \mathbf X$, its image under the end functor $\mathsf E(\sqsupset)$ was defined before \autoref{Def:treeofspokes}. If $(\sqsupset_n^m)$ is a sequence in $\mathbf X$, $(\mathsf E(\sqsupset_n^m))$ is a sequence in $\mathbf D$.

\begin{lemma}\label{lemma:endsplitting}
Let $(G_n,\sqsupset_n^m)$ be a sequence in $\mathbf X$. The following are equivalent.
\begin{itemize}
    \item $(\sqsupset_n^m)$ has an end-splitting subsequence;
    \item $(\mathsf E(\sqsupset_n^m))$ is Fra\"iss\'e in $\mathbf{D}$.
\end{itemize}
If further the spectrum is connected and  Hausdorff, then the above are equivalent to
\begin{itemize}
    \item $\mathsf EX$ is homeomorphic to the Cantor space.
\end{itemize}
\end{lemma}
\begin{proof}
Clearly, $(\sqsupset^m_n)$ has an end-splitting subsequence if and only if $(\mathsf{E}(\sqsupset^m_n))$ has an anti-injective subsequence, which, by \autoref{DFraisseEquivalents}, is equivalent to $(\mathsf{E}(\sqsupset^m_n))$ being Fra\"{i}ss\'e in $\mathbf{D}$ and to its spectrum being the Cantor space.
Under the extra hypotheses, by \autoref{FanCorrespondence}\ref{itm:Xends}, $\mathsf{E}X$ is the spectrum of $(\mathsf{E}({\sqsupset}^m_n))$.
\end{proof}

\begin{thm}\label{FraisseCantorFan}
Let $(G_n,\sqsupset_n^m)$ be a sequence in $\mathbf X$.
The following are equivalent.
 \begin{enumerate}
 \item\label{CF1} $(\sqsupset_n^m)$ is a lax-Fra\"iss\'e sequence in $\mathbf X$;
 \item\label{CF3} every branch of spokes is lax-Fra\"iss\'e in $\mathbf A$ and $(\mathsf E(\sqsupset_n^m))$ is Fra\"iss\'e in $\mathbf{D}$;
 \item\label{CF2} $(\sqsupset_n^m)$ has a star-refining end-splitting subsequence;
 \item\label{CF4} $\Spec{\PP}$ is homeomorphic to the Cantor fan.
 \end{enumerate}
Moreover, $(\sqsupset_n^m)$ is Fra\"iss\'e precisely when every branch of spokes $\bar S$ is Fra\"iss\'e in $\mathbf A$.
\end{thm}

\begin{proof}
\ref{CF3}$\Leftrightarrow$\ref{CF2} follows from \autoref{CantorEquivalences} and \autoref{lemma:endsplitting}.

\ref{CF1}$\Rightarrow$\ref{CF2}: Since star-refining end-splitting morphisms form a wide ideal in $\mathbf X$ (by \autoref{Ideals} and \autoref{lem:starreffans} for star-refining, and \autoref{lemma:enddense} for end-splitting), \autoref{IdealFraisseSubsequence} gives that any lax-Fra\"iss\'e sequence has a star-refining end-splitting subsequence.

\ref{CF2}$\Rightarrow$\ref{CF1}: Fix $m$, a fan $F$, and ${\sqsupset}\in\mathbf X_F^{G_m}$. Let $M=\max\{|S|: S\in*F\}$. By applying condition \ref{PCF3} of \autoref{CantorEquivalences}, to every spoke of $G_m$, and using that $(\mathsf E(\sqsupset_n^m))$ is a Fra\"iss\'e sequence in $\mathbf D$, we can find $n$ large enough such that
\begin{itemize}
\item there is ${\dashv}\in\mathbf D_{\mathsf EG_n}^{\mathsf EF}$ such that $\mathsf E(\sqsupset_n^m)=\mathsf{E}({\sqsupset})\circ{\dashv}$, and
\item if $S\in*G_m$ and $T\in *G_n$ are such that $T^{\sqsubset_n^m}=S$, the morphism ${\sqsupset_n^m}\restriction S\times T$ can be written as a composition of $M$ edge-witnessing morphisms in $\mathbf A$. 
\end{itemize}
As in the proof of \autoref{AFraisseEquivalents}, we get that for every adjacent $g,h\in G_m$, $M\leq |\{g,h\}_{\sqsupset_n^m}|$, as each two adjacent vertices belong to the same spoke. 

Fix $S\in*G_m$ and $T\in *G_n$ such that $T^{\sqsubset_n^m}=S$. Let $e\in \mathsf ET\setminus\{0_n\}$, and let $S'\in*F$ be such that $e^{\vdash}\subseteq S'$. We now apply \autoref{TypeSubFactorisation} to $P=S'$, $Q=T$, $R=S$, and we get that there is ${\Dashv_{ST}}\in\mathbf A_T^{S'}$ such that 
\[
{({\sqsupset}\restriction S\times S')}\circ{\Dashv_{ST}}\subseteq{\sqsupset_n^m}\restriction S\times T.
\]
Note that $0_n^{\Dashv_{ST}}=\{0_F\}$, and hence we can glue together all morphisms of the form $\Dashv_{ST}$, for relevant $S$ and $T$, and get a morphism ${\Dashv}\subseteq F\times G_n$ such that ${\sqsupset\circ\Dashv}\subseteq{\sqsupset_n^m}$.
This shows that the sequence $(\sqsupset_n^m)$ is lax-Fra\"iss\'e, and proves \ref{CF1}.

\ref{CF4}$\Rightarrow$\ref{CF2}: If $\Spec{\PP}$ is homeomorphic to Cantor fan, it is Hausdorff and connected, and its set of endpoints is homeomorphic to the Cantor space.
Therefore, $(\sqsupset^m_n)$ has a star-refining subsequence by \autoref{ContinuumTreeLimit} and an end-splitting subsequence by \autoref{lemma:endsplitting}.

\ref{CF1}$\Rightarrow$\ref{CF4}: By \autoref{LaxFraisseSpectra}, all spectra of posets induced by lax-Fra\"iss\'e sequences are homeomorphic. Since the sequence given in \autoref{lem:Cantorfansequence} is lax-Fra\"iss\'e (it satisfies, for example, the requirements of \ref{CF2}), we have the thesis.

Checking the last statement is an easy generalisation of the fact that for sequences in $\mathbf A$, being Fra\"iss\'e is precisely equivalent to having a strictly anti-injective sequence (see \autoref{prp:amalgarc}).
\end{proof}

\subsubsection{The Lelek fan}\label{sss.Lelek}
We now focus on the category $\mathbf L$ and to the study the spectrum of the poset induced by a Fra\"iss\'e sequence in $\mathbf L$. 

We call a morphism ${\sqsupset}\in\mathbf L_G^H$ 
\begin{itemize}
    \item \emph{end-dense} if every $h\in H$ is above some $e\in \mathsf EG$, i.e. $e\sqsubset h$. 
    \item \emph{end-splitting} if every $e\in\mathsf EH$ is above two distinct $g$ and $g'$ in $\mathsf EG$.
    \end{itemize}
End-dense morphisms are as far as possible from morphisms in $\mathbf X$.
\begin{lemma}\label{lemma:enddense}
Both end-dense and end-splitting morphisms form a wide ideal in $\mathbf L$. Also, end-splitting morphisms are wide in $\mathbf X$ (and so form a wide ideal in it).

Furthermore, if $(G_n,\sqsupset_n^m)$ is a sequence in $\mathbf L$, then
\begin{enumerate}
    \item\label{end1} $(\sqsupset_n^m)$ has an end-dense subsequence if and only if for every $m$ and $g\in G_m$ there are $n\geq m$ and $h\in\mathsf EG_n$ such that $h\sqsubset_n^m g$;
    \item\label{end2} $(\sqsupset_n^m)$ has an end-splitting subsequence if and only if for every $m$ and $g\in \mathsf EG_m$ there are $n\geq m$ and two distinct $h,h'\in\mathsf EG_n$ such that $h,h'\sqsubset_n^m g$;
    \item\label{end3} $(\sqsupset_n^m)$ has an end-dense end-splitting subsequence if and only if for every $m$ and $g\in G_m$ there are $n\geq m$ and distinct $h,h'\in\mathsf EG_n$ such that $h,h'\sqsubset_n^m g$.
\end{enumerate}
\end{lemma}
\begin{proof}
We leave to the reader the easy verification that end-dense/end-splitting morphisms are wide in $\mathbf L$. Since `doubling' the spokes in an obvious way is an end-splitting morphism in $\mathbf X$, end-splitting morphisms are wide in $\mathbf X$. 

To show these classes of morphisms form ideals, fix fans $F$, $G$, and $H$, and morphisms ${\sqsupset}\in\mathbf L_G^F$ and ${\sqni}\in\mathbf L_H^G$. 

If $\sqni$ is end-dense, fix $f\in F$ and let $g\in G$ with $g\sqsubset f$. Fix now $h\in\mathsf EF$ such that $h\sqin g$, so that $h\sqin g\sqsubset f$. As $f$ was arbitrary, $\sqsupset\circ\sqni$ is end-dense. If $\sqsupset$ is end-dense, fix again $f\in F$, and pick $g\in\mathsf EG$ with $g\sqsubset f$. Since $g$ is an end in $G$, we can find $h\in\mathsf EH$ such that $h\sqin g$, so that $h\sqin g\sqsubset f$, which shows that $\sqsupset\circ\sqni$ is end-dense. 

Similarly, if $\sqni$ is end-splitting, fix $f\in\mathsf EF$. Fix distinct $g,g'\in\mathsf EG$ which are below $f$. Since every end in $G$ is above (at least) one end in $h$ and $h'$ in $\mathsf EG$ which are respectively below $g$ and $g'$, and therefore $f$. On the other hand, if $\sqsupset$ is end-splitting, any end of $F$ is above at least one end of $G$, which is in turn above two distinct ends in $H$. This shows $\sqni\circ\sqsupset$ is end-splitting. (As $\mathbf X$ is a subcategory of $\mathbf L$ in which end-splitting morphisms are wide, end-splitting morphisms are a wide ideal in $\mathbf X$ too.)

Fix now a sequence $(G_n,\sqsupset_n^m)$ in $\mathbf L$. In light of \autoref{IdealSubsequences}, to prove \ref{end1} one just has to prove that 
\[
\forall m\in\mathbb N\, \forall g\in G_m\, \exists n\geq m\, \exists h\in\mathsf EG_n \, (h\sqsubset_n^m g)
\]
is equivalent to its uniform version, that is
\[
\forall m\in\mathbb N\,\exists n\geq m\, \forall g\in G_m\, \exists h\in\mathsf EG_n \, (h\sqsubset_n^m g).
\]
One implication is obvious. For the other one, fix $m$ and for every $g\in G_m$ find $n_g\geq m$ and $h_g\in\mathsf EG_{n_g}$ which is below $g$. Let $n=\max_{g\in G_m}n_g$. Since each $\sqsupset_{n}^{n_g}$ is end-surjective (\autoref{lem:locally_cobijective}) we can find $k_g\in\mathsf EG_n$ below each $h_g$. This concludes the proof. The proofs of \ref{end2} and \ref{end3} are analogous and left to the reader.
\end{proof}

\begin{lemma} \label{EndDenseSubsequence}
    Let $(G_n, \sqsupset^m_n)$ be a sequence in $\mathbf L$ with induced poset $\PP$ such that its spectrum $\Spec{\PP}$ is connected and Hausdorff.
    Then $(\sqsupset_n^m)$ has an end-dense subsequence if and only if the set of endpoints is dense in $\Spec{\PP}$.
\end{lemma}
\begin{proof}
    By \autoref{FanCorrespondence}, the endpoints of $\Spec{\PP}$ are exactly the selectors $(e_{\bar{S}, n})^\leq$ where $\bar{S}$ is a nondegenerate branch of spokes.
    
    Let us first show the forward direction.
    As $\{p^\in\}_{p\in\PP}$ is a base for the topology of $X$, we just need to check that every $p^\in$ contains an endpoint.
    Fix $p$ and $n$ such that $p\in G_n$.
    Since $(\sqsupset_n^m)$ has an end-dense subsequence, by \autoref{lemma:enddense} we can find $k \geq n$, $S_k \in *G_k$, and $e_k\in\mathsf EG_k \cap S_k$ such that $e_k \sqsubset^n_k p$.
    By \autoref{EverySpoke}, $S_k$ can be extended to a nondegenerate branch of spokes $\bar{S} = (S_n, \sqni^m_n)$ with $(\sqni^m_n)_{n \geq m \geq k}$ in $\mathbf{A}$.
    Hence, $e_{\bar{S}, k} = e_k \leq p$, and so $p^\in$ contains the endpoint $(e_{\bar{S}, n})^\leq$.
    
    Vice versa, assume that the endpoints are dense.
    Fix $m \in \omega$ and $g\in G_m$.
    As $(\sqsupset^m_n)$ is co-bijective, $g^\in$ is non-empty, and so contains the endpoint $(e_{\bar{S}, n})^\leq$ for some nondegenerate branch of spokes $\bar{S}$ where $S_n=[0_n,s_n]$.  Hence, $g \sqsupset^m_n e_{\bar{S}, n}$ for some $n \geq m$.

    Since $(e_{\bar{S}, n})$ is a thread of ends of the co-bijective modification of the surjective core $([0_n,t_n])$ of $\bar{S}$, by condition \ref{lemmacobijarc2} of \autoref{PathCoBijectiveModification} there is $k\geq n$ such that $t_k\sqsubset_k^n e_{\bar{S}, n}$. Now, fix $k'\geq k$ such that $S_{k'}{}^{\sqsubset_{k'}^k}=[0_k,t_k]$. If $t_{k'}{}^{\sqsubset_{k'}^k}=\{t_k\}$, then every $h\in [t_k',s_{k'}]$ must satisfy the same, otherwise $t_{k'}$ would separate $[0_k,t_k)^{\sqsupset_{k'}^k}$. In particular if this is the case $s_{k'}\sqsubset_{k'}^k t_k\sqsubset_k^m g$. Otherwise, ${\sqsupset_k^k}\restriction [0_k,t_k]\times [0_{k'},t_{k'}]$ is not co-bijective, hence $t_{k}\sim e_{\bar S,k}$, and therefore $s_{k'}\sqsubset_{k'}^k e_{\bar S,k}$. This is the thesis.
 \end{proof}

\begin{rmk}
Having seen the Cantor fan as the Fra\"iss\'e limit of $\mathbf X$, we have a second way to see the Lelek fan as a limit of a sequence in $\mathbf L$. Since this is not needed, thanks to \autoref{SmoothFansRealizable}, we just roughly sketch this argument below.

We pick a sequence in $\mathbf X$ defined recursively as follows: $G_0$ is the claw. If $G_n$ has been defined, and it has spokes $S_1,\ldots,S_k$, we let $G_{n+1}$ be a fan with $\sum_{i\leq k}|S_i|$ many spokes, where the $(\sum_{i<m}|S_i|+j)$-th spoke, for $m\leq k$ and $j\leq |S_m|$, is $\mathsf XS_i$ (for $m=1$ we set $\sum_{i<m}|S_i|=0$). To define the morphism $\sqsupset_{n+1}^n$, we glue together the morphisms $\in_{S_i}$ between the $(\sum_{i<m}|S_i|+j)$-th spoke, for $m\leq k$ and $j\leq |S_m|$, and $S_i$. This sequence is lax-Fra\"iss\'e in $\mathbf X$, and therefore has the Cantor fan as its spectrum.

Our intention is to define subsets $H_n\subseteq G_n$ giving a prime restriction with the Lelek fan as its spectrum. Let $H_0=G_0$. If $H_n$ has been defined, re-enumerate the spokes of $H_n$ so that $T_i\subseteq S_i$ be the $i$-th spoke of $H_n$. For each $i$, list $S_i=\{x_1,\ldots,x_{|S_i|}\}$ in such a way that elements with consecutive indexes are adjacent and $x_1=0_n$. Let $m_i$ be the natural such that $T_i=[x_1,x_{m_i}]$. For a fixed $i$, list all the spokes which are $\sqsupset_{n+1}^n$ predecessors of $S_i$ in $G_{n+1}$ as $U_{1,i},\ldots,U_{|S_i|,i}$. For $j\leq m_i$, let $W_{j,i}=\{x\in U_j: x^{\sqsubset_{n+1}^n}\subseteq [x_1,x_j]\}$, and if $j>m_i$, let $W_{j,i}=\{x\in U_j: x^{\sqsubset_{n+1}^n}\subseteq [x_1,x_{m_i}]\}$. Setting $H_{n+1}=\bigcup_{i,j}W_{i,j}$ and ${\sqni_{n+1}^n}={\sqsupset_{n+1}^n}\restriction H_n\times H_{n+1}$ gives the required relations, in light of \autoref{EndDenseSubsequence}.
\end{rmk}

Since star-refining/end-dense morphisms form wide ideals in $\mathbf L$, lax-Fra\"iss\'e sequences must have star-refining and end-dense subsequence. The condition on the sequence of ends being Fra\"iss\'e, needed in the case of $\mathbf X$, is in this case replaced by asking for an end-splitting subsequence.

\begin{thm}\label{ThmLelekFraisse}
Let $(G_n,\sqsupset_n^m)$ be a sequence in $\mathbf L$. The following are equivalent.
\begin{enumerate}
 \item\label{LF1} $(\sqsupset_n^m)$ is a lax-Fra\"iss\'e sequence in $\mathbf L$;
\item\label{LF2} $(\sqsupset_n^m)$ has a star-refining end-dense end-splitting subsequence;
\item\label{LF4} $\Spec{\PP}$ is homeomorphic to the Lelek fan.
\end{enumerate}
 
The sequence $(\sqsupset_n^m)$ is Fra\"iss\'e precisely when the surjective core of every branch of spokes has a  strictly anti-injective subsequence.
 \end{thm}

\begin{proof}
\ref{LF1}$\Rightarrow$\ref{LF2}: This follows from that star-refining end-dense end-splitting morphisms form a wide ideal in $\mathbf L$ and \autoref{IdealFraisseSubsequence}.

\ref{LF2}$\Rightarrow$\ref{LF1}: Fix $m$, a fan $F$, and assume that ${\sqsupset}\in\mathbf L_F^{G_m}$. Let $M=\max_{S\in *F}|S|$, and let $n$ such that 
\begin{itemize}
\item $\sqsupset_n^m$ is a composition of $M$ star-refining morphisms, and 
\item every vertex of $G_m$ has below at least $|F|$ elements of $\mathsf EG_n$.
\end{itemize}
For the first condition, we use that $\sqsupset_n^m$ has a star-refining subsequence.
For the second one, first find $m'$ such that $\sqsupset_{m'}^m$ is end-dense and then find $n$ large enough so that $|g^{\sqsupset_n^{m'}}\cap \mathsf EG_n|\geq |F|$ for all $g\in\mathsf EG_{m'}$.

Now, fix a vertex $g\in G_m$. By construction, we have that
\[
|\{S\in\ast F: S^{\sqsubset}=[0_n,g]\}|\leq |\{S\in\ast G_n: S^{\sqsubset_n^m}=[0_n,g]|,
\]
therefore there is a surjective function $f\colon {\ast G_n}\to\ast F$ with the property that $S^{\sqsubset_n^m}=f(S)^{\sqsubset}$. For every $S\in \ast G_n$, consider 
\[
{\sqsupset_n^m}\restriction S^{\sqsubset_n^m}\times S \quad\text{ and }\quad {\sqsupset}\restriction f(S)^{\sqsubset}\times S.
\]
Since $\sqsupset_n^m$ is a composition of $M$ star-refining morphisms and by the choice of $M$, we can find ${\sqni}_{S}\in \mathbf A_{S}^{f(S)}$ such that 
\[
{{({\sqsupset}\restriction f(S)^{\sqsubset}\times f(S))}\ \circ\sqni_{S} }\subseteq {\sqsupset_n^m}\restriction S^{\sqsubset_n^m}\times S.
\]
Gluing all the $\sqni_{S}$ at the root of $G_n$ gives a morphism (in $\mathbf X$) from $G_n$ to $F$, which amalgamates, proving that $(\sqsupset_n^m)$ is lax-Fra\"iss\'e.

\ref{LF4}$\Rightarrow$\ref{LF2}. By \autoref{EndDenseSubsequence}, since endpoints are dense in the Lelek fan, $(\sqsupset_n^m)$ has an end-dense subsequence. Further, since in every non-empty open set there are at least two endpoints of the Lelek fan, then every element of $\PP$ has at least two ends below. Lastly, since the Lelek is Hausdorff, $(\sqsupset_n^m)$ has a star-refining subsequence. 

\ref{LF2}$\Rightarrow$\ref{LF4}: Since the Lelek fan is a smooth fan, \autoref{SmoothFansRealizable} gives a given sequence in $\mathbf L$ having the Lelek fan as spectrum. Since \ref{LF4}$\Rightarrow$\ref{LF2}, this sequence is lax-Fra\"iss\'e. Since all lax-Fra\"iss\'e sequence give homeomorphic spectra, any given lax-Fra\"iss\'e sequence must have the Lelek fan as a spectrum.

The proof of the last statement (that Fra\"iss\'e sequence correspond to those which have strictly anti-injective subsequences) is the same as in the case of $\mathbf A$. This concludes the proof.
\end{proof}
\begin{rmk}
    We should note that our proof gives in fact as a byproduct the uniqueness of the Lelek fan. If we start with two distinct Lelek fans $L_1$ and $L_2$ (i.e. smooth fans whose endpoints are dense), by \autoref{SmoothFansRealizable} we obtain two sequences in $\mathbf L$ having $L_1$ and $L_2$ as spectra, respectively. By \autoref{ThmLelekFraisse}, both these sequences are lax-Fra\"iss\'e, and therefore these have homeomorphic spectra. We conclude that Lelek fans are unique.
\end{rmk}

\color{black}

\section{Concluding remarks}\label{s.Concluding}
\subsection{Other easy examples}
There are certainly other categories with amalgamation we might have considered, which give rise to other interesting continua. In these cases, one can characterise lax-Fra\"iss\'e sequence in a similar fashion to what we did in the case of the arc and for fans. Since the proofs are similar, we only sketch the non obvious part of the arguments, leaving to the reader to sort out the details from our previous work.

For example, recall that a graph is a cycle if every vertex has degree exactly $2$. Let $\mathbf C$ be the category whose objects are cycles and whose morphisms are monotone relations in $\mathbf B$ (note that, once the cycles are long enough, co-injectivity follows directly from monotonicity, by the same reason that in the path of setting co-injectivity only fails at ends, see \autoref{OnlyEndpointRedundantLemma}).
A cycle with three vertices is a weakly terminal object.
One can show that this category has amalgamation, and therefore Fra\"iss\'e sequences.
\begin{prp}
Let $(G_n,\sqsupset_n^m)$ be a sequence in $\mathbf C$ with induced poset $\PP$.
The following are equivalent.
\begin{enumerate}
\item\label{C1} $(\sqsupset_n^m)$ is a lax-Fra\"iss\'e sequence in $\mathbf C$;
\item\label{C2} $(\sqsupset^m_n)$ has an edge-witnessing subsequence;
 \item\label{C3} $(\sqsupset^m_n)$ has a star-refining subsequence;
 \item $\mathsf{S}\mathbb{P}$ is homeomorphic to the circle.
\end{enumerate}
A lax-Fra\"iss\'e sequence is Fra\"iss\'e precisely when it also has a strictly anti-injective subsequence.
\end{prp}
\begin{proof}[Sketch of the proof]
The equivalences between \ref{C1}, \ref{C2} and \ref{C3} are exactly as in the proof of \autoref{AFraisseEquivalents}, as we are dealing with triangle-free graphs (as long as $|G_n|\geq 4$) and monotone maps.
To show that $\Spec{\PP}$ is the circle, we choose a suitable sequence of minimal covers an apply Proposition~\ref{SequenceForSpace}.
This can be done by arranging open balls of smaller and smaller radius around roots of the unit.
For example, let $C_{n, k}$ be the intersection of the unit circle in $\mathbb{C} = \mathbb{R}^2$ and of the open ball with center $c_{n, k} = e^{(2k\pi i)/2^n}$ and radius such that the boundary contains the neighboring centers $c_{n, k + 1}, c_{n, k - 1}$.

To show that a sequence of cycles without an edge-witnessing subsequence cannot give the circle as the spectrum of its associated poset, note that if $G_m$ is a cycle and there is no $n\geq m$ such that $\sqsupset_n^m$ is edge-witnessing, then there is a pair of adjacent vertices $g,g'\in G_m$ such that $g^{\in}$ and $g'^{\in}$ are disjoint open sets. This gives either that the spectrum is disconnected (in the case $\wedge$ divides $G_m$ in at least two connected components), or a minimal cover of $\mathsf S\PP$ which is a chain.
Since morphisms in $\mathbf{C}$ are edge-surjective, $m$ could be arbitrarily large, and therefore $\mathsf S\PP$ cannot be the circle.
\end{proof}

When we do not consider only monotone maps, and we allow cycles to wrap and unwrap multiple times, we expect again a category with amalgamation (this can easily be checked by hand). In this case, the spectrum of the poset associated to a lax-Fra\"iss\'e sequence should depend on how many times we allow our cycles to wrap. We speculate that $P$-adic solenoids, where $P$ is a set of primes, can be obtained this way.

\subsection{The pseudoarc} \label{ss.pseudoarc}
We are left to comment the elephant in the room from \autoref{ss.arc}, that is, the category $\mathbf P$ whose objects are paths and morphisms are co-bijective edge-preserving relations, i.e., elements of $\mathbf B$. 

Since the clique graph of a path is a path, \autoref{TreesCliqueClosed} gives that $\mathbf P$ is clique-closed. By \autoref{cor:cliqueclosed}, to check amalgamation for $\mathbf P$ it thus suffices to check amalgamation for surjective edge-preserving functions between paths. This was done in \cite{IrwinSolecki2006}, hence $\mathbf P$ is a Fr\"aiss\'e category. What about its Fr\"aiss\'e limit? 

Intuition (and \cite{IrwinSolecki2006}, and subsequent work) suggests that the answer here is probably one of the most interesting one-dimensional continua, Bing's pseudoarc. This is a planar continuum which can be obtained as the intersection of a sequence of chains that get nested into each other in a crooked way. The pseudoarc was introduced by Bing in \cite{Bing1948}, and since then it has received lot of attention, being in some sense very close to an arc yet very intricate. As mentioned, the pseudoarc was treated as a projective Fra\"iss\'e limit by Irwin and Solecki in \cite{IrwinSolecki2006} (indeed, this was the first example of a topological space viewed from a Fra\"iss\'e-theoretic point of view), and by Bartoš and Kubiś in \cite{BartosKubis2022} in the approximate setting. 

This intuition was formalised by the second author and Malicki, who in \cite[\S4]{BiceMalicki} deeply analysed the notions of tangled morphisms and posets from \cite[\S2]{BartosBiceV.Compacta} to show that Bing's pseudoarc is indeed the \Fraisse limit of the category $\mathbf P$. This approach was subsequently used to reprove the existence and the uniqueness of the pseudoarc and its homogeneity, and moreover to show that the autohomeomorphisms of the pseudoarc have a dense conjugacy class.
In particular, Theorem~4.12 in \cite{BiceMalicki}  gives equivalent conditions (related respectively to the poset and to a weaker Fra\"iss\'e condition on the sequence of interest called sub-Fra\"iss\'e) to ensure that the spectrum of a sequence in $\mathbf P$ is hereditarily indecomposable. We do not know whether the exact analog of \autoref{AFraisseEquivalents} holds in this setting, and specifically whether all sequences having the pseudoarc as a spectrum are lax-Fra\"iss\'e, i.e. whether all sub-Fra\"iss\'e sequences in $\mathbf{P}$ are lax-Fra\"iss\'e.

\bibliography{Maths2}

\providecommand{\bysame}{\leavevmode\hbox to3em{\hrulefill}\thinspace}
\providecommand{\MR}{\relax\ifhmode\unskip\space\fi MR }
\providecommand{\MRhref}[2]{%
  \href{http://www.ams.org/mathscinet-getitem?mr=#1}{#2}
}
\providecommand{\href}[2]{#2}
\begin{thebibliography}{KTW03}

\bibitem[Ale29]{Alexandroff1928}
P.~Alexandroff, \emph{Untersuchungen \"{u}ber {G}estalt und {L}age
  abgeschlossener {M}engen beliebiger {D}imension}, Ann. of Math. (2)
  \textbf{30} (1928/29), no.~1-4, 101--187.

\bibitem[BBV25]{BartosBiceV.Compacta}
A.~Barto{\v{s}}, T.~Bice, and A.~Vignati, \emph{Constructing compacta from
  posets}, Publ. Mat. \textbf{69} (2025), no.~1, 217--265.

\bibitem[BC21]{BassoCamerlo.Fence}
G.~Basso and R.~Camerlo, \emph{Fences, their endpoints, and projective
  {F}ra\"{\i}ss\'{e} theory}, Trans. Amer. Math. Soc. \textbf{374} (2021),
  no.~6, 4501--4535.

\bibitem[Bin48]{Bing1948}
R.~H. Bing, \emph{A homogeneous indecomposable plane continuum}, Duke Math. J.
  \textbf{15} (1948), 729--742.

\bibitem[BK]{BartosKubis2022}
A.~Barto{\v{s}} and W.~Kubi{\'s}, \emph{Hereditarily indecomposable continua as
  generic mathematical structures}, arXiv:2208.06886.

\bibitem[BK15]{BartosovaKwiatkowska2015}
D.~Barto\v{s}ov\'{a} and A.~Kwiatkowska, \emph{Lelek fan from a projective
  {F}ra\"{i}ss\'{e} limit}, Fund. Math. \textbf{231} (2015), no.~1, 57--79.

\bibitem[BM]{BiceMalicki}
T.~Bice and M.~Malicki, \emph{Homeomorphisms of the pseudoarc},
  arXiv:2412.20401.

\bibitem[BO90]{BulaOver.CantorFans}
W.~Bula and L.~Oversteegen, \emph{A characterization of smooth {C}antor
  bouquets}, Proc. Amer. Math. Soc. \textbf{108} (1990), no.~2, 529--534.

\bibitem[Cha67]{Chara.Fans}
J.~J. Charatonik, \emph{On fans}, Dissertationes Math. (Rozprawy Mat.)
  \textbf{54} (1967), 39.

\bibitem[Cha89]{Chara.Lelek}
W.~Charatonik, \emph{The {L}elek fan is unique}, Houston J. Math. \textbf{15}
  (1989), no.~1, 27--34.

\bibitem[CK24]{codenotti2022projective}
A.~Codenotti and A.~Kwiatkowska, \emph{Projective {F}ra\"iss\'e{} limits and
  generalized {W}a\.zewski dendrites}, Fund. Math. \textbf{265} (2024), no.~1,
  35--73.

\bibitem[CKR25]{charatonik2022projective}
W.~Charatonik, A.~Kwiatkowska, and R.~Roe, \emph{The projective
  {F}ra\"{\i}ss\'{e} limit of the class of all connected finite graphs with
  confluent epimorphisms}, Trans. Amer. Math. Soc. \textbf{378} (2025), no.~2,
  1081--1126.

\bibitem[CKRY]{charatonik2024projectivefraisselimitstrees}
W.~Charatonik, A.~Kwiatkowska, R.~Roe, and S.~Yang, \emph{Projective
  fra\"iss\'e limits of trees with confluent epimorphisms}, arXiv:2312.16915.

\bibitem[DT18]{DebskiTymchatyn2018}
W.~Debski and E.~D. Tymchatyn, \emph{Cell structures and topologically complete
  spaces}, Topology Appl. \textbf{239} (2018), 293--307.

\bibitem[Fla61]{Flachsmeyer1961}
J.~Flachsmeyer, \emph{Zur {S}pektralentwicklung topologischer {R}\"{a}ume},
  Math. Ann. \textbf{144} (1961), 253--274.

\bibitem[IS06]{IrwinSolecki2006}
T.~Irwin and S.~Solecki, \emph{Projective {F}ra\"\i ss\'e limits and the
  pseudo-arc}, Trans. Amer. Math. Soc. \textbf{358} (2006), no.~7, 3077--3096.

\bibitem[KTW03]{KoppermanTkachukWilson2003}
R.~Kopperman, V.~Tkachuk, and R.~Wilson, \emph{The approximation of compacta by
  finite {$T_0$}-spaces}, Quaest. Math. \textbf{26} (2003), no.~3, 355--370.

\bibitem[Kub]{Kubis2012}
W.~Kubi\'s, \emph{Metric-enriched categories and approximate {F}ra\" iss\'e
  limits}, arXiv:1210.6506.

\bibitem[Kub14]{Kubis2014b}
\bysame, \emph{{F}ra\"\i ss\'e sequences: category-theoretic approach to
  universal homogeneous structures}, Ann. Pure Appl. Logic \textbf{165} (2014),
  no.~11, 1755--1811.

\bibitem[Nad92]{Nadler1992}
S.B. Nadler, \emph{Continuum theory}, Monographs and Textbooks in Pure and
  Applied Mathematics, vol. 158, Marcel Dekker, Inc., New York, 1992.

\bibitem[PS22]{PanagiotopoulosSolecki2022}
A.~Panagiotopoulos and S.~Solecki, \emph{A combinatorial model for the {M}enger
  curve}, J. Topol. Anal. \textbf{14} (2022), no.~1, 203--229.

\end{thebibliography}
 \bibliographystyle{amsalpha}
\end{document}